\documentclass[10pt,a4paper]{amsart}
\usepackage{amsmath,amssymb,amsthm}
\usepackage{comment}
\usepackage{mathtools}
\mathtoolsset{showonlyrefs=true}
\usepackage[top=4.3cm,bottom=4.3cm,left=3cm,right=3cm]{geometry}
\newtheorem{thm}{Theorem}[section]
\newtheorem{cor}[thm]{Corollary}
\newtheorem{prop}[thm]{Proposition}
\newtheorem{lem}[thm]{Lemma}

\theoremstyle{definition}
\newtheorem{defn}[thm]{Definition}
\newtheorem*{acknowledgement}{Acknowledgement}
\theoremstyle{remark}
\newtheorem{rem}[thm]{Remark}
\numberwithin{equation}{section}

\allowdisplaybreaks[4]

\newcommand{\de}{\delta}

\newcommand{\e}{\varepsilon}

\renewcommand{\th}{\theta}
\newcommand{\vth}{\vartheta}
\newcommand{\Th}{\Theta}

\newcommand{\p}{\partial}
\newcommand{\I}{\infty}
\newcommand{\Sc}[1]{\mathcal{#1}}
\newcommand{\F}{\Sc{F}}

\newcommand{\FR}[1]{\mathfrak{#1}}
\newcommand{\Bo}[1]{\mathbb{#1}}
\newcommand{\R}{\Bo{R}}

\renewcommand{\Re}{\,\mathrm{Re}}
\renewcommand{\Im}{\,\mathrm{Im}}

\newcommand{\lec}{\lesssim}
\newcommand{\gec}{\gtrsim}
\newcommand{\hhat}{\widehat}
\newcommand{\bbar}{\overline}
\newcommand{\ti}{\widetilde}

\newcommand{\supp}[1]{\operatorname{supp}\> #1}
\newcommand{\Supp}[2]{\supp{#1}\subset #2}
\newcommand{\shugo}[1]{\{ #1\}}
\newcommand{\Shugo}[2]{\big\{ \, #1 \, \big| \, #2 \, \big\}}

\newcommand{\LR}[1]{{\langle #1 \rangle }}
\newcommand{\chf}[1]{\textbf{1}_{#1}}

\newcommand{\norm}[2]{\big\| #1 \big\| _{#2}}
\newcommand{\tnorm}[2]{\| #1 \| _{#2}}
\newcommand{\dint}{\displaystyle\int}

\newcommand{\eq}[2]{\begin{equation} \begin{split} #2 \end{split}\label{#1} \end{equation}}
\newcommand{\eqq}[1]{\begin{equation*} \begin{split} #1 \end{split} \end{equation*}}
\newcommand{\eqs}[1]{\begin{gather*} #1 \end{gather*}}
\newcommand{\mat}[1]{\begin{smallmatrix} #1 \end{smallmatrix}}

\newcommand{\hx}{\hspace{10pt}}
\newcommand{\hxx}{\hspace{30pt}}

\begin{document}

\title[Scattering  for mass critical  NLS system]{Scattering for a mass critical NLS system below the ground state with and without mass-resonance condition}

\author[T.Inui]{Takahisa Inui}
\address{Department of Mathematics, Graduate School of Science, Osaka University, Toyonaka, Osaka, 560-0043, Japan}
\email{inui@math.sci.osaka-u.ac.jp}

\author[N.Kishimoto]{Nobu Kishimoto}
\address{Research Institute for Mathematical Sciences, Kyoto University, Kyoto, 606-8502, Japan}
\email{nobu@kurims.kyoto-u.ac.jp}

\author[K.Nishimura]{Kuranosuke Nishimura}
\address{Department of Mathematics, Graduate School of Science, Tokyo University of Science, 1-3 Kagurazaka, Shinjuku-ku, Tokyo 162-8601, Japan}
\email{1117614@ed.tus.ac.jp}

\maketitle

\begin{abstract}
We consider a mass-critical system of nonlinear Sch\"{o}dinger equations
\begin{align*}
\begin{cases}
i\partial_t u +\Delta u =\bar{u}v,\\
i\partial_t v +\kappa \Delta v =u^2,
\end{cases}
(t,x)\in \mathbb{R}\times \mathbb{R}^4,
\end{align*}
where $(u,v)$ is a $\mathbb{C}^2$-valued unknown function and $\kappa >0$ is a constant. If $\kappa =1/2$, we say the equation satisfies  mass-resonance condition.
We are interested in the scattering problem of this equation under the condition $M(u,v)<M(\phi ,\psi)$, where $M(u,v)$ denotes the mass and $(\phi ,\psi)$ is a ground state. 
In the mass-resonance case, we prove scattering by the argument of Dodson \cite{MR3406535}. Scattering is also obtained without mass-resonance condition under the restriction  that $(u,v)$ is radially symmetric. 
\end{abstract}

\tableofcontents


\section{Introduction}
\subsection{Introduction}
We consider 
\begin{align}\label{nls1}
&\left\{
\begin{array}{l}
i\p _tu+\frac{1}{2m} \Delta u=\lambda \bbar{u}v,\\[5pt]
i\p _tv+ \frac{1}{2M} \Delta v=\mu u^2,
\end{array}
\quad (t,x)\in \R \times \R ^d,
 \right.
 \\[2pt]
 &\quad\, (u,v)\big| _{t=0}=(u_0,v_0)\in L^2(\R ^d)^2,
\end{align}
where $(u,v)$ is a $\mathbb{C}^2$-valued unknown function, $m$ and  $M$ are  positive constants, and $\lambda, \mu \in \mathbb{C}\setminus \{0\}$ are constants.
From the viewpoint of physics, \eqref{nls1} is related to the Raman amplification in a plasma. See \cite{MR2569892} for details. Furthermore \eqref{nls1} is 
regarded as a non-relativistic limit of the system of nonliear Klein--Gordon equations under the mass-resonance condition $M=2m$ (see \cite{MR3082479}). 
Under the assumption $\lambda = c\bar{\mu}$ for some $c>0$, the solution to  $\eqref{nls1}$ conserves the mass and the energy, defined respectively by
\begin{align}
\text{Mass}&=\|u(t)\|_{L^2}^2 +c\|v(t)\|_{L^2}^2,\\
\text{Energy}&=\frac{1}{2m}\|\nabla u(t)\|_{L^2}^2 + \frac{c}{4M}\|\nabla v (t)\|_{L^2}^2 +\Re\left( \lambda \int _{\R^d}u^2(t,x)\bbar{v(t,x)}\,dx \right).
\end{align}
To use these conservation laws, we impose $\lambda =c\bbar{\mu} $ for some $c>0$. Then the system \eqref{nls1} is reduced to the following system by some changes of variables:
\begin{align}\label{dNLS}
&\left\{
\begin{array}{l}
i\p _tu+ \Delta u= \bbar{u}v,\\
i\p _tv+\kappa \Delta v= u^2,
\end{array}
\quad (t,x)\in \R \times \R ^d,
 \right.
 \\
 &\quad\, (u,v)\big| _{t=0}=(u_0,v_0)\in L^2(\R ^d)^2,
\end{align}
where $\kappa $ is a positive constant.
If $(u,v)$ is a solution to \eqref{dNLS}, then $(u^{\lambda }, v^{\lambda}) := \lambda^{-2}(u,v)(\lambda^{-2} t , \lambda^{-1} x)$  also solves \eqref{dNLS} for any $\lambda >0$.
This property is called scaling symmetry.
Setting $s_c := d/2 - 2$, then $\|(u,v)\|_{\dot{H}_x^{s_c}}$ is invariant under this scaling transformation.
In particular, if $d=4$, the mass is invariant under the scaling transformation. From this fact, if $d=4$, we say the system \eqref{dNLS} is mass-critical. Similarly, we call the case $d=5$ and $d=6$ $\dot{H}^{\frac{1}{2}}$-critical and energy-critical, respectively.
In this paper, we treat the following mass-critical system of NLS:
\begin{align}\label{NLS}
&\left\{
\begin{array}{l}
i\p _tu+ \Delta u= \bbar{u}v,\\
i\p _tv+\kappa \Delta v= u^2,
\end{array}
\quad (t,x)\in \R \times \R ^4,
 \right.
 \\
 &\quad\, (u,v)\big| _{t=0}=(u_0,v_0)\in L^2(\R ^4)^2,
\end{align}
where $\kappa$ is a positive constant. If $\kappa =1/2$, \eqref{NLS} has Galilean invariance, i.e., 
if $(u,v)$ is a solution to \eqref{NLS}, then $(e^{ix\cdot \xi}e^{-it|\xi|^2} u(t,x-2t\xi), e^{2ix\cdot \xi} e^{-2it|\xi|^2}v(t,x-2t\xi))$ is also a solution to \eqref{NLS} for any $\xi \in \mathbb{R}^4$.
Note that in our case, the mass and the energy are respectively defined by
\begin{align}
&M(u,v)(t):=\|u(t)\|_{L^2}^2 +\|v(t)\|_{L^2}^2,\\
&E(u,v)(t):=\|\nabla u(t)\|_{L^2}^2 +\frac{\kappa}{2}\|\nabla v(t)\|_{L^2}^2 +\operatorname{Re}\int_{\mathbb{R}^4} u^2 (t,x) \overline{v(t,x)}\, dx.
\end{align}
It is well known that local well-posedness and small-data scattering hold for the equation \eqref{NLS} in $L^2(\mathbb{R}^4)^2$. In this paper, we are interested in the behavior of the large solutions. 
\subsection{Known results for the single mass-critical NLS}
In this subsection, we introduce known results about the following mass-critical single  NLS: 
\begin{align}\label{single}
&i\p _tu+\Delta u=\mu |u|^\frac{4}{d} u, \quad (t,x)\in \mathbb{R}\times \mathbb{R}^d,\\
&u(0)=u_0\in L^2(\R ^d),
\end{align}
where $u$ is a $\mathbb{C}$-valued unknown function and $\mu \in \{-1,1\}$.
If $\mu =+1$, we say \eqref{single} is defocusing and if not, \eqref{single} is called focusing. 
In the defocusing case, any $L^2$-solution $u$ of \eqref{single} exists globally and scatters to a free solution, which means that
\begin{equation}
e^{-it\Delta} u(t) \rightarrow u_{\pm} \qquad \text{in} \quad L^2 (\mathbb{R}^d) \qquad \text{as} \quad t\rightarrow \pm \infty \qquad \text{for some} \quad u_{\pm}\in L^2 (\mathbb{R}^d).
\end{equation}
This fact was proved by Dodson in \cite{MR2869023}, \cite{MR3483476}, and \cite{MR3577369} in the case $d\geq 3$, $d=2$, and $d=1$, respectively. 
On the other hand, in the focusing case, the existence of non-scattering solution is known. Indeed, the existence and uniqueness of the ground state $Q$, which is the positive radial solution to 
\begin{align}
Q-\Delta Q =Q^{1+\frac{4}{d}},
\end{align}
is known (see \cite{MR695535}, \cite{MR969899}). 
In \cite{MR3406535}, Dodson proved 
that a solution to \eqref{single} with $\mu =-1$ exists globally and scatters to a free solution under the condition that $\|u(0)\|_{L_x^2} <\|Q\|_{L_x^2} $.
Concerning the system \eqref{NLS}, we expect a similar scattering result to the focusing single NLS \eqref{single} with $\mu=-1$, at least, when $\kappa=1/2$, where the system has the Galilean invariance. In this paper, we also treat the case of $\kappa \neq 1/2$.

\subsection{Main results}
Before stating our main results, we introduce some definitions.

 \begin{defn}
(i) Let $0\in I\subset \R$ be an interval.
We say that $(u,v):I\times \R^4\to \Bo{C}^2$ is a (strong) solution to \eqref{NLS} if it lies in the class $(C(I;L^2(\R^4))\cap L^3_{\text{loc}}(I;L^3(\R^4)))^2$ and obeys the Duhamel formula
\eqq{u(t)&=e^{it\Delta}u_0-i\int _0^te^{i(t-t')\Delta}\big[ \bbar{u(t')}v(t')\big] \,dt',\\
v(t)&=e^{it\kappa \Delta}v_0-i\int _0^te^{i(t-t')\kappa \Delta}\big[ u^2(t')\big] \,dt'}
for all $t\in I$.

(ii) We say that \eqref{NLS} admits global spacetime bounds for a set $D\subset L^2(\R^4)^2$ of initial data if there exists a function $C:[0,\I )\to [0,\I )$ such that for any $(u_0,v_0)\in D$ a solution to \eqref{NLS} exists on $I=\R$ and obeys 
\eqq{\int _\R \int _{\R^4}\big( |u(t,x)|^3+|v(t,x)|^3\big) \,dxdt\le C(M(u_0,v_0)).}

(iii) 
We say that the global solution $(u,v)$ to \eqref{NLS} scatters forward (backward) in time 
if there exist $(u_{\pm},v_{\pm})\in L^2(\R^4)^2$ such that
\eq{scattering}{(u(t),v(t))-(e^{it\Delta}u_\pm ,e^{it\kappa \Delta}v_\pm )\to 0\qquad \text{in}\quad L^2(\R^4)^2\qquad \text{as}\quad t\to \pm \I.}

\end{defn}

 Next, we recall the result of Hayashi, Ozawa, and Tanaka~\cite{MR3082479}. Consider the solution to \eqref{NLS} which has the following form:
\eq{sw}{(u,v)=\Big( e^{it}\phi (x),\, e^{2it}\psi (x)\Big)}
with real-valued functions $\phi ,\psi$. 

If \eqref{sw} is a solution of \eqref{NLS}, then $(\phi ,\psi )$ should satisfy the following system of elliptic equations:
\begin{equation}\label{SE}
\left\{
\begin{array}{@{\,}r@{\;}l}
-\phi +\Delta \phi &=\phi\psi ,\\
-2\psi +\kappa \Delta \psi &=\phi ^2,\qquad x\in \R ^4.
\end{array}
\right.
\end{equation}
Associated functional for $\eqref{SE}$ is defined by
\begin{align}
I(\phi,\psi):=\|\nabla \phi\|_{L_x^2}^2 +\frac{\kappa}{2}\|\nabla \psi \|_{L_x^2}^2 +
\|\phi\|_{L_x^2}^2 +\|\psi\|_{L_x^2}^2 + \int_{\mathbb{R}^4} \phi^2 \psi \, dx. 
\end{align}
Using this functional we give the definition of  ground states.
\begin{defn}
A pair of real-valued functions $(\phi_0 ,\psi_0)\in H^1 \times H^1$ is called a ground state for \eqref{SE} if
\begin{align}
&I(\phi_0 , \psi_0)=\inf \{I(\phi ,\psi)\mid (\phi ,\psi )\in \mathcal{C}\},\\
&\mathcal{C} =\{(\phi ,\psi)\in H^1 \times H^1 \mid (\phi ,\psi) \text{ is a nontrivial critical point of $I$ }\}.
\end{align}
\end{defn}

\begin{lem}[Existence of a ground state, sharp Gagliardo-Nirenberg inequality \cite{MR3082479}]\label{lem:GN}
Let $\kappa >0$.

(i) There exists at least one ground state of \eqref{SE}.

(ii) Let $(\phi ,\psi )$ be a ground state of \eqref{SE}.
Then, it holds that for any $(u,v)\in H^1(\R ^4)^2$ 
\eqq{\Big| \Re \int _{\R^4}u^2(x)\bbar{v(x)}\,dx\Big| \le \Big( \frac{M(u,v)}{M(\phi ,\psi )}\Big) ^{1/2}\big( \norm{\nabla u}{L^2}^2+\frac{\kappa}{2}\norm{\nabla v}{L^2}^2\big) .}
Moreover, equality is attained by the ground state.
\end{lem}

Note that $M(\phi ,\psi )$ does not depend on the choice of a ground state $(\phi ,\psi )$.
From Lemma~\ref{lem:GN}, we see that the energy is positive and controls the $\dot{H}^1$ norm of $H^1$ solutions to \eqref{NLS} for (nonzero) initial data $(u_0,v_0)$ satisfying
\eq{threshold}{M(u_0,v_0)<M(\phi ,\psi ).}
In particular, the initial value problem \eqref{NLS} is globally well-posed in $(H^1)^2$ for initial data satisfying \eqref{threshold}.

However, it is not clear whether the global well-posedness and scattering in $(L^2)^2$ for initial data satisfying \eqref{threshold} hold or not. We first give the following answer when $\kappa=1/2$. 
\begin{thm}\label{thm:main}
If $\kappa =1/2$, then \eqref{NLS} is globally well-posed and scattering holds in $L^2(\R ^4)^2$ for initial data obeying \eqref{threshold}.
\end{thm}
The Galilean invariance plays an important role in the proof of Theorem \ref{thm:main}. On the other hand, the system \eqref{NLS} does not have the Galilean invariance when $\kappa\neq1/2$. Nevertheless, we get the following theorem for $\kappa\neq 1/2$ under the additional assumption of radial symmetry. 
\begin{thm}\label{thm:main2}
Let $\kappa \neq 1/2$. If initial data is radially symmetric and satisfies \eqref{threshold}, then \eqref{NLS} is globally well-posed and scattering holds in $L^2(\R ^4)^2$.
\end{thm}

\subsection{Idea of proof}

It is known that the global existence and scattering are equivalent to that $S_{I_{\max}}(u,v)$ is finite, where $I_{\max}$ is the maximal lifespan of the solution and
\eqq{S_{I}(u,v):=\int _{I} \int _{\R^4}\big( |u(t,x)|^3+|v(t,x)|^3\big) \,dxdt.}
To prove Theorems \ref{thm:main} and \ref{thm:main2}, it is enough to show
 $S_{I}(u,v)< C(M(u,v))$ for any solution $(u,v)$ on a time interval $I$ satisfying \eqref{threshold}. 
With this in mind, let
\eqq{L(M):=\sup \Shugo{S_I(u,v)}{(u,v):I\times \R^4 \to \Bo{C}^2~~\text{solution to \eqref{NLS} such that}~~M(u,v)\le M},}
Note that $L:[0,\I) \to [0,\I ]$ is nondecreasing and, from the stability result (see Section 2), continuous.
Since the standing wave solution \eqref{sw} with $(\phi,\psi )$ a ground state of \eqref{SE} given in Lemma~\ref{lem:GN} is defined globally and $S_\R =\I$, we see that $L(M(\phi ,\psi ))=\I$.
By the continuity of $L$, there exists critical mass $M_c$ such that 
\eqq{L(M)<\I \text{~~for~~} M<M_c,\qquad L(M)=\I \text{~~for~~} M\ge M_c.}
The small data theory in Section 2 ensures that $M_c>0$. Note that $L(M(\phi ,\psi ))=\I$ implies that $M_c \leq M(\phi ,\psi)$.
It then suffices to show that $M_c <M(\phi,\psi)$ would lead to a contradiction.
Furthermore we define the critical mass for the radially symmetric case. Set $L_{\text{rad}}\colon [0,\infty) \rightarrow [0,\infty]$ by
\begin{equation}
L_{\text{rad}} (M) :=\sup \Shugo{S_I(u,v)}{(u,v):I\times \R^4 \to \Bo{C}^2~~\text{radial solution to \eqref{NLS} such that}~~M(u,v)\le M}.
\end{equation}
Then by the stability result in Section 2, $L_{\text{rad}}$ is continuous and so there exists radially symmetric critical mass $M_{c,\text{rad}}$ such that 
\eqq{L_{\text{rad}}(M)<\I \text{~~for~~} M<M_{c,\text{rad}},\qquad L_{\text{rad}}(M)=\I \text{~~for~~} M\ge M_{c,\text{rad}}.}
Since the ground state is radially symmetric, we obtain $M_{c,\text{rad}} \leq M (\phi,\psi)$. Note also that by the definition of $M_c$ it holds $M_{c}\leq M_{c,\text{rad}}$.

Our aim is to derive a contradiction by supposing $M_{c} <M(\phi ,\psi)$ with the mass-resonance condition and $M_{c,\text{rad}} <M(\phi ,\psi)$ without it.

Toward contradiction, we construct the minimal blow-up solution which has critical mass $M_c$ (in the radial case critical mass is $M_{c,\text{rad}}$) by the profile decomposition for the system. 
When we construct the minimal blow-up solution in the case $\kappa \neq 1/2$, we need the radial assumption due to the lack of Galilean invariance. After that we refine the minimal blow-up solution to apply the argument of Dodson \cite{MR3406535}. 
We exclude two possible scenarios which are called rapid frequency cascade and quasi-soliton. 
We eliminate the rapid frequency cascade scenario by additional regularity of the minimal blow-up solution, which comes from the long time Strichartz estimate. 
To exclude the quasi-soliton scenario, we rely on the estimate based on the virial identity (cf. \cite{MR3082479,MR2993843}) which is called the frequency localized interaction Morawetz estimate in \cite{MR3406535}.

 The mass-critical case $d=4$ is quite different from the $\dot{H}^{\frac{1}{2}}$-critical case $d=5$.  
 Hamano \cite{hamano} gave the threshold for scattering or blow-up below the ground state in $\dot{H}^{\frac{1}{2}}$-critical case $d=5$ and $H^1$ setting under the mass-resonance condition. To prove scattering, Hamano used the argument of Kenig--Merle \cite{MR2257393} which is organized by stability, profile decomposition, construction of critical element, and rigidity of it.  
There are two differences between his argument and ours. One is regularity of initial data. More precisely, Hamano assumed $H^1$ regularity to solve $\dot{H}^{\frac{1}{2}}$-critical problem, while we only assume the minimal regularity $L^2$. 
The other is a variety of parameters in the profile decomposition from the lack of compactness in $L^2$. Indeed, the translation in the  frequency side and the scaling transformation additionally breaks the compactness in $L^2$.
The radial assumption is used to remove the former in the case of $\kappa\neq 1/2$.

{\bf Organization of this paper.}
In Section 2, we prepare the stability result and  bilinear Strichartz estimate which are used in Section 3 and Section 4, respectively.
In Section 3, we introduce the profile decomposition for the system and use it to construct the minimal blow-up solution. 
In Section 4, we prove the long time Strichartz estimate. 
In Sections 5 and 6, we treat the rapid frequency cascade scenario and the quasi-soliton scenario, respectively. 
\section{Preliminaries}

First, we collect some notations. Let $\| (u,v) \|_{X}:=\| (u,v) \|_{X \times X}$ for any function space $X$ and $(u,v) \in X \times X=:X^2$. 
We denote the (spatial) Fourier transform of a function $f$ by $\hat{f}$ or $\F f$.

Let $\theta : [0,\infty )\to [0,1]$ be a smooth non-increasing function such that $\theta \equiv 1$ on $[0,1]$ and $\Supp{\theta}{[0,2]}$. 
For each number $N>0$, we define the Littlewood-Paley projection operators by
\eqs{P_{\le N}:=\F^{-1}\theta \Big( \frac{|\cdot |}{N}\Big) \F ,\qquad P_{N}:=P_{\le N}-P_{\le N/2},\qquad P_{>N}:=\mathrm{Id}-P_{\le N},\\
P_{<N}:=P_{\le N}-P_N=P_{\le N/2},\qquad P_{\ge N}:=\mathrm{Id}-P_{<N}=P_{>N}+P_N=P_{>N/2}.}
These operators are bounded uniformly in $N$ on $L^p(\R ^d)$ for any $1\le p\le \infty$, and they commute with each other, as well as with differential operators and Fourier multipliers such as $i\p _t+\Delta$ and $e^{it\Delta}$.
It is easy to see that $\Supp{\hhat{P_{\le N}f}}{\shugo{|\xi |\le 2N}}$, $\Supp{\hhat{P_Nf}}{\shugo{N/2\le |\xi |\le 2N}}$, $\Supp{\hhat{P_{>N}f}}{\shugo{|\xi |\ge N}}$ for any function $f$, and in particular, that $P_{>N}\big( P_{\le N/4}f\cdot P_{\le {N/4}}g\big) =0$ for any $f,g$.
We also define the Fourier projections with the frequency center $\xi _0\in \R^d$ as
\eqq{P_{|\xi -\xi _0|\le N}:=\F ^{-1}\theta \Big( \frac{|\cdot -\xi _0|}{N}\Big) \F ,\qquad P_{|\xi -\xi _0|>N}:=\mathrm{Id}-P_{|\xi -\xi _0|\le N}.}

\subsection{Local well posedness}
We follow the argument of Cazenave and Weissler for \eqref{single} to develop a standard local theory.
\begin{thm}[\cite{MR2856418, MR3082479}]\label{thm:lwp}
Let $(u_0,v_0)\in L^2(\R^4)^2$.
Then there exists a unique maximal-lifespan solution $(u,v):I\times \R^4\to \Bo{C}^2$ to \eqref{NLS}.
This solution also has the following properties:

$\bullet$ (Local existence) $I$ is an open neighborhood of $0$.

$\bullet$ (Mass conservation) $M(u,v)(t)$ is conserved for all $t\in I$.

$\bullet$ (Blow-up criterion) If $\sup I$ is finite, then there exists $t_0\in I$ such that
\eqq{\norm{u}{L^3_{t,x}([t_0,\sup I)\times \R^4 )}+\norm{v}{L^3_{t,x}([t_0,\sup I)\times \R^4 )}=\I .}
(In this case we say that $(u,v)$ blows up forward in time.)
A similar statement holds in the negative time direction.

$\bullet$ (Scattering) If $\sup I=+\I$ and $(u,v)$ does not blow up forward in time, then there exists $(u_+,v_+)\in L^2(\R^4)^2$ such that \eqref{scattering} holds.
Conversely, given $(u_+,v_+)\in L^2(\R^4)^2$ there exists a unique solution to \eqref{NLS} in a neighborhood of $+\I$ so that \eqref{scattering} holds.
A similar statements hold in the negative time direction.

$\bullet$ (Small data global existence) 
There exists $\eta >0$ such that for any $(u_0 ,v_0) \in L^2 (\mathbb{R}^4)^2$ with $M(u_0 ,v_0)\leq \eta $ , the solution to \eqref{NLS} with initial value $(u(0), v(0))=(u_0 , v_0)$ exists globally and satisfies following bound:
\begin{align}\label{bound}
\int_{\mathbb{R}^{1+4}} |u(t,x)|^3 + |v(t,x)|^3\, dxdt \lesssim M(u_0 ,v_0)^{\frac{3}{2}}.
\end{align}
\end{thm}

\begin{rem}
Note that scattering is equivalent to finiteness of $L_{t,x}^3$-norm on the maximal lifespan for any maximal-lifespan solution to \eqref{NLS}.
This is a consequence of Strichartz estimate and standard continuity argument.
\end{rem}

\subsection{Stability result}

In this section, we prepare the stability result.  Note that the stability result follows regardless of the condition for the coefficients $\lambda$ and $\mu$.
\begin{thm}[Mass-critical stability result]\label{theo1}
Let $I$ be an interval and let $(\tilde{u},\tilde{v}) $ be an approximate solution to \eqref{nls1} in the sense that
\begin{align*}
\begin{cases}
i\partial_t \tilde{u}+\frac{1}{2m}\Delta\tilde{u} =\lambda \tilde{v} \bar{\tilde{u}} +e_1,\\
i\partial_t \tilde{v}+\frac{1}{2M}\Delta \tilde{v} =\mu \tilde{u}^2 +e_2,
\end{cases}
\end{align*}
for some functions $(e_1 ,e_2)$. Assume that 
\begin{align}
&\|(\tilde{u} ,\tilde{v})\|_{L_t^{\infty} L_x^{2} (I\times \mathbb{R}^4)} \leq A,\\
&\|(\tilde{u},\tilde{v})\|_{L_{t,x}^3 (I \times \mathbb{R}^4)} \leq L,
\end{align}
for some $A,L>0$.
 Let $t_0 \in I$ and let $(u_0, v_0)$ obey 
 \begin{align}
 \|(u_0 -\tilde{u} (t_0) ,v_0 -\tilde{v}(t_0))\|_{L_x^2 (\mathbb{R}^4)} \leq A'
 \end{align}
 for some $A' >0$.
 Moreover, assume the smallness conditions:
\begin{align}
&\| 
(e^{i\frac{1}{2m}(t-t_0)\Delta}(u_0 - \tilde{u}(t_0) )  , 
e^{i\frac{1}{2M}(t-t_0) \Delta }(v_0-\tilde{v}(t_0) ),
)\|_{L_{t,x}^{3} (I\times \mathbb{R}^4)} \leq \varepsilon ,\\
&\|(e_1 ,e_2)\|_{L_{t,x}^\frac{3}{2} (I\times \mathbb{R}^4)}\leq  \varepsilon , 
\end{align}
for some $0<\varepsilon \leq \varepsilon_1$, where $\varepsilon_1 =\varepsilon_1 (A,A',L)>0$ is a small constant.
Then there exists a solution $(u,v)$ to \eqref{nls1} on $I\times \mathbb{R}^4$ with $(u(t_0) ,v(t_0)) =(u_0 ,v_0)$ satisfying
\begin{align}
&\|(u-\tilde{u} ,v-\tilde{v})\|_{L_{t,x}^3 (I\times \mathbb{R}^4)} \leq  C(A,L)\varepsilon, \\
&\|(u-\tilde{u} , v-\tilde{v})\|_{S^0 (I)^2} \leq C(A,L)A', \\
&\|(u,v)\|_{S^0 (I)^2} \leq C(A,A',L),
\end{align}
where $\|u\|_{S^{0}(I)}:=\displaystyle{\sup_{\textup{$(q,r)$: admissible}} \|u\|_{L_{t}^{q}L_x^{r}(I\times \mathbb{R}^4)}}$.
\end{thm}
\begin{proof}
The proof is very similar to the one of  \cite[Lemma~3.6]{MR2354495}, and so we omit the detail.
\end{proof}

\subsection{Bilinear Strichartz estimate}

Finally, we recall the bilinear Strichartz estimate.
\begin{lem}\label{lem:bs}
Let $M,N>0$, and let $f,g$ be functions on $[a,b)\times \R ^4$ with the support properties
\eqq{\Supp{\hhat{f}(t,\cdot )}{\shugo{|\xi |<M}},\qquad \Supp{\hhat{g}(t,\cdot )}{\shugo{|\xi |>N}}}
for all $t\in [a,b)$.
Let $\th _1,\th _2\in \R \setminus \{ 0\}$.
Then, we have
\eqq{\norm{fg}{L^2([a,b)\times \R^4 )}\lec \frac{M^{3/2}}{N^{1/2}}\norm{f}{S^0_{\th _1}([a,b)\times \R ^4)}\norm{g}{S^0_{\th _2}([a,b)\times \R ^4)},}
where $\tnorm{u}{S^0_\th ([a,b)\times \R ^4)}:=\tnorm{u(a)}{L^2(\R ^4)}+\tnorm{(i\p _t+\th \Delta )u}{L^{3/2}([a,b)\times \R ^4)}$ and the implicit constant depends only on $\th _1,\th _2$.
\end{lem}

\begin{proof}
Let $a=0$ by time translation.
We show only the homogeneous case:
\eq{est:bs-homo}{\norm{e^{it\th _1\Delta}\phi \cdot e^{it\th _2\Delta}\psi}{L^2([0,b)\times \R^4 )}\lec \frac{M^{3/2}}{N^{1/2}}\norm{\phi}{L^2(\R ^4)}\norm{\psi}{L^2(\R ^4)},}
for $\phi ,\psi \in L^2(\R ^4)$ satisfying $\Supp{\hhat{\phi}}{\shugo{|\xi |<M}}$ and $\Supp{\hhat{\psi}}{\shugo{|\xi |>N}}$.
Once \eqref{est:bs-homo} is established, the general case follows by the same argument as \cite[Lemma~2.5]{MR2318286}.

To show \eqref{est:bs-homo}, we basically follow the argument in \cite[Lemma~3.4]{MR2415387}.
Clearly, we may restrict frequencies to single dyadic regions;
\eqq{\Supp{\hhat{\phi}}{\shugo{|\xi |\sim M}},\qquad \Supp{\hhat{\psi}}{\shugo{|\xi |\sim N}}.}
Moreover, we may assume $|\th _1|M\ll |\th _2|N$, since otherwise the claim follows from the $L^4_tL^{8/3}_x$-Strichartz estimate and the Sobolev embedding $\dot{W}^{1,8/3}_x\hookrightarrow L^8_x$.
By a suitable decomposition with respect to the angle and rotation, we may further assume that 
\eqq{\Supp{\hhat{\psi}}{\Shugo{\xi =(\xi ^1,\underline{\xi})\in \R ^{1+3}}{|\xi |\sim N,\,\xi ^1\ge |\xi |/2}}.}

By duality, \eqref{est:bs-homo} is equivalent to
\eq{est:bs-homo'}{\Big| \iint W(-\th _1|\xi _1|^2-\th _2|\xi _2|^2,\xi _1+\xi _2)&\hhat{\phi}(\xi _1)\hhat{\psi}(\xi _2)\,d\xi _1\,d\xi _2\Big| 
\\
&\lec \frac{M^{3/2}}{N^{1/2}}\norm{\phi}{L^2(\R ^4)}\norm{\psi}{L^2(\R ^4)}\norm{W}{L^2(\R ^{1+4})}.}
Changing variables as $(\xi _1^1,\underline{\xi _1},\xi _2)\mapsto (u,\underline{\xi _1},v):=(-\th _1|\xi _1|^2-\th _2|\xi _2|^2,\underline{\xi _1},\xi _1+\xi _2)$ with the assumptions on the Fourier supports of $\phi$ and $\psi$, by which we have $dud\underline{\xi _1}dv=Jd\xi _1d\xi _2$ with $J=2|\th _1\xi _1^1-\th _2\xi _2^1|\sim N$, we encounter
\eqq{\Big| \iiint J^{-\frac{1}{2}}\chf{\{ |\underline{\xi_1}|\lec M\}}(\underline{\xi_1})W(u,v)H(u,\underline{\xi _1},v)J^{-\frac{1}{2}}\,du\,d\underline{\xi _1}\,dv\Big| ,\qquad H(u,\underline{\xi _1},v)=\hhat{\phi}(\xi _1)\hhat{\psi}(\xi _2).}
We apply the Cauchy-Schwarz inequality in $(u,\underline{\xi_1},v)$, which yields 
\eqq{N^{-\frac{1}{2}}\norm{\chf{\{ |\underline{\xi_1}|\lec M\}}}{L^2(\R ^3)}\norm{W}{L^2(\R ^{1+4})}\Big( \iiint |H(u,\underline{\xi _1},v)|^2J^{-1}\,du\,d\underline{\xi _1}\,dv\Big) ^{1/2},}
and then change back to the original variables to obtain \eqref{est:bs-homo'}.
\end{proof}


\section{Minimal mass blow-up solution}

\subsection{Inverse Strichartz inequality, Linear profile decomposition}

In this subsection, we prepare the profile decomposition for the system of NLS.

First we give some notations.
\begin{defn}[Symmetry group]\label{group}
Fix $d\geq 1$ and $\kappa >0$. For any phase $\theta \in \mathbb{R} / 2\pi \mathbb{Z}$, position $x_0 \in \mathbb{R}^d$, frequency $\xi_0 \in \mathbb{R}^d$, and scaling parameter $\lambda \in (0,\infty)$, we define the unitary transformation $g_{\kappa} (\theta ,\xi_0 ,x_0, \lambda_0) \colon L^2 (\mathbb{R}^d)^2 \rightarrow L^2 (\mathbb{R}^d)^2$ and $h(\theta ,\xi_0,x_0,\lambda)\colon L^2_x(\mathbb{R}^d) \rightarrow L^2_x (\mathbb{R}^d)$ by
\begin{align*}
&[h(\theta ,\xi_0,x_0,\lambda) \phi](x):=\lambda^{-\frac{d}{2}} e^{i\theta} e^{ix\cdot \xi_0} \phi(\lambda^{-1} (x-x_0)),\\
&\left[
g_{\kappa} (\theta ,\xi_0,x_0,\lambda) 
(\phi,
\psi)
\right] (x)
:=\left(h(\theta ,\xi_0 ,x_0,\lambda) \phi, h(\frac{\theta}{\kappa} ,\frac{\xi_0}{\kappa} ,x_0 ,\lambda)\psi\right).
\end{align*}
Furthermore we set the group $G_{\kappa}$ by $G_{\kappa}:=\{g_{\kappa}(\theta , \xi_0 , x_0, \lambda_0)\mid (\theta , \xi_0 , x_0, \lambda_0 ) \in \mathbb{R}/2\pi\mathbb{Z} \times \mathbb{R}^d \times \mathbb{R}^d \times \mathbb{R}_{>0} \}$. For a convenience, we also set $G:=G_{\frac{1}{2}}$ and $g(\theta, \xi_0 ,x_0 , \lambda):= g_{\frac{1}{2}}(\theta, \xi_0 ,x_0 , \lambda)$.
\end{defn}

Next, we refer to the following theorem which is called Inverse Strichartz inequality.
\begin{lem}[Inverse Strichartz inequality]\label{inv-str}
Fix $d\geq 1$. Let $\{u_n\}_n \subset L^2 (\mathbb{R}^d)$ be a bounded sequence. We also assume that
\begin{align*}
A&:=\lim_{n\rightarrow \infty} \|u_n\|_{L_x^2} ,\\
\varepsilon &:=\lim_{n\rightarrow \infty} \|e^{it\Delta} u_n\|_{L_{t,x}^{\frac{2(d+2)}{d}} (\mathbb{R}^{1+d}) }  >0 .
\end{align*}
Then,  passing to s subsequence if necessary,  there exist $\{\xi_n\} \subset \mathbb{R}^d , \{y_n\} \subset \mathbb{R}^d , \{\lambda_n\} \subset (0,\infty) , \{s_n\} \subset \mathbb{R}$, and $\phi \in L^2 (\mathbb{R}^d)$ such that
\begin{align*}
&h(0, \xi_n , y_n , \lambda_n  )^{-1} e^{i s_n \Delta}  u_n \rightharpoonup \phi \ \ in\ \ L^2 (\mathbb{R}^d),\\
&\|\phi\|_{L_x^2}^2 \gtrsim A^2 \left(\frac{\varepsilon}{A}\right)^{2(d+1)(d+2)}.
\end{align*} 
\end{lem}

\begin{proof}
See for instance \cite[Proposition 4.25]{MR3098643}.
\end{proof}

Using above Lemma, we give the Inverse Strichartz for the system.
\begin{prop}[Inverse Strichartz inequality for a system]\label{prop} Fix $d\geq 1$, $\kappa >0$, and $\{(u_n ,v_n) \} \subset L^2_x (\mathbb{R}^d)^2 . $ Suppose also that 
\begin{align*}
A&:=\lim_{n\rightarrow \infty} \|(u_n ,v_n)\|_{L^2_x (\mathbb{R}^d)},\\
\varepsilon &:=\lim_{n\rightarrow \infty} \|(e^{it\Delta } u_n ,e^{i\kappa t\Delta} v_n) \|_{L_{t,x}^{\frac{2(d+2)}{d} } (\mathbb{R}^{1+d}) } >0.
\end{align*}
Then, after  passing to a subsequence if necessary, there exist $(\phi ,\psi ) \in L^2_x (\mathbb{R}^d)^2$, $\{\lambda_n\} \subset (0,\infty)$, $\{\xi_n\}
\subset \mathbb{R}^d$, and $(s_n , y_n) \subset \mathbb{R}^{1+d}$ such that 
\begin{align}
&\left[g_{\kappa}(0,\xi_n,y_n,\lambda_n)^{-1}(
e^{is_n \Delta } u_n ,
e^{i\kappa s_n\Delta} v_n )\right] \rightharpoonup (\phi ,\psi)\ weakly \ in\ L_x^2 (\mathbb{R}^d)^2 ,\nonumber \\
&\lim_{n\rightarrow \infty} \left[\|u_n\|_{L^2_x }^2 -\|u_n -\phi_n\|_{L^2_x}^2\right] =\|\phi\|_{L^2_x}^2 ,\label{eq2}\\
&\lim_{n\rightarrow \infty} \left[\|v_n\|_{L^2_x }^2 -\|v_n -\psi_n\|_{L^2_x}^2 \right]=\|\psi\|_{L^2_x}^2 ,\label{eq3}\\
&\|\phi\|_{L_x^2}^2 +\|\psi\|_{L^2_x}^2 \gtrsim A^2 \left(\frac{\varepsilon}{A}\right)^{2(d+1)(d+2)},\label{eq4}
\end{align}
where $\phi_n $ and $\psi_n$ are defined by
\begin{align*}
\phi_n :=e^{-is_n \Delta } h(0,\xi_n,y_n, \lambda_n) \phi ,\quad \psi_n :=e^{-i\kappa s_n \Delta} h(0,\frac{\xi _n }{\kappa },y_n,\lambda_n)\psi.
\end{align*}
\end{prop}

\begin{proof}
Passing to  subsequences if necessary, we may assume that
\begin{align*}
A_1 &:=\lim_{n\rightarrow \infty} \|u_n\|_{L^2_x } ,\quad  A_2 := \lim_{n \rightarrow \infty} \|v_n\|_{L_x^2}, \\
\varepsilon_1 &:=\lim_{n\rightarrow \infty } \|e^{it\Delta}u_n\|_{L_{t,x}^{\frac{2(d+2)}{d}}},\quad  \varepsilon _2:= \lim_{n\rightarrow \infty } \|e^{i\kappa t\Delta } v_n\|_{L_{t,x}^{\frac{2(d+2)}{d}}}.
\end{align*}
For simplicity, we suppose that $\varepsilon_1 \leq \varepsilon_2$. Then we can apply  Lemma~\ref{inv-str} to $\{v_n\}$. Therefore, passing to a subsequence if necessary, there exist $\{\xi_n\} , \{y_n\} , \{\lambda_n \}, \{s_n\} $, and $\psi \in L^2 (\mathbb{R}^d)$ satisfying 
stated properties.
Since $\{h(0, \kappa \xi_n  , y_n , \lambda_n)^{-1} e^{i\frac{s_n}{\kappa} \Delta} u_n\}_n $ is bounded in $L^2 (\mathbb{R}^d)$, passing to a subsequence if necessary, there exists $\phi \in L^2 (\mathbb{R}^d )$ such that
\begin{align*}
h(0, \kappa \xi_n , y_n , \lambda_n )e^{i\frac{s_n}{\kappa} \Delta} u_n \rightharpoonup \phi \ \text{weakly in} \ L^2 (\mathbb{R}^d).
\end{align*}
Then \eqref{eq2} and \eqref{eq3} follow immediately. We can prove \eqref{eq4} as follows:
\begin{align*}
\|\phi\|_{L^2_x}^2 +\|\psi\|_{L_x^2}^2 \geq \|\psi\|_{L^2_x}^2 \gtrsim A_{2}^2 \left( \frac{\varepsilon_2}{A_2}\right)^{2(d+1)(d+2)} \gtrsim A^2 \left( \frac{\varepsilon}{A}\right)^{2(d+1)(d+2)}.
\end{align*}
\end{proof}

Now we are ready to give the profile decomposition.

\begin{thm}[Profile decomposition]\label{profile} Fix $d\geq 1$ and $\kappa >0$. 
Let $\{(u_n , v_n)\}$ be a bounded sequence in $L^2 (\mathbb{R}^d)^2$. Then, passing to a subsequence if necessary, there exist $J^{\ast} \in \mathbb{Z}_{\geq 0} \cup \{\infty\},\{(\phi^j ,\psi^j)\}_{j=1}^{J^{\ast}}$, 
$\{g_n^j  =g_{\kappa}(\theta_n^j , \xi_n^j , x_n^j , \lambda_n^j)\} \subset G_{\kappa}$, $\{W_n^J = (w_n^J , \zeta_n^J) \}_{J=0}^{J^{\ast}} \subset L^2 (\mathbb{R}^d)^2$, and $\{t_n^j\}_{j=1}^{J^{\ast}} \subset \mathbb{R}$ such that
\begin{align}
&(u_n ,v_n)=\sum_{j=1}^J g_n^j U_{\kappa }(t_n^j) (\phi^j ,\psi^j) +W_n^J,\label{eqeq2}\\
&\lim_{J\rightarrow J^{\ast}} \lim_{n\rightarrow \infty} \|U_{\kappa}(t)W_n^J\|_{L_{t,x}^{\frac{2(d+2)}{d}} (\mathbb{R}^{1+d})} =0,\label{eqeq3}\\
&U_{\kappa}(-t_n^j) (g_n^{j})^{-1}W_n^J \rightharpoonup 0\quad \text{weakly in $L^2 (\mathbb{R}^d)^2$}\quad \text{for each $1\leq j \leq J$},\label{eqeq4}\\
&\lim_{n\rightarrow \infty} [\|u_n\|_{L_x^2}^2 -\sum_{j=1}^J \|\phi^j\|_{L_x^2}^2 -\|w_n^J\|_{L_x^2}^2]=0,\label{eqeq6}\\
&\lim_{n\rightarrow \infty} [\|v_n\|_{L_x^2}^2 -\sum_{j=1}^J \|\psi^j\|_{L_x^2}^2 -\|\zeta _n^J\|_{L_x^2}^2]=0,\label{eqeq7}
\end{align}
where we set $U_{\kappa} (t)= (e^{it\Delta},e^{i\kappa t\Delta})$. Furthermore for each $j\neq k$ it follows that
\begin{align}
\qquad
\frac{\lambda_{n}^{j}}{\lambda_{n}^k} +\frac{\lambda_{n}^{k}}{\lambda_{n}^j} + \lambda_{n}^{j}\lambda_{n}^k |\xi_{n}^{j} -\xi_{n}^k|^2 +\frac{|t_{n}^j (\lambda_n^j)^2 -t_n^k (\lambda_n^k)^2|}{\lambda_n^j \lambda_n^k} +
\frac{|x_n^j -x_n^k -2t_n^j (\lambda_n^j)^2
(\xi_n^j -\xi_n^k)|^2}{\lambda_n^j \lambda_n^k} \rightarrow \infty.
\label{eqeq1}
\end{align}
\end{thm}

\begin{proof}
This theorem follows by  using Proposition \ref{prop} and the standard induction argument. For  the detail see \cite{MR3618884}.
\footnote{In that book, it was only treated the energy critical case. But we can similarly treat the mass critical system.}
 \end{proof}

\subsection{Construction of a minimal blow-up solution}
In this subsection we fix $\kappa =1/2$. Note that \eqref{NLS} has Galilean invariance in this case. 

First we show the existence of a minimal blow-up solution which does not necessarily have additional properties (Theorem \ref{thm3-1} below).

\begin{defn}
Let $g=g(\theta ,\xi_0 , x_0 ,\lambda) \in G$ and $(u,v): I \times\mathbb{R}^4 \rightarrow \mathbb{C}^2$. 
Then we define $T_{g}(u,v) \colon \lambda^2 I \times \mathbb{R}^4 \rightarrow \mathbb{C}^2$ as follows:
\begin{align*}
\left[T_g
\begin{pmatrix}
u \\
v
\end{pmatrix} 
\right]
(t,x)
:= 
\begin{pmatrix}
\lambda^{-2} e^{i\theta} e^{ix\cdot \xi_0}e^{-it|\xi_0|^2} u(\lambda^{-2}t, \lambda^{-1}(x-x_0 -2t\xi_0))\\
\lambda^{-2} e^{2i\theta} e^{2ix\cdot \xi_0} e^{-2it|\xi_0|^2} v(\lambda^{-2}t ,\lambda^{-1}(x-x_0-2t\xi_0))
\end{pmatrix}.
\end{align*}
For a solution $(u,v)$ to \eqref{NLS}, $T_g (u,v)$ is a solution to \eqref{NLS} with initial data $g(u_0,v_0)$ since we now assume  $\kappa =1/2$. 
\end{defn}

\begin{defn}[Almost periodic modulo symmetries]
Let $(u,v) : I \times \mathbb{R}^4 \rightarrow \mathbb{C}^2$  be a solution to \eqref{NLS}.
$(u,v)$ is said to be almost periodic modulo symmetries if there is a function $g\colon I\rightarrow G$ such that the set
$\{g(t)u(t)\mid t\in I\}$ is precompact in $L^2 (\mathbb{R}^4)^2$, or equivalently, if there exist functions $N\colon I\rightarrow \mathbb{R}_{>0}$, 
$x\colon I\rightarrow \mathbb{R}^4$, $\xi \colon I\rightarrow \mathbb{R}^4$, and $C\colon \mathbb{R}_{>0}\rightarrow \mathbb{R}_{>0}$ such that
\begin{equation}\label{tightness}
\begin{split}
\sup_{t\in I} \left[\int_{|x-x(t)|\geq \frac{C(\eta)}{N(t)}} |u(t,x)|^2 \, dx + \int_{|\xi -\xi (t)|\geq C(\eta)N(t)} |\hat{u}(t,\xi)|^2 \, d\xi\right]\leq \eta \\
\sup_{t\in I} \left[\int_{|x-x(t)|\geq \frac{C(\eta)}{N(t)}} |v(t,x)|^2\, dx +\int_{|\xi -2\xi(t)|\geq C(\eta)N(t)} |\hat{v}(t,\xi)|^{2}\, d\xi \right]\leq \eta
\end{split}
\end{equation}
for any $\eta >0$.
The functions $N(\cdot )$, $x(\cdot )$, $\xi (\cdot )$, and $C(\cdot )$ are called the frequency scale function, the spatial center function, the frequency center function, and the compactness modulus function, respectively.
\end{defn}

Note that the choice of $N(\cdot )$, $x(\cdot )$, and $\xi (\cdot )$ is not unique.
For instance, if another function $\ti{N}:I\to \R _+$ satisfies $C^{-1}\le N(t)/\ti{N}(t)\le C$ on $I$ for some $C>1$, then we can replace $N(t)$ with $\ti{N}(t)$.
In fact, it turns out that we can choose these functions to be continuous with respect to $t$, although we will not do so here.

\begin{rem}\label{rem2-3}
Note that a solution $(u,v)$ is almost periodic modulo symmetries if and only if $\{G(u(t),v(t))\mid t\in I\} $ is precompact in $G\setminus L^2 (\mathbb{R}^4)^2$.
\footnote{For $(\phi ,\psi)\in L^2 \times L^2$, we set $G(\phi,\psi) :=\{g(\phi,\psi) \mid g\in G\}$.}
The proof of this fact is very similar to the one of \cite[Lemma~A.2]{MR2470397}. For the convenience, we give the proof in Appendix \ref{AppendixA}.
\end{rem} 

\begin{thm}\label{thm3-1}
Assume $\kappa =1/2$.
There exists a maximal lifespan solution $(u,v)$ to \eqref{NLS} with $M(u,v)=M_c$, which is almost periodic modulo symmetries and which blows up both forward and backward in time.
\end{thm}

Before proving this theorem, we prepare the following proposition.
\begin{prop}\label{prop3-2}
Assume $\kappa =1/2$.
Let $(u_n , v_n ) \colon I_n \times \mathbb{R}^4 \rightarrow \mathbb{C}^2 $ be a sequence of solutions  to \eqref{NLS} and  $t_n \in I_n $ be a sequence of times such that $M(u_n , v_n) \leq M_c$, $M(u_n ,v_n) \rightarrow M_c$, and
\begin{align}\label{3-7}
\lim_{n\rightarrow \infty} S_{\geq t_n} (u_n , v_n) =\lim_{n\rightarrow \infty} S_{\leq t_n} (u_n, v_n) =+\infty,
\end{align} 
where $S_{\leq t_n} (u,v) :=\int_{(\inf I_n ,t_n] } \int_{\mathbb{R}^4} |u|^3 +|v|^3 \, dxdt $ and $S_{\geq t_n}:=\int_{[t_n ,\sup I_n)} \int_{\mathbb{R}^4} |u|^3 +|v|^3 \, dxdt $.
Then $G(u_n (t_n) ,v_n (t_n))$ has a subsequence which converges in $G\setminus L_x^2 (\mathbb{R}^4)^2$ topology.
\end{prop}

\begin{proof}
By time translation symmetry, we may assume $t_n =0 $ for any $n\in \mathbb{R}$. Applying  Theorem \ref{profile} to the bounded sequence $\{u_n (0) ,v_n(0)\}_n$ and passing to a subsequence if necessary, we obtain the profile decomposition
\begin{align}\label{minimal-1}
\begin{pmatrix}
u_n (0)\\
v_n (0)
\end{pmatrix}
=\sum_{j=1}^{J} g_n^j 
\begin{pmatrix}
e^{it_n^j \Delta} \phi^j \\
e^{i\frac{t_n^j}{2} \Delta} \psi^j 
\end{pmatrix}
+
W_n^J \end{align}
with stated properties, where we may assume  $t_n^j \equiv 0 $ or $t_n^j \rightarrow \pm{\infty}$ as $n\rightarrow \infty$. We define nonlinear profile $(a^j , b^j) :I^j \times \mathbb{R}^4 \rightarrow \mathbb{C}^2$ associated to 
$(\phi^j ,\psi^j) $ as follows:
\begin{itemize}
\item If $t_n^j \equiv 0$, we define $(a^j ,b^j)$ to be maximal-lifespan solution with $(a^j (0) , b^j (0))=(\phi^j ,\psi^j)$.
\item If $t_n^j \rightarrow \infty$, we define $(a^j , b^j)$ to be the maximal-lifespan solution which scatters forward in time to $(e^{it\Delta } \phi^j , e^{i\frac{t}{2}\Delta} \psi^j )$.
\item If $t_n^j \rightarrow -\infty$, we define $(a^j , b^j)$ to be the maximal-lifespan solution which scatters backward in time to $(e^{it\Delta } \phi^j , e^{i\frac{t}{2}\Delta} \psi^j )$.\end{itemize}
Finally, for each $j,n\geq 1$ we define $(a_n^j ,b_n^j):I_n^j \times \mathbb{R}^4 \rightarrow \mathbb{C}^2$ by
\begin{align}
\begin{pmatrix}
a_n^j \\
b_n^j 
\end{pmatrix}
:=T_{g_n^j}
\left[
\begin{pmatrix}
a^j (\cdot +t_n^j) \\
b^j (\cdot + t_n^j)
\end{pmatrix}
\right],
\end{align}
where $I_n^j :=\{t\in\mathbb{R} \mid (\lambda_n^j)^{-2} t +t_n^j \in I^j \}$.
\footnote{
 By passing to subsequences for each $j$ repeatedly and the standard diagonal argument, we may assume $t_n^j \in I^j$ for all $j,n\geq 1$.}
 Note that each $(a_n^j ,b_n^j)$ is a solution with initial value $(a_n^j (0) , b_n^j (0)) =g_n^j (a^j (t_n^j) ,b^j (t_n^j))$. By the definition of $(u_n^j ,v_n^j)$ and \eqref{minimal-1} we can easily check 
for each $J$ that 
\begin{align}
u_n (0 )-\sum_{j=1}^J a_n^j (0) -w_n^J \rightarrow 0 \quad \text{in $L_x^2$ as  $n\rightarrow \infty$} \label{3-3},\\
v_n(0) -\sum_{j=1}^{J} b_n^j (0) - \zeta_n^J \rightarrow 0\quad \text{in $L_x^2$ as $n\rightarrow \infty$ .}\label{3-4}
\end{align}\label{3-8}
By the property stated in Theorem \ref{profile}, we also obtain that
\begin{align}\label{3-1}
\sum_{j=1}^{J^{\ast}} M(\phi^j , \psi^j) \leq \limsup_{n\rightarrow \infty} M(u_n (0), v_n (0))\leq M_c 
\end{align}
and in particular $\sup_{j} M(\phi^j ,\psi^j) \leq M_c$.\\
 Now we prove by contradiction that 
 \begin{equation}\label{mc}
 M(\phi_1 ,\psi_1)=M_c. 
 \end{equation} If not, then we have 
\begin{align}
\sup_j M(\phi^j ,\psi^j) \leq M_c -\varepsilon .
\end{align}
for some $\varepsilon >0$. In this case, by the definition of  the critical mass $M_c$, all $(a_n^j ,b_n^j)$ are defined globally and obey the estimates
\begin{align*}
M(a_n^j ,b_n^j) =M(\phi^j ,\psi^j) \leq M_c -\varepsilon .
\end{align*} 
Moreover by small data scattering, we get
\begin{align}\label{3-2}
S_{\mathbb{R}} (a_n^j , b_n^j)\leq L(M(\phi^j ,\psi^j)) \leq C_0 M(\phi^j , \psi^j)^{\frac{3}{2}} \leq C_1 M(\phi^j ,\psi^j).
\end{align}
In order to lead a contradiction, we define the approximate solution 
\begin{align}
\begin{pmatrix}
u_n^J \\
v_n^J
\end{pmatrix}
:=\sum_{j=1}^{J} 
\begin{pmatrix}
a_n^j \\
b_n^j 
\end{pmatrix}
+
\begin{pmatrix}
e^{it\Delta} w_n^J \\
e^{i\frac{t}{2}\Delta} \zeta_n^J 
\end{pmatrix}
\end{align}
Note that by the asymptotic orthogonality conditions \eqref{eqeq1}, \eqref{3-1}, and \eqref{3-2}, 
\begin{align}
\limsup_{J\rightarrow J^{\ast}} \limsup_{n\rightarrow \infty} S_{\mathbb{R}} (u_n^J , v_n^J) 
&\leq 2^3\limsup_{J\rightarrow J^{\ast}} \limsup_{n\rightarrow \infty}  S_{\mathbb{R}} (\sum_{j=1}^J a_n^j ,\sum_{j=1}^J b_n^j) \\
&=2^3  \limsup_{J\rightarrow J^{\ast}} \limsup_{n\rightarrow \infty} \sum_{j=1}^J S_{\mathbb{R}} (a_n^j ,b_n^j) \\
&\leq C_2 \lim_{J\rightarrow J^{\ast}}  \sum_{j=1}^J M(\phi^j ,\psi^j) \leq C_2 M_c . \label{3-5}
\end{align}
Furthermore, $(u_n^J ,v_n^J)$ satisfies the following equation:
\begin{align*}
\begin{cases}
(i\partial_t +\Delta)u_n^J =\sum_{j=1}^J b_n^j \bar{a}_n^j = v_n^J\bar{u}_n^J  +e_{1,n}^J,  \\
(i\partial_t +\frac{1}{2} \Delta) v_n^J =\sum_{j=1}^J (a_n^j)^2 =(u_n^J)^2 +e_{2,n}^J ,
\end{cases}
\end{align*}
where $e_{1,n}^J :=\sum_{j=1}^J b_n^j \bar{a}_n^j -v_n^J\bar{u}_n^J ,\  e_{2,n}^J:=\sum_{j=1}^J (a_n^j)^2 -(u_n^J)^2$.
To apply Theorem \ref{theo1}, we prepare the following lemmas:
\begin{lem}\label{lem3-1}
For any $J\geq 1$, the following holds: 
\begin{align*}
\lim_{n\rightarrow \infty} M(u_n^J (0) -u_n (0) , v_n^J (0) - v_n (0)) =0 .
\end{align*}
\end{lem}

\begin{proof}
This follows immediately by \eqref{3-3}, \eqref{3-4}, and the definition of $(u_n^J ,v_n^J )$.
\end{proof}

\begin{lem}\label{lem3-2}
The following holds:
\begin{align*}
\lim_{J\rightarrow J^{\ast}} \limsup_{n\rightarrow \infty} \|(e_{1,n}^J , e_{2,n}^J)\|_{L_{t,x}^{\frac{3}{2}} (\mathbb{R}^{1+4})} =0.
\end{align*}
\end{lem}

\begin{proof}
By the definition, we have
\begin{align*}
\|(e_{1,n}^J , e_{2,n}^J)\|_{L_{t,x}^{\frac{3}{2}}(\mathbb{R}^{1+4})}
&=\|(\sum_{j=1}^{J}b_n^j \bar{a}_n^j -v_n^J\bar{u}_n^J,
\sum_{j=1}^J (a_n^j)^2 -(u_n^J)^2 )\|_{L_{t,x}^{\frac{3}{2}}(\mathbb{R}^{1+4})} \\
&=\|\sum_{j=1}^{J}b_n^j \bar{a}_n^j -v_n^J\bar{u}_n^J\|_{L_{t,x}^{\frac{3}{2}}} +\|\sum_{j=1}^J (a_n^j)^2 -(u_n^J)^2\|_{L_{t,x}^\frac{3}{2}}\\
&=:I+II.
\end{align*}
Using the triangle inequality and the H\"{o}lder inequality, we get
\begin{align*}
I
&=\|\sum_{j=1}^{J}b_n^j \bar{a}_n^j -(\sum_{j=1}^J b_n^j +e^{i\frac{t}{2}\Delta} \zeta_n^{J} )(\sum_{j=1}^J a_n^j +e^{it\Delta} w_n^J)^{\ast}\|_{L_{t,x}^{\frac{3}{2}}}\\
&\leq \sum_{j\neq k}\|a_n^j b_n^k \|_{L_{t,x}^{\frac{3}{2}}} +\|\sum_{j=1}^J b_n^j \|_{L_{t,x}^3} \|e^{it\Delta}w_n^J\|_{L_{t,x}^3} + \|\sum_{j=1}^J a_n^j\|_{L_{t,x}^3}\|e^{i\frac{t}{2}\Delta} \zeta_n^J\|_{L_{t,x}^3}
+\|e^{i\frac{t}{2}\Delta} \zeta_n^J\|_{L_{t,x}^3}\|e^{it\Delta}w_n^J\|_{L_{t,x}^3}, \\
II&=\|\sum_{j=1}^J (a_n^j)^2 -(\sum_{j=1}^J a_n^j +e^{it\Delta} w_n^J)^2\|_{L_{t,x}^{\frac{3}{2}}}\\
&\leq \sum_{j\neq k}\|a_n^j a_n^k\|_{L_{t,x}^{\frac{3}{2}}} +2 \|\sum_{j=1}^J a_n^j\|_{L_{t,x}^3}\|e^{it\Delta} w_n^J\|_{L_{t,x}^3} +\|e^{it\Delta}w_n^J\|_{L_{t,x}^3}^2.
\end{align*}
From asymptotic orthogonality \eqref{eqeq1}, first terms of $I,II$ converge to zero. Moreover the other terms converge to zero by \eqref{eqeq3} and \eqref{3-5}. 
\end{proof}

\begin{lem}\label{lem3-3}
There exists a constant $M_1 >0$ such that
\begin{align*}
\limsup_{J\rightarrow J^{\ast}} \limsup_{n\rightarrow \infty} \|(u_n^J , v_n^J)\|_{L_{t}^{\infty} L_{x}^{2} (\mathbb{R}^{1+4})} \leq M_1.
\end{align*}
\end{lem}

\begin{proof}
By Lemma \ref{lem3-1}, for each $J\geq 1$
\begin{align}\label{3-6}
\limsup_{n\rightarrow \infty} \|(u_n^J (0) , v_n^J (0))\|_{L_x^2 (\mathbb{R}^4)} \leq \sup_{n\geq 1}\|(u_n (0), v_n (0))\|_{L_{x}^2 (\mathbb{R}^4)} .
\end{align}
By Strichartz's estimate we have
\begin{align*}
\|(u_n^J , v_n^J)\|_{L_{t}^{\infty}L_x^2 (\mathbb{R}^{1+4})} \leq \|(u_n^J (0) , v_n^J (0))\|_{L_x^2 (\mathbb{R}^4)} +C(\|(e_{1,n}^J ,e_{2,n}^J )\|_{L_{t,x}^{\frac{3}{2}}(\mathbb{R}^{1+4}) }
+\|(u_n^J , v_n^J)\|_{L_{t,x}^3 (\mathbb{R}^{1+4})}^2
).
\end{align*}
Combining this with \eqref{3-5}, \eqref{3-6} and, Lemma \ref{lem3-2}, we obtain the result.
\end{proof}

Combining these Lemmas \ref{lem3-1}-\ref{lem3-3} and Theorem \ref{theo1} , we obtain boundedness of $\{S_{\mathbb{R}} (u_n , v_n)\}_n$ which contradicts \eqref{3-7}.\\
Therefore we establish \eqref{mc}.
Then  we see $J^{\ast} =1$. Consequently, the profile decomposition simplifies to
\begin{align}\label{3-9}
\begin{pmatrix}
u_n (0)\\
v_n (0)
\end{pmatrix}
=
g_n
\begin{pmatrix}
e^{it_n\Delta} \phi \\
e^{i\frac{t_n}{2}\Delta} \psi 
\end{pmatrix}
+
\begin{pmatrix}
w_n \\
\zeta_n
\end{pmatrix}
\end{align}
for some $t_n \in \mathbb{R} $ such that either $t_n \equiv 0$ or $t_n \rightarrow \pm{ \infty}$, some $g_n \in G$, some $(\phi ,\psi )\in L_x ^2 (\mathbb{R}^{4})^2$ of mass $M(\phi ,\psi )=M_c$, and some $(w_n ,\zeta_n) $ with
$M(w_n ,\zeta_n) \rightarrow 0$. If $t_n \equiv 0$, then 
\begin{align*}
M((u_n (0), v_n (0)) -g_n (\phi , \psi )) =M(w_n , \zeta_n) \rightarrow 0.
\end{align*}
This implies that $G(u_n (0), v_n (0))$ converges in $G\setminus L_x^2 (\mathbb{R}^4)^2$. So we consider the case $t_n \rightarrow +\infty$. The case $t_n \rightarrow -\infty$ can be treated similarly, and so we omit it.
By the Strichartz inequality we have
\begin{align*}
S_{\mathbb{R}} (e^{it\Delta}\phi ,e^{i\frac{t}{2}\Delta}\psi ) \lesssim \|(\phi ,\psi)\|_{L_x^2 (\mathbb{R}^4)} <\infty 
\end{align*}
and so 
\begin{align*}
S_{\geq 0} (e^{it\Delta}e^{it_n \Delta} \phi ,e^{i\frac{t}{2}\Delta}e^{i\frac{t_n}{2}\Delta} \psi ) =
S_{\geq t_n} (e^{it\Delta} \phi ,e^{i\frac{t}{2}\Delta} \psi ) \rightarrow 0\ as\ n\rightarrow \infty .
\end{align*}
Let $\theta_n , \xi_n , x_n , \lambda_n $ be  parameters of $g_n$ and set $h_{1,n} :=h(\theta_n , \xi_n ,x_n ,\lambda_n) , h_{2,n} :=h(2\theta_n , 2\xi_n ,x_n , \lambda_n )$.
Then  we establish 
\begin{align*}
S_{\geq 0} (e^{it\Delta} h_{1,n} e^{it_n \Delta}\phi , e^{i\frac{t}{2}\Delta} h_{2,n} e^{i\frac{t_n}{2} \Delta} \psi)
&=S_{\geq 0} (T_{h_{1,n}} e^{it\Delta}e^{it_n \Delta}\phi ,T_{h_{2,n}} e^{i\frac{t}{2}\Delta}  e^{i\frac{t_n}{2} \Delta} \psi)\\
&=S_{\geq 0} (e^{it\Delta}e^{it_n \Delta}\phi , e^{i\frac{t}{2}\Delta}  e^{i\frac{t_n}{2} \Delta} \psi) \rightarrow 0 ,
\end{align*}
as $n\rightarrow 0$. Since $S_{\mathbb{R}} (e^{it\Delta} w_n , e^{i\frac{t}{2} \Delta} \zeta_n) \rightarrow 0$,\footnote{
see \eqref{eqeq3} in Theorem {\ref{profile}}
} we see from \eqref{3-9} that
\begin{align*}
\lim_{n\rightarrow \infty } S_{\geq 0} (e^{it\Delta} u_n (0), e^{i\frac{t}{2}\Delta} v_n (0) ) =0 .
\end{align*}
Applying Theorem \ref{theo1} (using $ (0,0)$ as the approximate solution ), we conclude that
\begin{align*}
\lim_{n\rightarrow \infty} S_{\geq 0} (u_n , v_n) =0.
\end{align*}
This contradicts \eqref{3-7}.
\end{proof}

\begin{proof}[Proof of Theorem \ref{thm3-1}]
 Since $L(M_c)=\infty$, we can find a sequence $(u_n,v_n ) :I_n \times \mathbb{R}^4 \rightarrow \mathbb{C}^2$ of maximal-lifespan solutions with $M(u_n ,v_n) \leq M_c$ 
and $\lim_{n\rightarrow \infty}S_{I_n} (u_n ,v_n) =+\infty .$ Then there exist $t_n \in I_n$ such that
\begin{align*}
\lim_{n\rightarrow \infty} S_{\geq t_n} (u_n , v_n)=\lim_{n\rightarrow \infty} S_{\leq t_n} (u_n ,v_n)=\infty.
\end{align*}
By time translation invariance we may take $t_n =0$. Invoking Proposition \ref{prop3-2} and passing to a subsequence if necessary, we find $g_n \in G$ such that $g_n (u_n (0) , v_n (0)) \rightarrow (u_0, v_0)$ in $L^2 (\mathbb{R}^4)^2$
for some $(u_0 ,v_0 ) \in L^2 (\mathbb{R}^4)^2$.\footnote{
From $S_{I_n }(u_n , v_n) \rightarrow \infty$ and $M(u_n ,v_n ) \leq M_c$, passing to a subsequence if necessary, we have $\lim_{n\rightarrow \infty} M(u_n ,v_n) =M_c$.}

By applying $T_{g_n}$ to $(u_n, v_n)$ we may take $g_n =I$ for all $n\in \mathbb{N}$. Then we have $(u_n (0) , v_n (0)) \rightarrow (u_0 ,v_0)$ and so $M(u_0 ,v_0) \leq M_c$.
Let $(u,v):I\times \mathbb{R}^4 \rightarrow \mathbb{C}^2$ be the maximal-lifespan solution  with initial data $(u(0), v(0))=(u_0 ,v_0)$. We claim that $(u,v)$ blows up both forward and backward in time.
Indeed, if $(u,v)$ does not blow up forward in time, then $[0,\infty) \subset I$ and $S_{\geq 0} (u,v) < \infty $.  
By Theorem \ref{theo1} this implies that for sufficiently large $n$, we have $[0,\infty) \subset I_n$ and $
\limsup_{n\rightarrow \infty} S_{\geq 0} (u_n, v_n) < \infty$. This is a contradiction. 
It remains to show that the solution $(u,v)$ is almost periodic modulo symmetries.
Consider an arbitrary sequence 
\begin{align*}
\{(u(s_n),v(s_n))\}_n \subset (u,v)[I].
\end{align*}
Since $(u,v)$ blows up both forward and backward in time, we have 
\begin{align*}
S_{\geq s_n} (u,v) =S_{\leq s_n} (u,v) =\infty .
\end{align*}
Applying Proposition \ref{prop3-2} once again, we see that 
$G(u(s_n) ,v(s_n))$ has a convergent subsequence in $G\setminus L^2(\mathbb{R}^d)^2$.
Thus the orbit $\{G(u(t),v(t)) \mid t\in I\}$ is  precompact in $G\setminus L^2(\mathbb{R}^d)^2$.
\end{proof}

\subsection{Further refinements}
In this subsection, we again let $\kappa =1/2$. 
In the following, we consider to refine given solution in Theorem \ref{thm3-1} as we can apply the argument of \cite{MR3406535}. For this aim we give the some definitions and lemmas. Following argument is based on \cite{MR3098643} and \cite{MR2355070}.
\begin{defn}[Convergence of solutions ]Let $(u_n,v_n) : I_n \times \mathbb{R}^4 \rightarrow \mathbb{C}^2$ be a sequence of solutions to \eqref{NLS},
let $(u,v)\colon I\times \mathbb{R}^4 \rightarrow \mathbb{C}^2$ be another solution, and $K$ be a compact time interval. We say that $(u_n ,v_n)$ converges  uniformly  to  $(u,v)$ on $K$ if 
$K\subset I $, $K\subset I_n $ for sufficiently large $n$, and $(u_n ,v_n) \rightarrow (u,v)$ in $C(K, L^2(\mathbb{R}^4)^2)\cap L_{t,x}^{3}(K\times \mathbb{R}^4)^2$ as $n\rightarrow \infty$. We say that $(u_n ,v_n)$ converge locally uniformly to $(u,v)$
if $(u_n ,v_n)$  converges  uniformly  to  $(u,v)$ on every compact interval $K\subset I$. 
\end{defn}
\begin{defn}
Let $(u,v)$ be a solution to \eqref{NLS} which is almost periodic modulo symmetries with parameters $N(t) , x(t), \xi(t)$. We say that $(u,v)$ is normalized if the lifespan $I$ contains zero and 
\begin{align*}
N(0)=1,\ x(0)=\xi (0)=0.
\end{align*}
More generally , we can define the normalization of a solution $(u,v)$ at time $t_0 \in I$ by 
\begin{align}
(u^{[t_0]}, v^{[t_0]}):=T_{g(0, -\xi(t_0)/N(t_0) , -x(t_0)N(t_0),N(t_0) )} (u(\cdot +t_0),v(\cdot +t_0)).
\end{align} 
Note that $(u^{[t_0]} ,v^{[t_0]})$ is a normalized solution which is almost periodic modulo symmetries and has lifespan 
\begin{align*}
I^{[t_0]} := \{s\in \mathbb{R}| t_0 + sN(t_0)^{-2}\in I\}.
\end{align*} 
The parameters of $(u^{[t_0]},v^{[t_0]})$ are given by 
\begin{align}
&N^{[t_0]} (s) := \frac{N(t_0 +sN(t_0)^{-2})}{N(t_0)} \nonumber ,\\
&\xi^{[t_0]} (s):= \frac{\xi(t_0 +sN(t_0)^{-2}) -\xi(t_0)}{N(t_0)} ,\\
&x^{[t_0]} (s):=N(t_0) [x(t_0 + sN(t_0)^{-2}) -x(t_0)] -2\frac{\xi(t_0)}{N(t_0)} s \nonumber.
\end{align}
\end{defn} 
\begin{lem}[\cite{MR3098643},Theorem 5.15]\label{lem3-8}
Let $(u_n ,v_n)$ be a sequence of solutions to \eqref{NLS} with lifespans $I_n$, which is almost periodic modulo symmetries with parameters $N_n ,x_n , \xi_n$ and compactness modulus function $C$.
Suppose that $(u_n ,v_n)$ converges locally uniformly to  a non-zero solution $(u,v)$ with lifespan $I$. Then $(u,v)$ is almost periodic modulo symmetries with some parameters $N(t) , x(t) ,\xi (t) $ and the same compactness modulus function 
$C$. In particular we may take $N(t) =\limsup_{n\rightarrow \infty} N_n (t)$.
\end{lem}
\begin{proof}
1. We first show that
\begin{align}
0<\liminf_{n\rightarrow \infty} N_n (t) \leq \limsup_{n\rightarrow \infty} N_n (t) <\infty , \label{3-8-1}\\
\limsup_{n\rightarrow \infty} |x_n (t)| +\limsup_{n\rightarrow \infty} |\xi _n (t)| <\infty ,\label{3-8-2}
\end{align}
for all $t\in I$. If $\eqref{3-8-1} $ is failed for some $t\in I$, passing to a subsequence if necessary,
\begin{align*}
\lim_{n\rightarrow \infty} N_n (t) =0 \ \text{or}\ \lim_{n\rightarrow \infty} N_n (t) =+\infty .
\end{align*}
Then we have $u_n (t) , v_n (t) \rightharpoonup 0\ \text{weakly in}\  L^2_x(\mathbb{R}^4)$. Indeed, taking $\phi \in C_0^{\infty} (\mathbb{R}^4)$ and $\varepsilon>0$ arbitrarily , we have
\begin{align*}
|\int_{\mathbb{R}^4} u_n (t) \varphi \ dx|^2 
&\lesssim \|\phi\|_{L^2_x}^2 \varepsilon +(\int_{|x-x(t)|<C(\varepsilon )/ N_n (t)} |u_n (t)| |\phi| \ dx )^2\\
&\lesssim \|\phi\|_{L^2_x}^2 \varepsilon + \|\phi\|_{L^{\infty}} \|u_n (t)\|_{L^2}^2 (C(\varepsilon) / N_n (t))^4 \rightarrow \|\phi\|_{L^2}^2 \varepsilon ,
\end{align*}
if $N_n (t)\rightarrow +\infty$ as $n\rightarrow \infty$. In the case $N_n (t) \rightarrow 0$, we can get the result by similar argument and using Plancherel's theorem. This contradicts the nonzero assumption of $(u,v)$. 
We can easily obtain \eqref{3-8-2} by similar argument, and so we omit the detail.
Passing to a subsequences if necessary
\footnote{Note that subsequences depend on $t$.}, we may assume
\begin{align*}
&N(t):=\lim_{n\rightarrow \infty} N_n (t)  \\
&\xi (t) :=\lim_{n\rightarrow \infty} \xi_{n} (t)\\
&x (t) := \lim_{n\rightarrow \infty} x_n (t).
\end{align*}
Then we can easily prove almost periodicity of $(u,v)$ with parameters $N(t), \xi (t),x(t).$
\end{proof}
\begin{lem}\label{lem3-9}
Let $(u_n ,v_n)$ be a sequence of normalized maximal-lifespan solutions with lifespans $0\in I_n= (-T_n^{-} ,T_n^{+})$, which are almost periodic modulo symmetries on $[0,T_n^{+})$ with parameters $N_n (t) , x_n (t), \xi _n (t)$ and 
a uniform compactness modulus function $C$. Assume that we also have 
\begin{align*}
0<\inf_n M(u_n ,v_n) \leq \sup_n M(u_n ,v_n) <\infty .
\end{align*}
Then, after passing to a subsequence if necessary, there exists a non-zero maximal lifespan solution $(u,v)$ with lifespan $0\in I=(-T^{-} ,T^{+})$ that is almost periodic modulo symmetries on $[0,T^+) $, such that 
$(u_n ,v_n) $ converge locally uniformly to $(u,v)$ on I.
\end{lem}
\begin{proof}
By almost periodicity and the assumption $\sup_n M(u_n ,v_n) <\infty$, $\{(u_n ,v_n) (0) \mid n\in \mathbb{N}\}$ is precompact in $L^2 (\mathbb{R}^4)^2$ , and so passing to a subsequence if necessary, we may assume that
\begin{align*}
(u_n (0) , v_n (0) ) \rightarrow (u_0 ,v_0) \ in\  L^2 (\mathbb{R}^4)^2\ for \ some\  (u_0, v_0).
\end{align*} 
Let $(u,v):I \rightarrow \mathbb{C}^2$ be the maximal lifespan solution with initial data $(u,v)(0) =(u_0, v_0)$. Then we can apply the stability result, and  so we establish
$(u_n, v_n) \rightarrow (u,v)$ locally uniformly on $I$. Almost periodicity of $(u,v)$ on $[0,T^+)$ follows from Lemma \ref{lem3-8}.
\end{proof}
\begin{lem}\label{lem3-10}
Let $(u,v)$ be a non-zero maximal-lifespan solution with lifespan $I=(-T^- ,T^+)$ that is almost periodic modulo symmetries on $[0,T^+ )$ with parameters $N(t) ,x(t),\xi (t)$. Then there exists a small number $\delta =\delta (u,v)>0 $
such that  for every $t_0\in [0,T^+)$ we have 
\begin{align}
[0, t_0 +\delta N(t_0)^{-2}] \subset [0,T^+) \label{3-10-1}
\end{align}
and 
\begin{align}
N(t) \sim_{u,v} N(t_0) \quad \text{whenever}\quad \max\{0, t_0 -\delta N(t_0)^{-2} \}\leq t \leq t_0 + \delta N(t_0)^{-2}. \label{3-10-2}
\end{align}
\begin{proof}
First we prove \eqref{3-10-1} by contradiction.
If not there exist sequences $t_n \in [0,T^{+}) $ and $\delta_n >0$ such that 
\begin{align*}
t_n +\delta_n N(t_n)^{-2} \notin [0,T^+) ,\ \delta_n \rightarrow 0.
\end{align*}
Applying Theorem \ref{lem3-9} there exists the maximal-lifespan solution $(u,v) :I=(-T^- ,T^+) \rightarrow \mathbb{C}^2$ such that 
\begin{align*}
(u^{[t_n]} ,v^{[t_n]}) \rightarrow (u,v)  \quad \text{locally  uniformly  on $I$} ,
\end{align*}
where $(u^{[t_n]} ,v^{[t_n]}) :I^{[t_n]} \rightarrow \mathbb{C}^2$ is normalized solution at time $t_n $. Then there exists $n_0 \in \mathbb{N}$ such that $\delta_n \in [0,T^+)$ for all $n\geq n_0$, and so there exists $m(n) \in \mathbb{N}$ such that
$t_m + \delta_n N(t_m)^{-2} \in [0, T^+)$ for $n \geq n_0 , \ m\geq m(n)$. This is a contradiction.\\
Next, we prove also \eqref{3-10-2} by contradiction. If not we can find sequences $t_n , t' _n \in [0,T^+)$ such that
\begin{align*}
s_n :=(t'_n -t_n)N(t_n)^2\rightarrow 0 
\end{align*}
and
\begin{align*}
\frac{N(t'_n)}{N(t_n)} \rightarrow 0\quad \text{or} \quad +\infty \quad \text{as} \quad n\rightarrow \infty .
\end{align*}
Then we easily get $(u^{[t_n]}(s_n), v^{[t_n]}(s_n)) \rightharpoonup (0,0) $ weakly in $L^2 (\mathbb{R}^4)^2$. 
On the other hand, by Lemma \ref{lem3-9}, there exists a maximal-lifespan non-zero solution $(u,v):I\rightarrow \mathbb{C}^2 $ such that $(u^{[t_n]}, v^{[t_n]}) \rightarrow (u,v) $ locally uniformly on $I$.
Then we have $(u^{[t_n]},v^{[t_n]})(s_n) \rightarrow (u(0),v(0))$ in $L^2 (\mathbb{R}^d)^2 $ as $n\rightarrow \infty$, and so $(u,v)=(0,0)$. This is a contradiction .
\end{proof}
\end{lem}
\begin{cor}\label{cor3-11}
Let $(u,v)$ be a maximal lifespan solution with lifespan $0\in I=(-T^- ,T^+)$ that is almost periodic modulo symmetries  on $[0, T^+)$ with frequency scale function $N \colon [0,T^+) \rightarrow \mathbb{R}_{>0}$.
If $T^+ <\infty$, then it follows that $\lim_{t\rightarrow T^+}N(t)=\infty$. If $T^+ =\infty $, then we have $N(t)\gtrsim_{u,v} \min \{N(0), t^{-1/2}\} \ {\forall} t\in [0, \infty)$.\end{cor}
\begin{proof}
If $T^+ <\infty$, the result follows easily by Lemma \ref{lem3-10}. 
Consider the case $T^+ =\infty$. Take any $ t\in [0,\infty)$. If $\max \{0, t-\delta N(t)^{-2}\} =0$, then by   Lemma \ref{lem3-10}, we obtain $N(t) \gtrsim N(0)$. 
On the other hand, if $0 \leq t-\delta N(t)^{-2} $, then we get easily that $N(t)\geq \frac{\delta^{1/2}}{t^{1/2}}$.
\end{proof}
\begin{thm}[Existence of minimal mass blow-up solution ]\label{thm:mmbs}
There exists a maximal-lifespan solution $(u,v)$ with lifespan $I=(-T^- ,\infty)$ that is almost periodic modulo symmetries on $[0,\infty)$ and blows up forward in time satisfying $M(u,v)=M_c$, $N(0)=\sup_{t\in [0,\infty)}N(t)=1$, and  $x(0)=\xi (0)=0$. 
\end{thm}
\begin{proof}
Let $(\tilde{u},\tilde{v}):J\rightarrow \mathbb{C}^2$ be an almost periodic  solution constructed in Theorem \ref{thm3-1} with frequency scale function $\tilde{N} :J\rightarrow \mathbb{R}_{>0}$.
 Take a sequence $\{J_n\}$ of compact intervals satisfying $J_n \nearrow J$. 
Since $\sup_{t\in J_n} \tilde{N} (t) <\infty$, there exist $t_n \in J_n$ such that 
\begin{align*}
\sup_{t\in J_n} \tilde{N} (t) \leq 2 \tilde{N} (t_n) .
\end{align*}
Note that $M(\tilde{u}^{[t_n]} ,\tilde{v}^{[t_n]}) =M(u,v) >0$. Since $(\tilde{u}^{[t_n]} (0) ,\tilde{v}^{[t_n]}(0))$ is precompact in $L^2 (\mathbb{R}^4)^2$, passing to a subsequence if necessary, it follows that 
\begin{align*}
(\tilde{u}^{[t_n]} (0) ,\tilde{v}^{[t_n]}(0)) \rightarrow (u_0,v_0)\quad \text{in $L^2 (\mathbb{R}^4)^2$ for  some  $(u_0,v_0)$.}
\end{align*}
Let $(u,v):I\times \mathbb{R}^4 \rightarrow \mathbb{C}^2$ be the maximal-lifespan solution with $(u(0), v(0)) =(u_0 ,v_0)$ and $(u^n,v^n) :I_n \times \mathbb{R}^4 \rightarrow \mathbb{C}^2$ be the maximal-lifespan solution with 
$(u^n (0),v^n (0)) =(\tilde{u}^{[t_n]} (0) ,\tilde{v}^{[t_n]}(0))$.
If $0\in K$ is compact subinterval of $I=(-T^{-} ,T^{+})$, from the stability result we see that
\begin{align*}
(u^n ,v^n ) \rightarrow (u,v) \quad \text{uniformly  on} \quad K.
\end{align*}
In particular $\limsup_{n\rightarrow \infty} \|(u^n,v^n )\|_{L_{t,x}^3(K\times \mathbb{R}^4)} <\infty $ .   
On the other hand, we have
\begin{align*}
\|(u^n ,v^n)\|_{L_{t,x}^3(J_n^{[t_n]} \times \mathbb{R}^4)} =\|(u,v)\|_{L_{t,x}^3 (J_n \times \mathbb{R}^4)} \rightarrow \infty .
\end{align*}
Therefore we obtain $J_n^{[t_n]} \not\subset K$ for sufficiently large $n$.
Since $0\in K$ is an arbitrary  compact subinterval of $I$, after passing to a subsequence, we may assume one of the following holds:
\begin{itemize}
\item  For every $t\in (0, T^+)$, $t\in J_n^{[t_n]}$ for all sufficiently large $n$.
\item For every $t\in (0, T^{-})$, $t\in J_n^{ [t_n]}$ for all sufficiently large $n$.
\end{itemize}
By time reversal symmetry, it suffices to consider the former possibility.
Then it follows that 
\begin{align*}
(\tilde{u}^{[t_n]} ,\tilde{v}^{[t_n]}) \rightarrow (u,v)\quad \text{locally  uniformly  on} \quad [0,T^{+}).
\end{align*}
Applying Lemma \ref{lem3-8} , we see that $(u,v) $ is almost periodic modulo symmetry with frequency scale function $N_0 (t) =\limsup_{n\rightarrow \infty} \tilde{N}^{[t_n]} (t) \leq 2$.
By Corollary \ref{cor3-11}, then we get $T^{+} =\infty$.
Setting 
\begin{align*}
N(t) :=
\begin{cases}
 N_0 (t) / \sup_{s\in [0,\infty)} N_0 (s)\ \ &\ \  t>0 \\
 1\ \ &\ \ t=0,
 \end{cases}
 \end{align*}
 we have
\begin{align*}
\min\{ \frac{1}{\sup_{s\in [0,\infty)}N_0 (s)} ,\frac{1}{N_0 (0)} \} \leq \frac{N(t)}{N_0 (t)} \leq \max\{\frac{1}{\sup_{s\in [0,\infty)} N_0 (s)} ,\frac{1}{N_0 (0)} \}
\end{align*}
and so we can replace $N_0$ with $N$.
It remains to show that $(u,v)$ blows up forward in time . If not $(u,v)$ scatters to $(e^{it\Delta} u^+ , e^{i\frac{t}{2}\Delta } v^+  )$ as $t\rightarrow \infty$ for some $(u^+ ,v^+) \in L^2 (\mathbb{R}^4)^2$.
Furthermore we can show that $(e^{-it\Delta}u(t) , e^{-i\frac{t}{2}\Delta} v(t) ) \rightharpoonup (0,0)$ weakly in $L^2 (\mathbb{R}^4)^2$.
Indeed, for any $(\phi ,\psi )\in C_0^{\infty} (\mathbb{R}^4)^2$, we can calculate as follows :
\begin{align*}
&|(e^{-it\Delta}u(t) ,\phi )_{2}|^2 \leq \eta \|\phi\|_{L^2}^2 +\|u(t)\|_{L^2}^2 \|e^{it\Delta} \phi\|_{L^2 (B(x(t) , C(\eta)/N(t) ))}^2, \\
&|(e^{-i\frac{t}{2}\Delta} v(t)) , \psi )_{2}|^2 \leq \eta \|\psi\|_{L^2}^2 +\|v(t)\|_{L^2}^2\|e^{i\frac{t}{2} \Delta } \psi \|_{L^2 (B(x(t) , C(\eta ) /N(t)))}^2. 
\end{align*} 
By the dispersive estimate and Corollary \ref{cor3-11}, we obtain the desired result. Therefore we get $(u^+ ,v^+)=(0,0)$. This is a contradiction.
\end{proof}

\subsection{Minimal blow-up solution in the radial case}
In this subsection we introduce the minimal blow-up solution in the radial case. 
\begin{thm}[Minimal blow-up solution in the radial case]\label{mnsr}
There exists a maximal-lifespan solution $(u,v)$ with lifespan $I=(-T^{-}, \infty)$ that is radial, almost periodic modulo symmetries on $[0,\infty)$, and blows up forward in time satisfying $M(u,v)=M_{c,\text{rad}}$, $N(0)=\sup_{t\in [0,\infty)} =1$, 
and $x\equiv \xi \equiv 0$, where $M_{c,\text{rad}}$ is radially symmetric critical mass defined in Section 1.
\end{thm}

The proof of this theorem is parallel to the proof of Theorem \ref{thm:mmbs}. However, we can not use Galilean invariance and so we remove the parameter $\xi_n^j$ in the profile decomposition by using radial assumption. Indeed, For a sequence of radial $L^2$ function, we may refine profile decomposition as follows: 

\begin{prop}[Profile decomposition for radially symmetric sequence]\label{radial-profile}
Consider the case $d\geq 2$. Let $\{(u_n,v_n)\} \subset L^2_{rad} (\mathbb{R}^{d})^2$ be bounded. Then in a profile decomposition given in Theorem \ref{profile}, we can replace all $\xi _n^{j}$ 
and $x_n ^j$ by zero. Furthermore we may take $W_n^{J}$ and $(\phi^{j} ,\psi ^j)$ are radially symmetric.
\end{prop}
\begin{proof}
The proof of this fact is very similar to Theorem 7.3 in \cite{MR2445122} but we give it in Appendix \ref{AppendixB}.
\end{proof}

In Theorem \ref{radial-profile} we may also assume $\theta_{n}^j \equiv 0 $ after modifying the remainder term. Then we can prove Theorem \ref{mnsr} by quite similar argument in the proof of Theorem \ref{thm:mmbs} because we 
do not need to use Galilean invariance.

\subsection{Properties of almost periodic solutions}
\label{sec3.5}
We collect various properties of almost periodic solutions (\mbox{cf.} \cite{MR3098643}, Lemma~5.13--Proposition~5.23).
Proofs will be given in Appendix \ref{AppendixC}.

\begin{lem}\label{lem:localconstancy}
Let $(u,v):I\times \R ^4\to \Bo{C}^2$ be a non-zero solution to \eqref{NLS} that is almost periodic modulo symmetries.
Let $J$ be a subinterval of $I$ such that $S_J(u,v)<\I$.
Then, there exists $C=C(u,v,S_J(u,v))>0$ such that we have
\eqq{\sup _{t\in J}N(t)\le C\inf _{t\in J}N(t).}
\end{lem}

\begin{lem}\label{lem:movement-xi}
Let $(u,v):I\times \R ^4\to \Bo{C}^2$ be a non-zero solution to \eqref{NLS} that is almost periodic modulo symmetries.
Let $J$ be a subinterval of $I$ such that $S_J(u,v)<\I$.
Then, there exists $C=C(u,v,S_J(u,v))>0$ such that we have
\eqq{\sup _{t_1,t_2\in J}|\xi (t_1) -\xi (t_2)|\le C\sup _{t\in J}N(t).}
\end{lem}

\begin{lem}\label{lem:tightness-Str}
Let $(u,v):I\times \R ^4\to \Bo{C}^2$ be a solution to \eqref{NLS} that is almost periodic modulo symmetries.
Let $J$ be a subinterval of $I$ such that $S_J(u,v)<\I$.
Then, for any $\eta >0$ there exists $R=R(u,v,\eta )>0$ such that we have
\eqq{&\int _J\int _{|x-x(t)|\ge \frac{R}{N(t)}}\Big( |u|^3+|v|^3\Big) \,dx\,dt+\norm{P_{|\xi -\xi (t)|\ge RN(t)}u}{L^3(J\times \R^4 )}^3+\norm{P_{|\xi -2\xi (t)|\ge RN(t)}v}{L^3(J\times \R^4 )}^3\\
&\le \eta \big( 1+S_J(u,v)\big) .}
\end{lem}

\begin{lem}\label{lem:Strichartz-SJ}
Let $(u,v):I\times \R ^4\to \Bo{C}^2$ be a non-zero solution to \eqref{NLS} that is almost periodic modulo symmetries.
Let $J$ be a subinterval of $I$.
Then, there exists $C=C(u,v)>0$ such that we have
\eq{Strichartz-SJ}{C^{-1}\int _JN(t)^2\,dt\le S_J(u,v)\le 1+C\int _JN(t)^2\,dt.}
\end{lem}

The following is an immediate consequence of Lemma~\ref{lem:Strichartz-SJ} and its proof.
\begin{cor}\label{lem:chinterval}
Let $(u,v):I\times \R ^4\to \Bo{C}^2$ be a non-zero solution to \eqref{NLS} that is almost periodic modulo symmetries.
Let $J$ be a subinterval of $I$ such that $0<S_J(u,v)<\I$.
Then, there exists $C=C(u,v,S_J(u,v))>0$ such that we have
\eqq{C^{-1}\le \int _JN(t)^2\,dt\le C.}
In particular, from Lemma~\ref{lem:localconstancy} we have
\eqq{\sup _{t\in J}N(t)\sim _{u,v,S_J(u,v)}\int _JN(t)^3\,dt.}
\end{cor}


Now, since the minimal mass blow-up solution $(u,v)$ given in Theorem~\ref{thm:mmbs} or \ref{mnsr} satisfies $(u(t),v(t))\neq 0$ for all $t\ge 0$ and $S_{[0,\I )}(u,v)=\I$, there is a unique sequence $\shugo{t_k}_{k=0}^\I \subset [0,\I)$ such that
\eqq{0=t_0<t_1<t_2<\cdots ,\qquad S_{[t_k,t_{k+1})}(u,v)=1\quad \text{for any}\quad k\ge 0.}
We also have $t_k\to \I$ ($k\to \I$), because for any compact interval $J\subset [0,\I )$ it follows from Lemma~\ref{lem:Strichartz-SJ} and $N(t)\le 1$ for $t\ge 0$ that $S_J(u,v)<\I$.
We call these subintervals $J_k:=[t_k,t_{k+1})$ the \emph{characteristic intervals}.

Then, we see that $N(t)$ and $\xi (t)$ given in Theorem~\ref{thm:mmbs} or \ref{mnsr} can be taken so that
\eq{condition-Nxi}{\left\{ ~\begin{split}
&\text{$N(t)$, $\xi (t)$ are constant on each characteristic interval $J_k$,}\\
&N(0)=\textstyle\sup _{t\ge 0}N(t)=1,\quad \xi (0)=0,\quad N(t)\in \shugo{1,\,C_0^{-1},\,C_0^{-2},\dots}\quad (t\ge 0),\text{\hx and}\\
&N(t_{k+1})\in \shugo{C_0^{-1}N(t_k),\,N(t_k),\,C_0N(t_k)}\quad (k\ge 0)
\end{split}\right.
}
for some $C_0=C_0(u,v)>1$.
In fact, Lemmas~\ref{lem:localconstancy} and \ref{lem:movement-xi} show that \eqref{tightness} still holds if we modify $N(t)$ and $\xi (t)$ on $[0,\I )$ to
\eqq{\bar{N}(t):=\sum _{k=0}^\I C_0^{\big[ \frac{\log N(t_k)}{\log C_0}\big]}\chf{J_k}(t),\qquad \bar{\xi}(t):=\sum _{k=0}^\I \xi (t_k)\chf{J_k}(t),}
for any $C_0>1$.
(Note that $C_0^{-1}N(t_k)<C_0^{\big[ \frac{\log N(t_k)}{\log C_0}\big]}\le N(t_k)$.)
These functions $\bar{N}(t)$, $\bar{\xi} (t)$ also satisfy the same properties as $N(t)$, $\xi(t)$, namely, $\bar{N}(0)=\sup _{t\ge 0}\bar{N}(t)=1$ and $\bar{\xi}(0)=0$.
Moreover, if $C_0=C_0(u,v)$ is sufficiently large, it holds that $\bar{N}(t_{k+1})\in \shugo{C_0^{-1}\bar{N}(t_k),\,\bar{N}(t_k),\,C_0\bar{N}(t_k)}$ for any $k\ge 0$.

Therefore, in what follows we additionally assume \eqref{condition-Nxi}.
These additional properties will be useful later.

Furthermore, in the case $\xi (t)\not\equiv 0$, we see from Lemmas~\ref{lem:movement-xi} and \ref{lem:chinterval} that
\eqq{|\xi (t_k)-\xi (t_{k+1})|\lec _{u,v}N(t_k)\sim _{u,v}\int _{J_k}N(t)^3\,dt.}
Hence, we can take a constant $C_*=C_*(u,v)\gg 1$ such that
\eq{def:Cstar}{|\xi (t_k)-\xi (t_{k+1})|\le 2^{-10}C_*N(t_k),\qquad |\xi (t_k)-\xi (t_{k+1})|\le 2^{-10}C_*\int _{J_k}N(t)^3\,dt}
for any $k\ge 0$.
We set $C_*:=1$ when $\xi (t)\equiv 0$.

\section{Long-time Strichartz estimate}

This section is devoted to the following estimate.
\begin{thm}[Long-time Strichartz estimate]\label{thm:ls}
Let $(u,v)$ be the minimal mass blow-up solution of \eqref{NLS} on a time interval $[0,\I )$ given in Theorem~\ref{thm:mmbs} or \ref{mnsr}.
(Hence, we consider the particular situation that either $\kappa =1/2$ or $\xi (t)\equiv 0$ holds.)
Then, there exists a bounded non-increasing function $\rho :\R _+\to \R _+$ with $\rho (N)\to 0$ as $N\to \I$ such that the following holds:

Let $J\subset [0,\I )$ be an arbitrary interval which is a union of \emph{finite} number of characteristic intervals $\shugo{J_k}$, and let $K:=\int _JN(t)^3\,dt\,(<\I)$.
Then, for any $N\le C_*K$ (with $C_*$ given in \eqref{def:Cstar}) we have
\eq{ls}{\norm{P_{|\xi -\xi (t)|>N}u}{L^2L^4(J\times \R^4 )}+\norm{P_{|\xi -2\xi (t)|>N}v}{L^2L^4(J\times \R^4 )}\le \Big( \frac{K}{N}\Big) ^{1/2}\rho (N).}
\end{thm}

\begin{rem}
The minimality of the mass of $(u,v)$ is not used in the proof of Theorem~\ref{thm:ls}, and in fact we can obtain a similar estimate for general almost periodic solutions satisfying $N(t)\le 1$.
However, when $\xi (t)\not\equiv 0$ we still need to assume $\kappa =1/2$ so that the system has the Galilean invariance.
\end{rem}

\subsection{Non-radial, mass-resonance case}

Let us consider the case of general $\xi (t)$ under the assumption $\kappa =1/2$.
We first derive recursive bounds.
Let $J\subset [0,\I )$ be an interval, and define the functions $\Sc{A}_J,\Sc{S}_J$ on $2^{\Bo{Z}}$ as
\eqq{\Sc{A}_J(N)&:=\norm{P_{|\xi -\xi (t)|>N}u}{L^\I L^2(J\times \R ^4)}+\norm{P_{|\xi -2\xi (t)|>N}v}{L^\I L^2(J\times \R ^4)},\\
\Sc{S}_J(N)&:=\norm{P_{|\xi -\xi (t)|>N}u}{L^2L^4(J\times \R ^4)}+\norm{P_{|\xi -2\xi (t)|>N}v}{L^2L^4(J\times \R ^4)}.}
Note that $\Sc{A}_J(N)\le \Sc{A}_{\R_+}(N)\lec 1$ for any $N$, and $\Sc{A}_{\R _+}(N)$ tends to $0$ as $N\to \I$ by \eqref{tightness} and $N(t)\le 1$, while $\Sc{S}_J(N)$ can be infinite if $J$ contains infinitely many characteristic intervals.  

\begin{lem}\label{lem:ls}
Let $\kappa =1/2$, and let $(u,v)$ be as in Theorem~\ref{thm:ls}.
Then, there exists $C_1=C_1(u,v)>0$ such that the following inequalities hold:
Let $J\subset [0,\I )$ be an arbitrary interval which is a union of (possibly infinite) $\shugo{J_k}$ such that $K:=\int _JN(t)^3\,dt<\I$.

(i) For any $R\ge C_*$ and $0<N\le RK$ such that $\Sc{S}_J(2^{-6}N)<\I$, we have
\eq{ls_induction}{&\Sc{A}_J(N)+\Sc{S}_J(N)\le C_1\de (R)\Sc{S}_J(2^{-6}N)+C_1R^{\frac{5}{2}}\Big( \frac{K}{N}\Big) ^{\frac{1}{2}}\Big( \chf{(0,R)}(N)+\Sc{A}_J(2^{-8}N)^{\frac{1}{3}}\Big) .
}

(ii) For any $R\ge C_*$, $N\ge \max \{ RK,\,R\}$ such that $\Sc{S}_J(2^{-8}N)<\I$, we have
\eq{ls_induction'}{\Sc{A}_J(N)+\Sc{S}_J(N)&\le 
\inf _k\norm{P_{|\xi -\xi (t_k)|>N/4}u(t_k)}{L^2(\R ^4)}+\inf _k\norm{P_{|\xi -2\xi (t_k)|>N/4}v(t_k)}{L^2(\R ^4)} \\
&\hx +C_1\de (R)\Sc{S}_J(2^{-6}N)+C_1R^{\frac{5}{2}}\Big( \frac{K}{N}\Big) ^{\frac{1}{2}}\Big( \Sc{A}_J(2^{-8}N)+\Sc{S}_J(2^{-8}N)\Big) .
}
Here, 
\eqq{\de (R):=&\norm{P_{|\xi -\xi (t)|>2^{-10}RN(t)}u}{L^{\I}L^2(\R _+\times \R^4 )}+\norm{P_{|\xi -2\xi (t)|>2^{-10}RN(t)}v}{L^{\I}L^2(\R _+\times \R^4 )}\\
&+\norm{u}{L^{\I}L^2(\shugo{t\in \R _+,\,|x-x(t)|>R/N(t)})}+\norm{v}{L^{\I}L^2(\shugo{t\in \R _+,\,|x-x(t)|>R/N(t)})}.}
In particular, $\de (R)\to 0$ as $R\to \I$ by \eqref{tightness}.
\end{lem}

\begin{proof}
Since the time-dependent projection operator $P_{|\xi -\xi (t)|>N}$ does not commute with $i\p_t+\Delta$, we need to freeze the frequency center $\xi(t)$ before applying the Strichartz estimate.

(i)
Note that the both sides of \eqref{ls_induction} are finite by the assumption $\Sc{S}_J(2^{-6}N)<\I$.
We make a special decomposition of $J=\cup _{k}J_k$.
Let $\shugo{B_j}$ be the collection of the characteristic intervals $J_k$ for which $N(t_k)>N/R$.
Since
\eqq{\sum _kN(t_k)\sim \sum _k\int _{J_k}N(t)^3\,dt=K}
by Lemma~\ref{lem:chinterval}, the number of $B_j$ is at most $O(RK/N)$.
Also note that such a characteristic interval does not exist if $N\ge R$.

On each $B_j$, we estimate the left hand side of \eqref{ls_induction} crudely by $O(1)$ using the Duhamel formula and the Strichartz estimates.
Then, the contribution from $\cup _jB_j$ is 
\eqq{&\Sc{A}_{\cup _jB_j}(N)+\Sc{S}_{\cup _jB_j}(N)\lec (\# B_j)^{\frac{1}{2}}\lec R^{\frac{1}{2}}\Big(\frac{K}{N}\Big) ^{\frac{1}{2}}\chf{(0,R)}(N).}

We next find a decomposition of $J\setminus \cup _jB_j$ into mutually disjoint intervals $\shugo{G_l}$, each of which is a union of characteristic intervals, such that for each $G_l$ it holds that
\eqq{\frac{N}{R}<\sum _{k;\,J_k\subset G_l}N(t_k)\le \frac{2N}{R},}
or that 
\eqq{\sum _{k;\,J_k\subset G_l}N(t_k)\le \frac{N}{R},\qquad \sup G_l=\inf B_j\hx \text{for some $B_j$\hx or}\hx \sup G_l=\sup J.} 
This is possible because $N(t_k)\le N/R$ for all $k$ such that $J_k\subset J\setminus \cup _jB_j$.
Since $N\le RK$, we have
\eqq{\# G_l\le C\frac{RK}{N}+\# B_j +1\lec \frac{RK}{N}.}
\eqref{def:Cstar} and the assumption $R\ge C_*$ imply that if $t,t^*\in G_l$, 
\eqq{|\xi (t) -\xi (t^*)|\le 2^{-10}C_*\sum _{k;\,J_k\subset G_l}N(t_k)\le 2^{-9}N.}
(Recall that $\xi (t)\equiv \xi (t_k)$ on $J_k$.)
We thus have
\eqq{\shugo{|\xi -\xi (t)|\ge N}\subset \shugo{|\xi -\xi (t^*)|\ge N/2}}
for any $t,t^*\in G_l$, which implies that $P_{|\xi -\xi (t)|>N}=P_{|\xi -\xi (t)|>N}P_{|\xi -\xi (t^*)|>N/4}$.
In the same manner, we have $P_{|\xi -\xi (t^*)|>N/4}=P_{|\xi -\xi (t^*)|>N/4}P_{|\xi -\xi (t)|>2^{-4}N}$.

Let us focus on the estimate for $u$; the argument for $v$ is analogous.
For each $G_l$, using the Duhamel formula and the Strichartz estimates, we have
\eqq{&\norm{P_{|\xi -\xi (t)|>N}u}{L^\I L^2\cap L^2L^4(G_l\times \R^4 )}\lec \norm{P_{|\xi -\xi (t^*_l)|>N/4}u}{L^\I L^2\cap L^2L^4(G_l\times \R^4 )}\\
&\lec \norm{P_{|\xi -\xi (t^*_l)|>N/4}u(t^*_l)}{L^2(\R^4)}+\norm{P_{|\xi -\xi (t^*_l)|>N/4}(\bbar{u}v)}{L^2L^{4/3}(G_l\times \R^4 )}\\
&\lec \norm{P_{|\xi -\xi (t^*_l)|>N/4}u(t^*_l)}{L^2(\R^4)}+\norm{P_{|\xi -\xi (t)|>2^{-4}N}(\bbar{u}v)}{L^2L^{4/3}(G_l\times \R^4 )},}
where $t^*_l$ is an arbitrary point in $G_l$ and the implicit constant does not depend on $t^*_l$.
By square-summing the above estimate over $G_l$'s and applying the bound on $\# G_l$, we obtain
\eqq{&\norm{P_{|\xi -\xi (t)|>N}u}{L^\I L^2\cap L^2L^4(\cup _lG_l\times \R^4 )}\\
&\lec R^{\frac{1}{2}}\Big(\frac{K}{N}\Big) ^{\frac{1}{2}}\Sc{A}_J(2^{-2}N)+\Big( \sum _{k;\,J_k\subset J\setminus \cup _jB_j}\norm{P_{|\xi -\xi (t_k)|>2^{-4}N}(\bbar{u}v)}{L^2L^{4/3}(J_k\times \R^4 )}^2\Big) ^{\frac{1}{2}}.}
Note that $\Sc{A}_J(2^{-2}N)\lec \Sc{A}_J(2^{-8}N)^{1/3}$, since $\Sc{A}_J$ is bounded and non-increasing in $N$.

All we have to do is the estimate of $\tnorm{P_{|\xi -\xi (t_k)|>2^{-4}N}(\bbar{u}v)}{L^2L^{4/3}(J_k\times \R^4 )}$ for each $J_k\subset J\setminus \cup _jB_j$, on which $N(t)\equiv N(t_k)\le N/R$.
The following properties of the Galilean transforms are easily verified.
\begin{lem}
For a smooth function $\phi$ on $\R^4$, let $P_\phi$ be the Fourier multiplier defined by $P_\phi f:=\F^{-1}_\xi\phi \F _xf$.
Also define the shift operator $\tau (\xi _0)$ and the Galilean transforms $\Sc{G}^u_{\xi _0},\Sc{G}^v_{\xi _0}$ for $\xi _0\in \R^4$ by 
\eqs{[\tau (\xi _0)f](\xi ):=f(\xi +\xi _0),\\
\big[ \Sc{G}^u_{\xi _0}u\big] (t,x):=e^{ix\cdot \xi _0}e^{-it|\xi _0|^2}u(t,x-2\xi _0t),\quad
\big[ \Sc{G}^v_{\xi _0}v\big] (t,x):=e^{2ix\cdot \xi _0}e^{-2it|\xi _0|^2}v(t,x-2\xi _0t).
}
Then, the following holds for $x_0,\xi _0\in \R^4$.
\begin{enumerate}
\item $P_\phi e^{ix\cdot \xi _0}f=e^{ix\cdot \xi _0}P_{\tau (\xi _0)\phi}f$,\hx $P_\phi \tau (x_0)f=\tau (x_0)P_\phi f$.\\[-5pt]
\item $\tnorm{P_\phi \Sc{G}^u_{\xi _0}u}{}=\tnorm{P_{\tau (\xi _0)\phi}u}{}$, $\tnorm{P_\phi \Sc{G}^v_{\xi _0}u}{}=\tnorm{P_{\tau (2\xi _0)\phi}v}{}$\\
for the $L^pL^q(I\times \R^4)$-norm with any $1\le p,q\le \I$ and $I\subset \R$.\\[-5pt]
\item $\Sc{G}^u_{\xi _0}(\bbar{u}v)=\bbar{\Sc{G}^u_{\xi _0}u}\Sc{G}^v_{\xi _0}v$,\hx $\Sc{G}^v_{\xi _0}(u^2)=(\Sc{G}^u_{\xi _0}u)^2$.\\[-5pt]
\item $(i\p _t+\Delta )\Sc{G}^u_{\xi _0}u=\bbar{\Sc{G}^u_{\xi _0}u}\Sc{G}^v_{\xi _0}v$,\hx $(i\p _t+\frac{1}{2}\Delta )\Sc{G}^v_{\xi _0}v=(\Sc{G}^u_{\xi _0}u)^2$\hx if $(u,v)$ solves \eqref{NLS} with $\kappa =1/2$.
\end{enumerate}
\end{lem}

Hence, we evaluate
\begin{align}
&\norm{P_{|\xi -\xi (t_k)|>2^{-4}N}(\bbar{u}v)}{L^2L^{4/3}(J_k\times \R^4 )}=\norm{P_{>2^{-4}N}(\bbar{\Sc{G}^u_{-\xi (t_k)}u}\Sc{G}^v_{-\xi (t_k)}v)}{L^2L^{4/3}(J_k\times \R^4 )}\notag \\
&\lec \norm{\bbar{P_{>2^{-6}N}\Sc{G}^u_{-\xi (t_k)}u}P_{>2^{-6}N}\Sc{G}^v_{-\xi (t_k)}v}{L^2L^{4/3}(J_k\times \R^4 )}\label{ls1}\\
&\hx +\norm{\bbar{P_{>2^{-6}N}\Sc{G}^u_{-\xi (t_k)}u}\Sc{G}^v_{-\xi (t_k)}v}{L^2L^{4/3}(J_k\times \R^4 )}\label{ls2}\\
&\hx +\norm{\bbar{\Sc{G}^u_{-\xi (t_k)}u}P_{>2^{-6}N}\Sc{G}^v_{-\xi (t_k)}v}{L^2L^{4/3}(J_k\times \R^4 )}.\label{ls3}
\end{align}
In the last inequality we have used the identities
\eqs{\bbar{u}v=\bbar{P_{\le 2^{-6}N}u}P_{\le 2^{-6}N}v-\bbar{P_{>2^{-6}N}u}P_{>2^{-6}N}v+\bbar{P_{>2^{-6}N}u}v+\bbar{u}P_{>2^{-6}N}v,\\
P_{>2^{-4}N}(\bbar{P_{\le 2^{-6}N}u}P_{\le 2^{-6}N}v)=0,}
and the Bernstein inequality.

Since $N\ge RN(t_k)$, we have $P_{>2^{-6}N}=P_{>2^{-6}N}P_{>2^{-7}RN(t_k)}$ and obtain a bound on \eqref{ls1} as
\eqq{\eqref{ls1}&\lec \norm{P_{>2^{-7}RN(t_k)}\Sc{G}^u_{-\xi (t_k)}u}{L^\I L^2(J_k\times \R ^4)}\norm{P_{>2^{-6}N}\Sc{G}^v_{-\xi (t_k)}v}{L^2L^4(J_k\times \R^4 )}\\
&=\norm{P_{|\xi -\xi (t_k)|>2^{-7}RN(t_k)}u}{L^\I L^2(J_k\times \R ^4)}\norm{P_{|\xi -2\xi (t_k)|>2^{-6}N}v}{L^2L^4(J_k\times \R^4 )}.}
Then, summing up over $J_k$'s, we obtain the total contribution from \eqref{ls1} as $C\de (R)\Sc{S}_J(2^{-6}N)$.

We make further decomposition for \eqref{ls2} as
\begin{align}
\eqref{ls2}&\le \norm{\bbar{P_{>2^{-6}N}\Sc{G}^u_{-\xi (t_k)}u}\Sc{G}^v_{-\xi (t_k)}v}{L^2L^{4/3}(\shugo{t\in J_k,\,|x+2\xi (t_k)t-x(t)|>R/N(t_k)})}\label{ls4}\\
&\hx +\norm{\bbar{P_{>2^{-6}N}\Sc{G}^u_{-\xi (t_k)}u}P_{>2^{-9}N}\Sc{G}^v_{-\xi (t_k)}v}{L^2L^{4/3}(\shugo{t\in J_k,\,|x+2\xi (t_k)t-x(t)|\le R/N(t_k)})}\label{ls5}\\
&\hx +\norm{\bbar{P_{>2^{-6}N}\Sc{G}^u_{-\xi (t_k)}u}P_{\le 2^{-9}N}\Sc{G}^v_{-\xi (t_k)}v}{L^2L^{4/3}(\shugo{t\in J_k,\,|x+2\xi (t_k)t-x(t)|\le R/N(t_k)})}.\label{ls6}
\end{align}
For \eqref{ls4} and \eqref{ls5}, we estimate as
\eqq{\norm{P_{|\xi -\xi (t_k)|>2^{-6}N}u}{L^2L^4(J_k\times \R^4 )}\Big( \norm{v}{L^\I L^2(\shugo{t\in J_k,\,|x-x(t)|>R/N(t_k)})}+\norm{P_{|\xi -2\xi (t_k)|>2^{-10}RN(t_k)}v}{L^\I L^2(J_k\times \R^4 )}\Big) .}
We sum up this bound over $J_k$'s and obtain $C\de (R)\Sc{S}_J(2^{-6}N)$.
Finally, we apply the bilinear Strichartz estimate (Lemma~\ref{lem:bs}) to \eqref{ls6}.%
\footnote{The Fourier supports of $P_{|\xi -\xi (t_k)|>
2^{-6}N}u(t)$ and $P_{|\xi -2\xi (t_k)|\le 2^{-9}N}v(t)$ are not necessarily separated when $\xi (t_k)\neq 0$.
To apply the bilinear Strichartz estimate, we exploit the Galilean transforms to adjust the frequency centers to the origin.
This is the only part where we essentially use the Galilean invariance of the system in the proof of Theorem~\ref{thm:ls}.
}
We have
\eqq{\eqref{ls6}&\lec \frac{R}{N(t_k)}\norm{\bbar{P_{>2^{-6}N}\Sc{G}^u_{-\xi (t_k)}u}P_{\le 2^{-9}N}\Sc{G}^v_{-\xi (t_k)}v}{L^2(J_k\times \R ^4)}\\
&\lec \frac{R}{N(t_k)}\frac{(RN(t_k))^{3/2}}{N^{1/2}}\norm{P_{>2^{-6}N}\Sc{G}^u_{-\xi (t_k)}u}{S^0_u(J_k\times \R ^4)}\norm{P_{\le 2^{-9}N}\Sc{G}^v_{-\xi (t_k)}v}{S^0_v(J_k\times \R ^4)},}
where $\tnorm{u}{S^0_u(J_k\times \R ^4)}:=\tnorm{u(t_k)}{L^2(\R ^4)}+\tnorm{(i\p _t+\Delta )u}{L^{3/2}(J_k\times \R ^4)}$ and $\tnorm{v}{S^0_v(J_k\times \R ^4)}:=\tnorm{v(t_k)}{L^2(\R ^4)}+\tnorm{(i\p _t+\frac{1}{2}\Delta )v}{L^{3/2}(J_k\times \R ^4)}$.
Since $(i\p _t+\frac{1}{2}\Delta )P_{\le 2^{-9}N}\Sc{G}^v_{-\xi (t_k)}v=P_{\le 2^{-9}N}(\Sc{G}^u_{-\xi (t_k)}u)^2$, the last norm is bounded by
\eqq{&\norm{P_{\le 2^{-9}N}\big[ \Sc{G}^v_{-\xi (t_k)}v\big] (t_k,\cdot )}{L^2(\R^4)}+\norm{P_{\le 2^{-9}N}(\Sc{G}^u_{-\xi (t_k)}u)^2}{L^{3/2}(J_k\times \R^4)}\\
&\lec \norm{v(t_k)}{L^2(\R^4)}+\norm{u}{L^3(J_k\times \R^4)}^2\lec 1.}
Similarly, we have
\eqq{&\norm{P_{>2^{-6}N}\Sc{G}^u_{-\xi (t_k)}u}{S^0_u(J_k\times \R ^4)}\\
&=\norm{P_{|\xi -\xi (t_k)|>2^{-6}N}u(t_k)}{L^2(\R^4)}+\norm{P_{>2^{-6}N}(\bbar{\Sc{G}^u_{-\xi (t_k)}u}\Sc{G}^v_{-\xi (t_k)}v)}{L^{3/2}(J_k\times \R ^4)}.}
For the second term on the right-hand side, we make a similar decomposition as \eqref{ls1}--\eqref{ls3}.
Applying the H\"older inequality, interpolation, and that $\Sc{S}_{J_k}(M)\lec 1$ for any $M>0$, we have
\eqq{&\norm{P_{>2^{-6}N}(\bbar{\Sc{G}^u_{-\xi (t_k)}u}\Sc{G}^v_{-\xi (t_k)}v)}{L^{3/2}(J_k\times \R ^4)}\\
&\lec \norm{\bbar{P_{>2^{-8}N}\Sc{G}^u_{-\xi (t_k)}u}P_{>2^{-8}N}\Sc{G}^v_{-\xi (t_k)}v}{L^{3/2}(J_k\times \R ^4)}\\
&\hx +\norm{\bbar{P_{>2^{-8}N}\Sc{G}^u_{-\xi (t_k)}u}\Sc{G}^v_{-\xi (t_k)}v}{L^{3/2}(J_k\times \R ^4)}+\norm{\bbar{\Sc{G}^u_{-\xi (t_k)}u}P_{>2^{-8}N}\Sc{G}^v_{-\xi (t_k)}v}{L^{3/2}(J_k\times \R ^4)}\\
&\lec \Big( \norm{P_{|\xi -\xi (t)|>2^{-8}N}u}{L^3(J_k\times \R ^4)}+\norm{P_{|\xi -2\xi (t)|>2^{-8}N}v}{L^3(J_k\times \R ^4)}\Big) \Big( \norm{u}{L^3(J_k\times \R ^4)}+\norm{v}{L^3(J_k\times \R ^4)}\Big) \\
&\lec \norm{P_{|\xi -\xi (t)|>2^{-8}N}u}{L^3(J_k\times \R ^4)}+\norm{P_{|\xi -2\xi (t)|>2^{-8}N}v}{L^3(J_k\times \R ^4)}\\
&\lec \Sc{A}_{J_k}(2^{-8}N)^{\frac{1}{3}}\Sc{S}_{J_k}(2^{-8}N)^{\frac{2}{3}}\lec \Sc{A}_{J_k}(2^{-8}N)^{\frac{1}{3}}.}
Thus, the total contribution from \eqref{ls6} is bounded by
\eqq{&\frac{R^{5/2}}{N^{1/2}}\Big( \sum _{k}N(t_k)\Big) ^{\frac{1}{2}}\Big[ \Sc{A}_J(2^{-6}N)+\Sc{A}_{J}(2^{-8}N)^{\frac{1}{3}}\Big] 
\lec R^{\frac{5}{2}}\Big( \frac{K}{N}\Big) ^{\frac{1}{2}}\Sc{A}_{J}(2^{-8}N)^{\frac{1}{3}}.}
This completes the estimate for \eqref{ls2}.
\eqref{ls3} can be treated in a similar manner, and we have \eqref{ls_induction}.

(ii)
Since $N\ge R$ and $N(t)\le 1$, we have no $B_j$.
Moreover, from \eqref{def:Cstar} and $N\ge C_*K$ we have
\eqq{|\xi (t)-\xi (t^*)|\le 2^{-10}C_*\sum _k\int _{J_k}N(t)^3\,dt =2^{-10}C_*K\le 2^{-10}N}
for any $t,t^*\in J$.
Similarly to the estimate on $G_l$ in (i), we show
\eqq{&\norm{P_{|\xi -\xi (t)|>N}u}{L^\I L^2\cap L^2L^4(J\times \R^4 )}\\
&\lec \inf _{k}\norm{P_{|\xi -\xi (t_k)|>N/4}u(t_k)}{L^2(\R ^4)}+\Big( \sum _{k;\,J_k\subset J}\norm{P_{|\xi -\xi (t_k)|>2^{-4}N}(\bbar{u}v)}{L^2L^{4/3}(J_k\times \R^4 )}^2\Big) ^{\frac{1}{2}}.}
The estimate on the nonlinear term is the same as before, except that at the last step we use
\eqq{\Sc{A}_{J_k}(2^{-8}N)^{\frac{1}{3}}\Sc{S}_{J_k}(2^{-8}N)^{\frac{2}{3}}\le \Sc{A}_{J_k}(2^{-8}N)+\Sc{S}_{J_k}(2^{-8}N).}
This implies \eqref{ls_induction'}.
\end{proof}

We are now ready to prove Theorem~\ref{thm:ls} in the case of $\kappa =1/2$.
\begin{proof}[Proof of Theorem~\ref{thm:ls} ($\kappa =1/2$)]
The proof will be done via an induction on $N$.
Note that $\Sc{S}_J(N)<\I$ for any $N$, since $J$ consists of finitely many $J_k$'s.

We begin with the base case.
We always have a crude bound $\Sc{S}_J(N)\lec (\# J_k)^{1/2}$, which can be seen by applying the Duhamel formula and the Strichartz estimates on each $J_k$.
This bound is acceptable as long as $N$ is so small that $\# J_k\le K/N$.
Hence, there exists $C_2=C_2(u,v)>0$ such that
\eq{ls_basecase}{\Sc{S}_J(N)\le C_2\Big( \frac{K}{N}\Big) ^{\frac{1}{2}},\qquad N\le \frac{K}{\# J_k}.}

With the base case in mind, we impose the following condition on the function $\rho$:
\eq{ls7}{\rho (N)\ge C_2,\qquad N\le \frac{K}{\# J_k}.}
On the other hand, if we had \eqref{ls}, then the formula \eqref{ls_induction} would imply that
\eqs{\Sc{S}_J(N)\le \Big[ 8C_1\de (R)\rho (2^{-6}N) +C_1R^{\frac{5}{2}}\ti{\rho}(N)\Big] \Big( \frac{K}{N}\Big) ^{\frac{1}{2}},\qquad \ti{\rho}(N):=\chf{(0,R)}(N)+\Sc{A}_{\R_+}(2^{-8}N)^{\frac{1}{3}}.}
(Note that $\ti{\rho}(N)$ is non-increasing, $\sup\limits _{N>0}\ti{\rho}(N)\lec 1$ and $\lim\limits _{N\to \I}\ti{\rho}(N)=0$.)
We thus need to define $\rho (N)$ so that
\eq{ls8}{\rho (N)\ge 8C_1\de (R)\rho (2^{-6}N) +C_1R^{\frac{5}{2}}\ti{\rho}(N),\qquad N>0.}

We first fix $R\ge C_*$ sufficiently large so that $8C_1\de (R)\le \frac{1}{2}$, and then define the function $\rho _{N_*}$ for $N_*\in 2^{\Bo{Z}}$ by 
\eqs{\rho _{N_*}(N):=\begin{cases}
\max \Big\{ 2C_1R^{\frac{5}{2}}\sup\limits _{M>0}\ti{\rho}(M) ,\, C_2\Big\}, &N\le N_*,\\[10pt]
C_1R^{\frac{5}{2}}\ti{\rho}(N) +\frac{1}{2}\rho _{N_*} (2^{-6}N), &N>N_*\hx \text{(recursively).}
\end{cases}
}
It is easy to verify that $\rho _{N_*}$ is bounded uniformly in $N_*$, and that \eqref{ls7}, \eqref{ls8} hold if $N_*\ge K/\#J_k$.  
Moreover, $\rho _{N_*}$ is non-increasing and thus has a limit $\rho _{N_*}(\I )\ge 0$ as $N\to \I$.
Letting $N\to \I$ in the above recursive formula, we have $\rho _{N_*} (\I )=\frac{1}{2}\rho _{N_*} (\I )$, concluding $\rho _{N_*} (\I )=0$.

Now, since $K=\int _JN(t)^3\,dt\le \int _JN(t)^2\,dt \lec S_J(u,v)=\# J_k$ by Lemma~\ref{lem:Strichartz-SJ}, the quantity $K/\# J_k$ has an upper bound $N_0$ which is independent of $J$.
Therefore, we finally define $\rho :=\rho _{N_0}$.
The claimed estimate \eqref{ls} can be shown by an induction on $N$, with the base case \eqref{ls_basecase} and the recursive formula \eqref{ls_induction}, noticing \eqref{ls7} and \eqref{ls8}.

This completes the proof.
\end{proof}

\subsection{Radial case}

Here, we consider the case $\xi (t)\equiv 0$.
Since we do not need the Galilean invariance, the same argument as above can be applied for any $\kappa >0$.
Moreover, it turns out that we do not have to consider the intervals $\{ G_l\}$.
Note that the estimate of \eqref{ls5}--\eqref{ls6} is slightly different; we decompose $v$ as $P_{>\e N}v+P_{\le \e N}v$ with $\e>0$ sufficiently small depending on $\kappa$, so that we can apply Lemma~\ref{lem:bs}.
As a result, we show the following:
\begin{lem}\label{lem:lsradial}
Let $(u,v)$ be as in Theorem~\ref{thm:ls}, and assume that $\xi (t) \equiv 0$.
Then, there exists $C_1=C_1(u,v)>0$ such that the following inequalities hold:
Let $J\subset [0,\I )$ be an arbitrary interval which is a union of (possibly infinite) $\shugo{J_k}$ such that $K:=\int _JN(t)^3\,dt<\I$.

(i) For any $R>0$ and $0<N\le K$ such that $\Sc{S}_J(2^{-2}N)<\I$, we have
\eq{ls_inductionrad}{&\Sc{A}_J(N)+\Sc{S}_J(N)\le C_1\de (R)\Sc{S}_J(2^{-2}N)+C_1R^{\frac{5}{2}}\Big( \frac{K}{N}\Big) ^{\frac{1}{2}}\Big( \chf{(0,R)}(N)+\Sc{A}_J(2^{-4}N)^{\frac{1}{3}}\Big) .
}

(ii) For any $R>0$, $N\ge \max \{ K,\,R\}$ such that $\Sc{S}_J(2^{-4}N)<\I$, we have
\eq{ls_induction'rad}{\Sc{A}_J(N)+\Sc{S}_J(N)&\le 
\inf _k\norm{P_{>N}(u,v)(t_k)}{L^2(\R ^4)} \\
&\hx +C_1\de (R)\Sc{S}_J(2^{-2}N)+C_1R^{\frac{5}{2}}\Big( \frac{K}{N}\Big) ^{\frac{1}{2}}\Big( \Sc{A}_J(2^{-4}N)+\Sc{S}_J(2^{-4}N)\Big) .
}
Here, 
\eqq{\de (R):=&\norm{P_{>\e RN(t)}(u,v)}{L^{\I}L^2(\R _+\times \R^4 )}+\norm{(u,v)}{L^{\I}L^2(\shugo{t\in \R _+,\,|x-x(t)|>R/N(t)})}}
with $0<\e =\e (\kappa )\ll 1$.
\end{lem}
Then, Theorem~\ref{thm:ls} can be shown by the same argument as before using Lemma~\ref{lem:lsradial} (i).

\section{Additional regularity: Rapid frequency cascade scenario}

For the minimal mass blow-up solution $(u,v)$ given in Theorem~\ref{thm:mmbs} or \ref{mnsr}, the following two scenarios are possible:

\smallskip
$\bullet$ $\dint _0^\I N^3(t)\,dt<\I$\hx (Rapid frequency cascade scenario),

\smallskip
$\bullet$ $\dint _0^\I N^3(t)\,dt=\I$\hx (Quasi-soliton scenario).

\smallskip
In this section we shall derive additional regularity for the former case from Theorem~\ref{thm:ls}, Lemma~\ref{lem:ls}~(ii) or Lemma~\ref{lem:lsradial} (ii) and use it to preclude this scenario.

\begin{thm}\label{thm:ar}
In the rapid frequency cascade scenario, the solution $(u,v)$ is in $L^\I (\R _+;H^s(\R^4)^2)$ for any $s>0$ and, with $K:=\int _0^\I N(t)^3\,dt<\I$, satisfies
\eqq{\norm{u}{L^\I (\R _+;H^s(\R^4 ))}+\norm{v}{L^\I (\R _+;H^s(\R^4 ))}\lec _{s}\LR{K}^{s}.}
\end{thm}

\begin{proof}
Let us focus on the case $\kappa =1/2$; for the case $\xi (t)\equiv 0$ the claim is shown similarly by means of Lemma~\ref{lem:lsradial} (ii) instead of Lemma~\ref{lem:ls} (ii).
We continue to use the notation in the preceding section.

In order to prove Theorem~\ref{thm:ar} it suffices to show that, for any $s>0$, there exist $C_s>0$, $C_s'\ge C_*$ depending on $u,v,s$ such that 
\eq{ar-finalclaim}{\Sc{A}(N)\le C_s\Big( \frac{K}{N}\Big) ^s,\qquad N\ge C_s'K.}
In fact, if we have \eqref{ar-finalclaim}, then $\tnorm{P_{>4N}u}{L^\I L^2(\R _+\times \R ^4)}+\norm{P_{>4N}v}{L^\I L^2(\R _+\times \R ^4)}\lec C_s(K/N)^s$ for $N\ge C_s'K$, since \eqref{def:Cstar} implies that for any $t\ge 0$
\eq{bd:xi}{|\xi (t)|=|\xi (t)-\xi (0)|\le 2^{-10}C_*\sum _{k;\, J_k\subset [0,t)}\int _{J_k}N(t)^3\,dt\le 2^{-10}C_*K\le 2^{-10}N.}
Therefore, we have
\eqq{\norm{u}{L^\I (\R _+;H^s(\R^4 ))}&\le \norm{P_{\le 4C_{2s}'K}u}{L^\I (\R _+;H^s(\R^4 ))}+\sum _{N\ge C_{2s}'K}\norm{P_{8N}u}{L^\I (\R _+;H^s(\R^4 ))}\\
&\lec _s\LR{K}^s\norm{u}{L^\I L^2(\R _+\times  \R^4 )}+\sum _{N\ge C_{2s}'K}\LR{N}^s\Big( \frac{K}{N}\Big) ^{2s}\lec _s\LR{K}^{s},}
and similarly for $v$, proving Theorem~\ref{thm:ar}.

To show \eqref{ar-finalclaim}, we first observe that $\Sc{A}_{\R_+}(C_*K)+\Sc{S}_{\R_+}(C_*K)\lec 1$.
In fact, the monotone convergence theorem reduces to showing it for any compact interval $J\subset \R _+$, which follows from Theorem~\ref{thm:ls}.
(We do not exploit the decaying factor $\rho(N)$ here.)
In particular, we have $\Sc{A}_{\R_+}(N)+\Sc{S}_{\R_+}(N)\lec 1$ for any $N\ge C_*K$.

Next, we set $R\ge C_*$ and $C_3\ge \max \{ R,RK^{-1}\}$ sufficiently large (according to $s$) so that
\eqq{2^{6s}C_1\delta (R)\le \frac{1}{2},\quad 2^{8s}C_1R^{\frac{5}{2}}\Big( \frac{K}{2^8C_3K}\Big) ^{\frac{1}{2}}\le \frac{1}{2}.}
Since $\inf\limits _kN(t_k)=0$ by $\int _0^\I N(t)^3\,dt<\I$, it holds
\eqq{\inf _{k}\norm{P_{|\xi -\xi (t_k)|>N}u(t_k)}{L^2(\R ^4)}=\inf _{k}\norm{P_{|\xi -\xi (t_k)|>\frac{N}{N(t_k)}N(t_k)}u(t_k)}{L^2(\R ^4)}=0}
for any $N$ by the almost orthogonality \eqref{tightness}, and similarly for $v$.
From Lemma~\ref{lem:ls}~(ii), we have
\eqq{\Sc{A}_{\R_+}(N)+\Sc{S}_{\R_+}(N)\le \frac{1}{2}\Big\{ 2^{-6s}\Sc{S}_{\R_+}(2^{-6}N)+2^{-8s}\Big( \Sc{A}_{\R_+}(2^{-8}N)+\Sc{S}_{\R_+}(2^{-8}N)\Big) \Big\} ,\quad N\ge 2^8C_3K.}

Now, we choose $C_4=C_4(u,v,s)>0$ large enough so that 
\eqq{\Sc{A}_{\R_+}(2^jC_3K)+\Sc{S}_{\R_+}(2^jC_3K)\le 2^{-sj}C_4,\qquad j=0,1,\dots ,7.}
Then, it is easy to show by induction that \eqref{ar-finalclaim} holds with $C_s=C_3^sC_4$ and $C_s'=C_3$.
\end{proof}

Suppose that there exists a minimal mass blow-up solution $(u,v)$ in the rapid frequency cascade scenario.
Since $\int _0^\I N^3(t)\,dt<\I$, there is a sequence of time $\shugo{t_n}$ such that $N(t_n)\to 0$ as $n\to \I$.

We first consider the radial case: $\xi (t)\equiv 0$ and $\kappa >0$ is arbitrary.
From the almost periodicity of $(u,v)$, we see that
\eqq{\sup _n\norm{P_{>RN(t_n)}(u,v)(t_n)}{L^2}=o(1)\quad (R\to \I ).}
By Theorem~\ref{thm:ar} with $s=2$ (in fact $s=1+\e$ is sufficient) we obtain
\eqq{\limsup _{n\to \I}\norm{(u,v)(t_n)}{\dot{H}^1}&\le \limsup _{n\to \I}\Big( \norm{P_{>RN(t_n)}(u,v)(t_n)}{\dot{H}^1}+\norm{P_{\le RN(t_n)}(u,v)(t_n)}{\dot{H}^1}\Big) \\
&\lec \limsup _{n\to \I}\Big( \norm{(u,v)(t_n)}{\dot{H}^2}^{1/2}\norm{P_{>RN(t_n)}(u,v)(t_n)}{L^2}^{1/2}+RN(t_n)\Big) \\
&= \LR{K}o(1)\quad (R\to \I ),}
which implies that $\tnorm{(u,v)(t_n)}{\dot{H}^1}\to 0$ as $n\to \I$.
Using the Gagliardo-Nirenberg inequality we have $E((u,v)(t_n))\to 0$ as $n\to \I$, and then $E((u,v)(0))=0$ by the energy conservation.
This combined with $M(u,v)=M_{c,\text{rad}}<M(\phi,\psi )$ and the sharp Gagliardo-Nirenberg inequality (Lemma~\ref{lem:GN} (ii)) shows that $\tnorm{(u,v)(0)}{\dot{H}^1}=0$.
By the Sobolev embedding we conclude that $\tnorm{(u,v)(0)}{L^4}=0$, which clearly contradicts the fact that $(u,v)$ is a non-zero solution.

For the case of general $\xi (t)$ and $\kappa =1/2$, we apply the Galilean transformation and argue with $(u_n,v_n):=\Sc{T}_{g(0,0,-\xi (t_n),1)}(u,v)$ instead of $(u,v)$.
Then, we have
\eqq{\sup _n\norm{P_{>RN(t_n)}(u_n,v_n)(t_n)}{L^2}=o(1)\quad (R\to \I )}
and
\eqq{\norm{(u_n,v_n)(t_n)}{\dot{H}^2}=\norm{|\cdot |^2\big( \hat{u}(t_n,\cdot +\xi (t_n)),\hat{v}(t_n, \cdot +2\xi (t_n)\big)}{L^2}\lec \LR{K}^2}
by Theorem~\ref{thm:ar} together with \eqref{bd:xi}, and hence the above argument implies $\tnorm{(u_n,v_n)(t_n)}{\dot{H}^1}\to 0$ as $n\to \I$, which yields $E(u_n(0),v_n(0))\to 0$.
Since $M(u_n,v_n)=M(u,v)<M(\phi,\psi )$, the sharp Gagliardo-Nirenberg inequality shows that $\tnorm{(u_n,v_n)(0)}{\dot{H}^1}\to 0$, and then $\tnorm{(u_n,v_n)(0)}{L^4}\to 0$ as $n\to \I$.
But now $\tnorm{(u_n,v_n)(0)}{L^4}=\tnorm{(u,v)(0)}{L^4}$, and we have the same conclusion $\tnorm{(u,v)(0)}{L^4}=0$.

Therefore, the rapid frequency cascade scenario is impossible.

\section{Virial argument: Quasi-soliton scenario}

Now we consider the scenario $\int _0^\I N(t)^3\,dt =\I$.

\subsection{Radial case}

Let $(u,v)$ be the minimal mass blow-up solution $(u,v)$ given in Theorem~\ref{mnsr}.
We recall here that $x(t)=\xi (t)\equiv 0$, while $\kappa >0$ is arbitrary.

We use the cut-off function $\theta : [0,\infty )\to [0,1]$ introduced in Section 2. Recall that $\theta$ is smooth, non-increasing, and satisfies $\theta \equiv 1$ on $[0,1]$, $\theta \equiv 0$ on $[2,\infty )$.
Let
\eqq{\Theta (r):=\frac{1}{r}\int _0^r\theta (s)\,ds\quad (r>0),\qquad \Theta (0):=\theta (0)=1.}
We notice that
\eq{eq:theta}{0\le \theta (r)\le \Theta (r)\le \min \{ 1,\,2/r\} ,\qquad 0\le -\Theta '(r)=\frac{(\Theta -\theta )(r)}{r}\begin{cases}
\lec r^{-2} &(r>1)\\
=0 &(0\le r\le 1).
\end{cases}}

We will derive a contradiction by taking a close look at the quantity 
\eqq{\FR{M}(t):=\int _{\R^4}\Theta \Big( \frac{\ti{N}(t)|x|}{L}\Big) \ti{N}(t)x\cdot \Im \Big[ \bbar{U}\nabla U+\frac{1}{2}\bbar{V}\nabla V\Big] (t,x)\,dx}
on some time interval $[0,T)$ which is a union of characteristic intervals.
Here, $L>0$ is a positive large constant to be chosen later and $(U,V):=P_{\le K}(u,v)$ with $K:=\int _0^TN(t)^3\,dt<\I$.
(We need to localize to low frequencies so that $\FR{M}(t)$ will be finite for $L^2$ solutions.)
Also, a $C^1$ function $\ti{N}:[0,\I )\to \R _+$ is a variant of the frequency scale function $N(\cdot )$ of $(u,v)$ which will also be defined in the proof, and we just assume for now that
\eq{Ntilde}{\left\{ \begin{split}
&~\ti{N}(t)\le N(t)\,(\le 1)\quad \text{for any $t\ge 0$;}\quad \sup _{t\in J_k}\ti{N}(t)\le C_0\inf _{t\in J_k}\ti{N}(t)\quad \text{for any $k\ge 0$;}\\[-5pt]
&~\text{on each $J_k$, $\ti{N}(t)$ is monotone and}\quad |\ti{N}'(t)|\le \frac{2C_0\ti{N}(t)}{|J_k|}, 
\end{split}\right.}
where $C_0=C_0(u,v)>1$ is the constant given in \eqref{condition-Nxi}.

\begin{rem}
Recall that Dodson's argument for \eqref{single} in \cite{MR3406535} essentially used the virial identity; if $u$ is a nontrivial solution of \eqref{single} and $\tnorm{u}{L^2}<\tnorm{Q}{L^2}$, then
\eqq{
\frac{d^2}{dt^2}\norm{xu(t)}{L^2}^2=\frac{d}{dt}\Big( 4\Im \int _{\R ^d} x\cdot \bbar{u(t)}\nabla u(t)\Big) =16E_{\text{NLS}}(u(t))>0,}
where
\eqq{E_{\text{NLS}}(u)=\frac{1}{2}\norm{\nabla u}{L^2}^2-\frac{1}{2_*}\norm{u}{L^{2_*}}^{2_*},\qquad 2_*=2+\frac{4}{d}}
is the conserved energy for \eqref{single}.
In our case, the following analogue is valid; for a solution $(u,v)$ of \eqref{NLS}, 
\eq{eq:virial}{
\frac{d}{dt}\Big( 4\Im \int x\cdot \Big( \bbar{u(t)}\nabla u(t)+\frac{1}{2}\bbar{v(t)}\nabla v(t)\Big) \Big) =8E(u,v)(t).} 
Roughly speaking, $\FR{M}(t)$ is a modification of the virial quantity $\Im \int x\cdot (\bbar{u}\nabla u+\frac{1}{2}\bbar{v}\nabla v)$, for which the time derivative is always positive and away from zero by \eqref{eq:virial} and the minimality of the mass $M(u,v)=M_{c,\text{rad}}<M(\phi ,\psi )$.
If we assume $N(t)\equiv 1$, then $\FR{M}(T)-\FR{M}(0)\gec T=K$, whereas the almost periodicity suggests that the solution is in some sense localized to low frequencies uniformly in $t\ge 0$; we can in fact show that $\tnorm{\nabla (U,V)}{L^\I ([0,T);L^2)}=o(K)$ as $K\to \I$.
Since the weight function $\Theta (|x|/L)x$ is bounded, we obtain $\FR{M}(t)=o(K)$ uniformly in $t$, which contradicts the fact $\FR{M}(T)-\FR{M}(0)\gec K$ for $K$ sufficiently large.
We will explain later the idea for the case that $N(t)$ varies, where we need to introduce a time-dependent weight function with a carefully chosen $\ti{N}(t)$.
\end{rem}

$(U,V)$ solves the following perturbed system:
\eqs{i\p _tU +\Delta U=\bbar{U}V+F,\qquad i\p _tV +\kappa \Delta V=U^2+G,\\
F:=P_{\le K}(\bbar{u}v)-\bbar{U}V,\qquad G:=P_{\le K}(u^2)-U^2.}
Using equations, we calculate time derivative of the momentum density as
\eqq{
\p _t\Im \Big( \bbar{U}\p _jU+\frac{1}{2}\bbar{V}\p _jV\Big)
&=-2\p _k\Re \Big( \p _jU\bbar{\p _kU}+\frac{\kappa}{2}\p _jV\bbar{\p _kV}\Big) -\frac{1}{2}\p _j\Re (U^2\bbar{V})\\
&\hx +\frac{1}{2}\p _j\Delta \Big( |U|^2+\frac{\kappa}{2}|V|^2\Big) +\big\{ (U,V),(F,G)\big\} _{p,j},}
where (and hereafter) we always take summation with respect to repeated indices and
\eqq{
\big\{ (U,V),(F,G)\big\} _{p,j}:=&\, \Re (\bbar{F}\p _jU-\bbar{U}\p _jF)+\frac{1}{2}\Re (\bbar{G}\p _jV-\bbar{V}\p _jG)\\
=&\, 2\Re (\bbar{F}\p _jU+\frac{1}{2}\bbar{G}\p _jV)-\p _j \Re (\bbar{F}U+\frac{1}{2}\bbar{G}V).}
Then, we have
\eqq{\p _t\FR{M}(t)=&\int \p _t\Big[ \Theta \Big( \frac{|A|}{L}\Big) A_j\Big] \Im \Big[ \bbar{U}\p _jU+\frac{1}{2}\bbar{V}\p _jV\Big] (t,x)\,dx\\
&+2\int \p _k\Big[ \Theta \Big( \frac{|A|}{L}\Big) A_j\Big] \Re \Big[ \p _jU\bbar{\p _kU}+\frac{\kappa}{2}\p _jV\bbar{\p _kV}\Big] (t,x)\,dx\\
&+\frac{1}{2}\int \p _j\Big[ \Theta \Big( \frac{|A|}{L}\Big) A_j\Big] \Re \big[ U^2\bbar{V}\big] (t,x)\,dx\\
&-\frac{1}{2}\int \p _j\Delta \Big[ \Theta \Big( \frac{|A|}{L}\Big) A_j\Big] \Big[ |U|^2+\frac{\kappa}{2}|V|^2\Big] (t,x)\,dx\\
&+\int \Theta \Big( \frac{|A|}{L}\Big) A_j\big\{ (U,V),(F,G)\big\} _{p,j}(t,x)\,dx.}
where we have written $A:=\ti{N}(t)x$.
We now see
\eqq{\p _t\Big[ \Theta \Big( \frac{|A|}{L}\Big) A_j\Big] &=\theta \Big( \frac{|A|}{L}\Big) \ti{N}'(t)x_j,\\
\p _k\Big[ \Theta \Big( \frac{|A|}{L}\Big) A_j\Big] &=\Big[ \theta \Big( \frac{|A|}{L}\Big) \delta_{jk} +(\Theta -\theta )\Big( \frac{|A|}{L}\Big) B_{jk}\Big] \ti{N}(t),\qquad B_{jk}:=\delta _{jk}-\frac{x_jx_k}{|x|^2},\\
\p _j\Big[ \Theta \Big( \frac{|A|}{L}\Big) A_j\Big] &=\Big[ 4\theta \Big( \frac{|A|}{L}\Big) +3(\Theta -\theta )\Big( \frac{|A|}{L}\Big) \Big] \ti{N}(t).
}
Therefore, we obtain
\begin{align}
\p _t&\FR{M}(t)\notag \\
=\,&\ti{N}'(t)\int \th \Big( \frac{|A|}{L}\Big) x_j\Im \Big[ \bbar{U}\p _jU+\frac{1}{2}\bbar{V}\p _jV\Big] (t,x)\,dx \label{Ndash-rad}\\
&+2\ti{N}(t)\int \th \Big( \frac{|A|}{L}\Big) \Big[ |\nabla U|^2+\frac{\kappa}{2}|\nabla V|^2\Big] (t,x)\,dx \label{positiveH1-rad} \\
&+2\ti{N}(t)\int (\Th -\th )\Big( \frac{|A|}{L}\Big)  B_{jk}\Re \Big[ \p _jU\bbar{\p _kU}+\frac{\kappa}{2}\p _jV\bbar{\p _kV}\Big] (t,x)\,dx \label{positiveAD-rad} \\
&+2\ti{N}(t)\int \th \Big( \frac{|A|}{L}\Big) \Re \big[ U^2\bbar{V}\big] (t,x)\,dx \label{negativeL3-rad} \\
&+\frac{3}{2}\ti{N}(t)\int (\Th -\th )\Big( \frac{|A|}{L}\Big) \Re \big[ U^2\bbar{V}\big] (t,x)\,dx \label{L3tooku-rad} \\
&-\frac{1}{2}\ti{N}(t)\int \Delta \Big[ (\th +3\Th )\Big( \frac{|A|}{L}\Big) \Big] \Big[ |U|^2+\frac{\kappa}{2}|V|^2\Big] (t,x)\,dx \label{L2L2-rad} \\
&+\ti{N}(t)\int \Th \Big( \frac{|A|}{L}\Big) x_j\big\{ (U,V),(F,G)\big\} _{p,j}(t,x)\,dx. \label{gomiP-rad} 
\end{align}

We shall prove the following:
\begin{thm}\label{thm:im-rad}
There exists a large constant $L>0$ and a $C^1$ function $\ti{N}:[0,\I )\to \R _+$ satisfying \eqref{Ntilde} such that $\FR{M}(t)$ defined as above satisfies  
\eqq{\int _0^T\frac{d\FR{M}}{dt}(t)\,dt \gec K}
for any sufficiently large $T$.
\end{thm}
Note that we can choose $T$ for which $K=K(T)=\int _0^TN(t)^3\,dt$ is arbitrarily large, since we are dealing with the case $\int _0^\I N(t)^3\,dt=\I$.

\begin{proof}
Let us begin with the estimate on 
\eqq{\eqref{positiveH1-rad}+\eqref{negativeL3-rad}=2\ti{N}(t)\int \th \Big( \frac{|A|}{L}\Big) \Big[ |\nabla U|^2+\frac{\kappa}{2}|\nabla V|^2+\Re (U^2\bbar{V})\Big] (t,x)\,dx .}
We can find the energy of the system in it, so by the minimality assumption $M(u,v)=M_c<M(\phi,\psi )$ and the sharp Gagliardo-Nirenberg inequality (Lemma~\ref{lem:GN} (ii)) we expect these terms to be positive and the main part of $\int _0^T\frac{d\FR{M}}{dt}(t)\,dt$.
To show that, we introduce another cutoff $\chi :[0,\infty ) \to [0,1]$, a smooth non-increasing function satisfying $\chi (r)=1$ on $[0,\frac{1}{2}]$ and $\chi (r)=0$ on $[1,\infty )$, so that $\theta \equiv 1$ on the support of $\chi$.
For a small $\e >0$, we have%
\footnote{We keep a small amount of the $\dot{H}^1$ norm for later use.}
\eqq{&\int \th \Big( \frac{|A|}{L}\Big) \Big[ (1-\e )\big( |\nabla U|^2+\frac{\kappa}{2}|\nabla V|^2\big) +\Re [U^2\bbar{V}] \Big] \,dx\\
&\ge \int \Big[ (1-\e )\chi ^2\Big( \frac{|A|}{L}\Big) \big( |\nabla U|^2+\frac{\kappa}{2}|\nabla V|^2\big) +\chi ^3\Big( \frac{|A|}{L}\Big) \Re [U^2\bbar{V}] +\big[ \th (1-\chi ^3)\big] \Big( \frac{|A|}{L}\Big) \Re [U^2\bbar{V}] \Big] \,dx\\
&\ge \frac{1-\e}{1+\e}\int \Big[ | \nabla \ti{U}|^2+\frac{\kappa}{2}|\nabla \ti{V}|^2\Big] \,dx+\int \Re [\ti{U}^2\bbar{\ti{V}}] \,dx\\
&\hx -\frac{1-\e}{\e}\Big( \frac{\ti{N}(t)}{L}\Big) ^2\int \Big| \chi '\Big( \frac{|A|}{L}\Big) \Big| ^2\Big[ |U|^2+\frac{\kappa}{2}|V|^2\Big] \,dx-\int \big[ \th (1-\chi ^3)\big] \Big( \frac{|A|}{L}\Big) \big[ |U|^3+|V|^3\big] \,dx\\
&\ge \frac{1-\e}{1+\e}\int \Big[ | \nabla \ti{U}|^2+\frac{\kappa}{2}|\nabla \ti{V}|^2\Big] \,dx+\int \Re [\ti{U}^2\bbar{\ti{V}}] \,dx\\
&\hx -C(\e ,\kappa )\Big( \frac{\ti{N}(t)}{L}\Big) ^2-\int _{|x|\ge \frac{L}{2\ti{N}(t)}}\big[ |U|^3+|V|^3\big] \,dx,}
where $(\ti{U},\ti{V}):=\chi (|A|/L)(U,V)$ and we have used the inequality
\eq{triangle_square}{|X|^2\ge \frac{1}{1+\e}|X+Y|^2-\frac{1}{\e}|Y|^2,\qquad \e >0,\quad X,Y\in \Bo{C}^4.}
We invoke the sharp Gagliardo-Nirenberg inequality (Lemma~\ref{lem:GN} (ii)) for the function $(\ti{U}(t),\ti{V}(t))$.
Since 
\eqq{\frac{M(\ti{U}(t),\ti{V}(t))}{M(\phi ,\psi )}=\frac{M(\chi (|A|/L)P_{\le K}(u(t),v(t)))}{M(\phi ,\psi )}\le \frac{M(u(t),v(t))}{M(\phi ,\psi )}=\frac{M_{c,\text{rad}}}{M(\phi ,\psi )}=:\eta _0\in (0,1),}
we can select $\e >0$ according to $\eta _0,\kappa$ so that
\eqq{&\frac{1-\e}{1+\e}\int \Big[ | \nabla \ti{U}|^2+\frac{\kappa}{2}|\nabla \ti{V}|^2\Big] \,dx+\int \Re [\ti{U}^2\bbar{\ti{V}}] \,dx\ge \e \int \Big[ | \nabla \ti{U}|^2+|\nabla \ti{V}|^2\Big] \,dx,}
which is, via the Gagliardo-Nirenberg inequality for a single function, bounded from below by
\eqq{&c\e \int \Big[ |\ti{U}|^3+|\ti{V}|^3\Big] \,dx\ge c\e \int _{|x|\le \frac{L}{2\ti{N}(t)}}\big[ |U|^3+|V|^3\big] \,dx.}
Consequently, we have
\eqq{&\int _0^T\Big[ \eqref{positiveH1-rad}+\eqref{negativeL3-rad}\Big] \,dt \\
&\ge \frac{1}{C}\int _0^T\ti{N}(t)\int \th \Big( \frac{|A|}{L}\Big) \Big[ |\nabla U|^2+\frac{\kappa}{2}|\nabla V|^2\Big] \,dx\,dt+\frac{1}{C}\int _0^T\ti{N}(t)\int _{|x|\le \frac{L}{2\ti{N}(t)}}\big[ |U|^3+|V|^3\big] \,dx\,dt \\
&\hx -\frac{C}{L^2}\int _0^T\ti{N}(t)^3\,dt-2\int _0^T\ti{N}(t)\int _{|x|\ge \frac{L}{2\ti{N}(t)}}\big[ |U|^3+|V|^3\big] \,dx\,dt}
for some $C=C(\eta _0,\kappa )\gg 1$.

We next observe that $\eqref{positiveAD-rad}\ge 0$.
This follows from the fact that the matrix $(B_{jk}(x))_{1\le j,k\le 4}$ is non-negative for any $x$.

All the remaining terms are considered as error terms.
Among them, $\eqref{L3tooku-rad}+\eqref{L2L2-rad}$ is easy to handle.
In fact, we see
\eqs{\Big| \int _0^T\eqref{L3tooku-rad}\,dt\Big| \le 3\int _0^T\ti{N}(t)\int _{|x|\ge \frac{L}{\ti{N}(t)}}\big[ |U|^3+|V|^3\big] \,dx\,dt,\\
\Big| \int _0^T\eqref{L2L2-rad}\,dt\Big| \le \frac{C(\kappa )}{L^2}\int _0^T\ti{N}(t)^3\,dt.}

Let us turn to the control of \eqref{Ndash-rad} including the time derivative of $\ti{N}(t)$.
By the Cauchy-Schwarz inequality, for $\e >0$ we have
\eqq{\Big| \int _0^T\eqref{Ndash-rad}\,dt\Big| &\le 2L\int _0^T\frac{|\ti{N}'(t)|}{\ti{N}(t)}\int \th \Big( \frac{|A|}{L}\Big) \Big[ |U||\nabla U|+\frac{1}{2}|V||\nabla V|\Big] \,dx\,dt\\
&\le \e \int _0^T\ti{N}(t)\int \th \Big( \frac{|A|}{L}\Big) \Big[ |\nabla U|^2+\frac{\kappa}{2}|\nabla V|^2\Big] \,dx\,dt+C(\kappa )\frac{L^2}{\e}\int _0^T\frac{\ti{N}'(t)^2}{\ti{N}(t)^3}\,dt.}
Taking $\e =\e (\eta _0,\kappa )>0$ sufficiently small, together with the estimates obtained so far and $\ti{N}(t)\le N(t)$, we come to 
\eqq{&\int _0^T\Big[ \eqref{Ndash-rad}+\cdots +\eqref{L2L2-rad}\Big] \,dt \\
&\ge \frac{1}{C}\int _0^T\ti{N}(t)\int _{|x|\le \frac{L}{2\ti{N}(t)}}\big[ |U|^3+|V|^3\big] \,dx\,dt -5\int _0^T\ti{N}(t)\int _{|x|\ge \frac{L}{2\ti{N}(t)}}\big[ |U|^3+|V|^3\big] \,dx\,dt\\
&\hx -\frac{C}{L^2}\int _0^T\ti{N}(t)^3\,dt-CL^2\int _0^T\frac{\ti{N}'(t)^2}{\ti{N}(t)^3}\,dt\\
&\ge \frac{1}{C}\int _0^T\ti{N}(t)\int _{|x|\le \frac{L}{2N(t)}}\big[ |U|^3+|V|^3\big] \,dx\,dt -5\int _0^T\ti{N}(t)\int _{|x|\ge \frac{L}{2N(t)}}\big[ |U|^3+|V|^3\big] \,dx\,dt\\
&\hx -\frac{C}{L^2}\int _0^T\ti{N}(t)N(t)^2\,dt-CL^2\int _0^T\frac{\ti{N}'(t)^2}{\ti{N}(t)^3}\,dt}
for some $C=C(\eta _0,\kappa )\gg 1$.

Now, we deduce from the almost periodicity and Lemma~\ref{lem:tightness-Str} that
\eqq{\int _{J_k}\int _{|x|\le \frac{L}{2N(t)}}\Big[ |U|^3+|V|^3\Big] \,dx\,dt \sim \int _{J_k}\int _{\R^4}\Big[ |u|^3+|v|^3\Big] \,dx\,dt=1}
for sufficiently large $K,L$, uniformly in $k\ge 0$, and 
\eqq{\sup _k\int _{J_k}\int _{|x|>\frac{L}{2N(t)}}\Big[ |U|^3+|V|^3\Big] \,dx\,dt =o(1)\qquad (K,L\to \I ).}
Since we are assuming $\ti{N}(t)\sim \ti{N}(t_k)$ for $t\in J_k$ uniformly in $k\ge 0$, it follows from Corollary~\ref{lem:chinterval} that
\eqq{&\int _0^T\ti{N}(t)\int _{|x|\le \frac{L}{2N(t)}}\Big[ |U|^3+|V|^3\Big] \,dx\,dt\sim \sum _{k;\,J_k\subset [0,T)}\ti{N}(t_k)
\sim \int _0^T\ti{N}(t)N(t)^2\,dt}
for sufficiently large $K,L$ and 
\eqq{&\int _0^T\ti{N}(t)\int _{|x|>\frac{L}{2N(t)}}\Big[ |U|^3+|V|^3\Big] \,dx\,dt=\int _0^T\ti{N}(t)N(t)^2\,dt\cdot o(1)\qquad (K,L\to \I ).}

We therefore have proved the following:
\begin{lem}\label{lem:im1-rad}
Let $\ti{N}:[0,T)\to \R _+$ be a $C^1$ function satisfying \eqref{Ntilde}.
Then, there exist constants $C\gg 1$ and $L\gg 1$ depending only on $(u,v)$ and $\kappa$ such that we have 
\eq{im2-rad}{\int _0^T\Big[ \eqref{Ndash-rad}+\cdots +\eqref{L2L2-rad}\Big] \,dt \ge \frac{1}{C}\int _0^T \ti{N}(t)N(t)^2\,dt-C\int _0^T\frac{\ti{N}'(t)^2}{\ti{N}(t)^3}\,dt}
whenever $K>0$ is sufficiently large.
\end{lem}

From now on we fix such an $L\gg 1$.
Our next task is to find an appropriate function $\ti{N}(t)$ satisfying \eqref{Ntilde} and
\eq{claim_Ntilde}{\int _0^T\frac{\ti{N}'(t)^2}{\ti{N}(t)^3}\,dt\ll \int _0^T \ti{N}(t)N(t)^2\,dt.}
This is the key step in the proof of Theorem~\ref{thm:im}, and we follow the argument of Dodson \cite{MR3406535}.
To see the idea, we observe that \eqref{Ntilde} and Corollary~\ref{lem:chinterval} imply
\eqq{\int _0^T\frac{\ti{N}'(t)^2}{\ti{N}(t)^3}\,dt\lec \sum _{k;\, J_k\subset [0, T)}\frac{1}{\ti{N}(t_k)^3}\frac{\ti{N}(t_k)}{|J_k|}\int _{J_k}|\ti{N}'(t)|\,dt\sim \sum _{k;\, J_k\subset [0, T)}\frac{N(t_k)^2}{\ti{N}(t_k)^2}|\ti{N}(t_k)-\ti{N}(t_{k+1})|.}
Let us consider the case $\ti{N}(t)\sim N(t)$.
(In fact, we can easily construct $\ti{N}(t)$ satisfying \eqref{Ntilde} and $C_0^{-1}\le \ti{N}(t)/N(t)\le 1$.)
We have
\eqq{\sum _{k;\, J_k\subset [0, T)}\frac{N(t_k)^2}{\ti{N}(t_k)^2}|\ti{N}(t_k)-\ti{N}(t_{k+1})|\lec \sum _{k;\, J_k\subset [0, T)}N(t_k)\sim \int _0^TN(t)^3\,dt\sim \int _0^T \ti{N}(t)N(t)^2\,dt,}
which is not sufficient to conclude \eqref{claim_Ntilde}.
However, if we consider the extremal case where $N(t)$ is monotone on the whole interval $[0,T)$, then we can construct $\ti{N}(t)$ which is also monotone on $[0,T)$ and
\eqq{\sum _{k;\, J_k\subset [0, T)}\frac{N(t_k)^2}{\ti{N}(t_k)^2}|\ti{N}(t_k)-\ti{N}(t_{k+1})|\lec \ti{N}(0)-\ti{N}(T)\le 1\ll K\sim \int _0^T \ti{N}(t)N(t)^2\,dt,}
since we can take $K$ arbitrarily large.
This observation suggests that \eqref{claim_Ntilde} is easier to achieve if $N(t)$ is less oscillatory.
The idea for constructing $\ti{N}(t)$ is that we deform $N(t)$ to be less undulating by leveling `peaks' in the graph of $N(t)$.

\begin{lem}\label{lem:im-Ndash}
There exists a sequence of $C^1$ functions $\ti{N}_m:[0,\I )\to \R$ ($m=0,1,2,\dots $) such that each $\ti{N}_m$ satisfies \eqref{Ntilde} and 
\eqq{&\sup _{t\ge 0}\frac{N(t)}{\ti{N}_m(t)}\le C_0^{m+2},\\
&\sum _{k;\, J_k\subset [0, T)}\frac{N(t_k)^2}{\ti{N}_m(t_k)^2}|\ti{N}_m(t_k)-\ti{N}_m(t_{k+1})|\lec \Big( \frac{C_0^m}{K}+\frac{1}{m+1}\Big) \int _0^T \ti{N}_m(t)N(t)^2\,dt}
for any $T>0$.
\end{lem}
\begin{proof}
Recall that $N(t)$ is a step function associated with the partition $\shugo{J_k}_{k=0}^\I$ of $[0,\I )$ and satisfies \eqref{condition-Nxi}.
Define the $C_0^{\Bo{Z}_{\le 0}}$-valued step functions $N_m(t)$ ($m=0,1,2,\dots$) associated with $\shugo{J_k}$ inductively in $m$ as follows: 
Let $N_0(t):= N(t)$.
Then, for given $N_m(t)$, define $N_{m+1}(t)$ by
\eqq{N_{m+1}(t_k)&:=\begin{cases} C_0^{-1}N_m(t_k) &\text{if $J_k\subset$ a \emph{peak} in $N_m$,}\\ N_m(t_k) &\text{otherwise,}\end{cases}\qquad \text{for $k=0,1,2,\dots$,}\\
N_{m+1}(t)&:=\sum _{k=0}^\I N_{m+1}(t_k)\chf{J_k}(t).}
where for a positive integer $l$ we call a union of consecutive $l$ characteristic intervals $[t_{k_0},t_{k_0+l})$ a peak of length $l$ in $N_m$ if $N_m(t)\equiv N_m(t_{k_0})$ on $[t_{k_0},t_{k_0+l})$ and $N_m(t_{k_0-1})=N_m(t_{k_0+l})=C_0^{-1}N_m(t_{k_0})$.
It is easily verified that $C_0^{-m}N(t)\le N_m(t)\le N(t)$, $N_m(t_{k+1})/N_m(t_k)\in \shugo{C_0^{-1},\,1,\,C_0}$ for any $m,k\ge 0$.
Moreover, we claim that: (i) Every peaks in $N_m$ has length $\ge 2m+1$; (ii) If $N_m(t_k)\neq N(t_k)$, then $N_m(t_k)=N_m(t_{k+1})$.
(i) follows from the definition.
To verify (ii), assume $N_m(t_k)\neq N(t_k)$ for some $m,k\ge 0$.
Then, there exists $0<m'\le m$ such that $N_{m'}(t_k)=C_0^{-1}N_{m'-1}(t_k)$, \mbox{i.e.} $J_k$ is included in a peak of $N_{m'-1}$, which implies $N_{m'}(t_k)=N_{m'}(t_{k+1})$.
We see that this coincidence of the value of $N_{m'}$ at consecutive points $t_k$, $t_{k+1}$ persists throughout the construction procedure of $\shugo{N_m}$.

Let $k_*$ be a positive integer, and let $\shugo{[t_{k^m_j},t_{k^m_j+l^m_j})}_{j=1}^{j_*(m,k_*)}$ be the set of all peaks of $N_m$ included in the interval $[0,t_{k_*})$.
($j_*(m,k_*)$ is the number of peaks in $N_m$ before $t_{k_*}$, and for the $j$-th peak in $N_m$ we denote by $k^m_j$ and $l^m_j$ the index of the beginning characteristic interval and the length, respectively.
Note that this set is possibly empty.)
Then, the total variation of $N_m$ on $[0,t_{k_*})$ is estimated as
\eqq{\sum _{k=0}^{k_*-1}|N_m(t_k)-N_m(t_{k+1})|\le N_m(0)+N_m(t_{k_*})+2\sum _{j=1}^{j_*(m,k_*)}N(t_{k_j^m}).}
Hence, the properties (i), (ii) imply that
\eqq{&\sum _{k=0}^{k_*}\frac{N(t_k)^2}{N_m(t_k)^2}|N_m(t_k)-N_m(t_{k+1})|=\sum _{k=0}^{k_*}|N_m(t_k)-N_m(t_{k+1})|\\
&\le 2+2\sum _{j=1}^{j_*(m,k_*)}N(t_{k_j^m})\le 2+\frac{2}{2m+1}\sum _{k=0}^{k_*}N_m(t_k)\lec 1+\frac{1}{m+1}\int _0^{t_{k_*}} N_m(t)N(t)^2\,dt.}

Finally, we can construct a $C^1$ function $\ti{N}_m(t)$ on $[0,\I )$ satisfying \eqref{Ntilde} such that $C_0^{-2}N_m(t)\le \ti{N}_m(t)\le N_m(t)$ and $\ti{N}_m(t_k)=C_0^{-1}N_m(t_k)$ ($k\ge 0$).
(For instance, it suffices to connect the points $\shugo{(t_k,\,N_m(t_k))}_{k\ge 0}$ on the graph smoothly and multiply it by $C_0^{-1}$.)
This $\ti{N}_m$ also satisfies
\eqq{&\sum _{k;\, J_k\subset [0, T)}\frac{N(t_k)^2}{\ti{N}_m(t_k)^2}|\ti{N}_m(t_k)-\ti{N}_m(t_{k+1})|\le C_0^3\sum _{k;\, J_k\subset [0, T)}\frac{N(t_k)^2}{N_m(t_k)^2}|N_m(t_k)-N_m(t_{k+1})|\\
&\lec \frac{1}{K}\int _0^TN(t)^3\,dt+\frac{1}{m+1}\int _0^T N_m(t)N(t)^2\,dt\le C_0^2(\frac{C_0^m}{K}+\frac{1}{m+1})\int _0^T \ti{N}_m(t)N(t)^2\,dt}
for any $T>0$, as desired.
\end{proof}

From Lemma~\ref{lem:im-Ndash} we deduce the following results.
\begin{cor}\label{cor:im-Ndash}
For any $\e >0$ there exists $C(\e )>0$ and a $C^1$ function $\ti{N}_\e :[0,\I )\to \R _+$ satisfying \eqref{Ntilde} such that
\eqq{\sup _{t\ge 0}\frac{N(t)}{\ti{N}_{\e}(t)}\le C(\e ),\qquad \int _0^T\frac{\ti{N}_\e '(t)^2}{\ti{N}_\e (t)^3}\,dt\le \Big( \e +\frac{C(\e )}{K}\Big) \int _0^T \ti{N}_\e (t)N(t)^2\,dt}
for any $T>0$.
\end{cor}

Now we choose $\e =\frac{1}{4C^2}$ with the constants $C$ given in \eqref{im2-rad}, and fix the function $\ti{N}$ to be $\ti{N}_\e$ given in Corollary~\ref{cor:im-Ndash}.
Then, whenever $K$ is sufficiently large and $C(\e )\le \e K$, we have
\eq{im3-rad}{\int _0^T\Big[ \eqref{Ndash-rad}+\cdots +\eqref{L2L2-rad}\Big] \,dt\ge \frac{K}{2C\cdot C(\e )}.}

All we have to do is the estimate for the error term \eqref{gomiP-rad} arising from frequency localization by $P_{\le K}$.
To conclude the proof of Theorem~\ref{thm:im-rad}, we claim that
\eq{im4-rad}{\int _0^T\eqref{gomiP-rad}\,dt =o(K)\qquad (K\to \I ).}

The long-time Strichartz estimate (Theorem~\ref{thm:ls}) now plays an essential role, since each of $F$, $G$ contains at least one high-frequency function $P_{>K/4}(u,v)$; for instance,
\eqq{F:=&\,P_{\le K}(\bbar{u}v)-\bbar{P_{\le K}u}P_{\le K}v=\bbar{u}v-\bbar{P_{\le K}u}P_{\le K}v-P_{>K}(\bbar{u}v)\\
=&\,\bbar{P_{>K}u}v+\bbar{P_{\le K}u}P_{>K}v-P_{>K}(\bbar{P_{>K/4}u}v)-P_{>K}(\bbar{P_{\le K/4}u}P_{>K/4}v).}
First, we see that
\eq{highfreq-rad}{\norm{P_{>K/4}(u,v)}{L^2L^4([0,T)\times \R^4)}=o(1),\qquad K\to \I ,}
which is a direct consequence of Theorem~\ref{thm:ls}.%
\footnote{This is the only point where we exploit the decaying factor $\rho (N)$ in the long-time Strichartz estimate.}
We prepare one more estimate:
\eq{lowfreq-rad}{\norm{|\nabla |^sP_{\lec K}(u,v)}{L^2L^4([0,T)\times \R^4)}\lec K^s\qquad \text{for $s>1/2$},}
which is also deduced from Theorem~\ref{thm:ls} as
\eqq{\norm{|\nabla |^sP_{\lec K}(u,v)}{L^2L^4}&\lec \sum _{N<K}N^s\norm{P_{>N}(u,v)}{L^2L^4}\lec \sum _{N<K}N^s\Big( \frac{K}{N}\Big) ^{1/2}\lec K^s.}

To verify \eqref{im4-rad}, we begin with observing that
\eqq{\int _0^T\eqref{gomiP-rad}\,dt &=2\int _0^T\int \Th \Big( \frac{|A|}{L}\Big) A_j\Re \Big[ \bbar{F}\p _jU+\frac{1}{2}\bbar{G}\p _jV\Big] \,dx\,dt\\
&\hx +\int _0^T\ti{N}(t)\int (\theta +3\Theta ) \Big( \frac{|A|}{L}\Big) \Re \Big[ \bbar{F}U+\frac{1}{2}\bbar{G}V\Big] \,dx\,dt.}
Noticing $\Theta (r)\le 2/r$ and using \eqref{highfreq-rad}--\eqref{lowfreq-rad}, we estimate the first integral by
\eqq{&4L\norm{(F,G)}{L^2L^{4/3}}\norm{\nabla (U,V)}{L^2L^4}\\
&\lec \norm{P_{>K/4}(u,v)}{L^2L^4}\norm{(u,v)}{L^\I L^2}\norm{\nabla P_{\le K}(u,v)}{L^2L^4}=o(K).}
The second integral is estimated by the Sobolev embedding and \eqref{highfreq-rad}--\eqref{lowfreq-rad} as
\eqq{&4\Big( \int _0^T\ti{N}(t)^3\,dt\Big) ^{1/3}\norm{(F,G)}{L^2L^{4/3}}\norm{(U,V)}{L^6L^4}\\
&\lec K^{1/3}\norm{P_{>K/4}(u,v)}{L^2L^4}\norm{(u,v)}{L^\I L^2}\norm{P_{\le K}(u,v)}{L^\I L^3}^{2/3}\norm{P_{\le K}(u,v)}{L^2L^{12}}^{1/3}\\
&\lec o(K^{1/3})\cdot \norm{|\nabla |^{2/3}P_{\le K}(u,v)}{L^\I L^2}^{2/3}\norm{|\nabla |^{2/3}P_{\le K}(u,v)}{L^2L^4}^{1/3}=o(K).}

This is the end of the proof of Theorem~\ref{thm:im-rad}.
\end{proof}

In view of Theorem~\ref{thm:im-rad}, the quasi-soliton scenario is precluded once we show
\eq{M-rad}{\sup _{t\in [0,T)}|\FR{M}(t)|=o(K),\qquad K\to \I .}
Now, by the almost periodicity of $(u,v)$ and the fact $N(t)\le 1$, we have
\eqq{&\norm{\nabla P_{\le K}(u,v)}{L^\I L^2}\lec K^{1/2}\norm{P_{\le K^{1/2}}(u,v)}{L^\I L^2}+K\norm{P_{>K^{1/2}N(t)}(u,v)}{L^\I L^2}=o(K)}
as $K\to \I$.
This and the boundedness of the weight function imply \eqref{M-rad}.

So far, the proof of Theorem~\ref{thm:main2} has been completed.


\subsection{Non-radial, mass-resonance case}

For non-radial case, we need the Galilean invariance of the system and thus restrict ourselves to the mass-resonance case $\kappa =1/2$.

To treat the situation where the spatial center function $x(t)$ varies, we follow \cite{MR3406535} and introduce the interaction-type modification of $\FR{M}(t)$ of the form
\eqq{\Sc{M}(t):=\iint _{\R^4\times \R^4}\Theta _L\big( \ti{N}(t)|x-y|\big) \ti{N}(t)(x-y)\cdot \Im \Big[ \bbar{U}\nabla U+\frac{1}{2}\bbar{V}\nabla V\Big] (t,x)\Big[ |U|^2+|V|^2\Big] (t,y)\,dx\,dy}
on some time interval $[0,T)$ which is a union of characteristic intervals.
$\ti{N}:[0,\I )\to \R _+$ is the same $C^1$ function as we used in the radial case, which satisfies \eqref{Ntilde}.
In this subsection, we use the Fourier projection operator $P_{\le C_*K}$ instead of $P_{\le K}$, and so $(U,V):=P_{\le C_*K}(u,v)$, where the constant $C_*=C_*(u,v)>0$ is given in \eqref{def:Cstar} so that
\eqq{|\xi (t)|\le 2^{-10}C_*\sum _{k;\,J_k\subset [0,t)}\int _{J_k}N(t)^3\,dt\le 2^{-10}C_*K,\qquad t\in [0,T).}

Definition of the weight function $\Theta _L$ is quite different from the radial case.
Let $L$ be a large positive number to be chosen later, 
and let $\theta :[0,\I )\to [0,1]$ be the same as before.
First, we define the smooth radial function $\vth_L :\R^4 \to [0,1]$ by
\eqq{\vth _L(x):=\th (\max \{ 0,|x|-L+2\}),\qquad x\in \R ^4.}
Note that $\vth _L$ is non-increasing in $|x|$ and satisfies $\vth _L(x)=1$ for $0\le |x|\le L-1$ and $\vth _L(x)=0$ for $|x|\ge L$.
Then, define the functions $\theta _L, \Theta _L:[0,\I )\to [0,\I )$ by 
\eqq{\theta _L(r):=\int _{\R^4}\vth _L(re_1-z)\vth _L(z)\frac{dz}{L^4},\qquad \Theta _L(r):=\frac{1}{r}\int _0^r\theta _L(s)\,ds\qquad (\Theta _L(0):=\theta _L(0)),}
where $e_1:=(1,0,0,0)\in \R ^4$.
It is worth noticing that, since $\vth _L$ is radially symmetric, $\theta _L$ satisfies 
\eq{th_integral}{\theta _L (|x-y|)=\int _{\R^4}\vth _L(x-z)\vth _L(y-z)\frac{dz}{L^4},\qquad x,y\in \R^4.}
This helps us separate two spatial variables in the analysis of the interaction-type quantity $\Sc{M}(t)$.

It follows from the definition that $\theta _L$ and $\Theta _L$ are non-negative, bounded uniformly in $L$, smooth (on $(0,\I )$, for $\Theta _L$) functions such that $\theta _L\equiv 0$ outside $[0,2L]$ and $\Theta _L(r) \lec \min \shugo{1,\,L/r}$.
Moreover, we can show that:
\begin{lem}\label{bound_Thth}
The following holds.
\begin{enumerate}
\item $\theta _L$ is non-increasing. In particular, $\Theta _L\ge \theta _L\ge 0$.
\item $|\theta _L'(r)|\lec \min \shugo{1/L,\,r/L}$, $|\theta _L''(r)|\lec 1/L$.
\item $\Theta _L$ is non-increasing and $0\le -\Theta _L'(r)=\big( \Theta _L(r)-\theta _L(r)\big) /r\lec \min \shugo{L/r^2,\,1/L,\,r/L}$.
\end{enumerate}
\end{lem}
\begin{proof}
(i) Although this is intuitively obvious, the proof is rather long.
We see that
\eqq{
\theta _L'(r)&=\int (\p _{x_1}\vth _L)(re_1-z)\vth _L(z)\frac{dz}{L^4}=\int _{\mathbb{R}^{3}}\int _{\mathbb{R}} (\partial _{x_1}\vth _L) (\zeta , -z') \vth _L(r -\zeta ,z') \,d\zeta \frac{dz'}{L^4},\\
&=\int _{\mathbb{R}^{3}}\left( \int _0^\infty (\partial _{x_1}\vth _L) (\zeta , -z') \vth _L(r-\zeta ,z') \,d\zeta +\int _0^{\infty}(\partial _{x_1}\vth _L) (-\eta , -z') \vth _L(r+\eta ,z') \,d\eta \right) \frac{dz'}{L^4}.}
Observe that if $\varrho$ is an even function on $\mathbb{R}$ which is non-increasing in $[0,\infty )$, then it holds that $\varrho (t_0+t)\le \varrho (t_0-t)$ for any $t_0,t\ge 0$.%
\footnote{If $0\le t\le t_0$, then $0\le t_0-t\le t_0+t$, which implies $\varrho (t_0+t)\le \varrho (t_0-t)$.
If $t>t_0$, then $0<t-t_0\le t_0+t$, so we have $\varrho (t_0+t)\le \varrho (t-t_0)=\varrho (t_0-t)$.}
We apply this to the second integral of the above with $\varrho =\vth _L(\cdot ,z')$, $t_0=r$ and $t=\eta$.
Note that $(\partial _{x_1}\vth _L) (-\eta , -z')\ge 0$ for any $\eta >0$ and $z'\in \mathbb{R}^{3}$, because $\vth _L$ is radially symmetric and non-increasing in the radial direction.
Then, we obtain
\eqq{\theta _L'(r)&\le \int _{\mathbb{R}^{3}}\left( \int _0^\infty (\p _{x_1}\vth _L) (\zeta , -z') \vth _L(r-\zeta ,z') \,d\zeta +\int _0^\infty (\partial _{x_1}\vth _L) (-\eta , -z') \vth _L(r-\eta ,z') \,d\eta \right) \frac{dz'}{L^4}.}
Since $(\partial _{x_1}\vth _L) (\cdot ,-z')$ is an odd function, the right hand side is equal to zero, and thus $\theta _L$ is non-increasing.

(ii) We observe that
\eqq{\theta _L'(r)=\int (\p _{x_1}\vth )((r,0,0,0)-z)\vth (z)\frac{dz}{L^4},\qquad \theta _L''(r)=\int (\p _{x_1}^2\vth )((r,0,0,0)-z)\vth (z)\frac{dz}{L^4}}
are of $O(L^{-1})$ uniformly in $r$, since $|\supp{\p _{x_1}\vth}|\lec L^3$.
Integrating the second quantity with respect to $r$ we obtain $|\theta _L'(r)|\lec r/L$, since $\theta _L'(0)=\int (\p _{x_1}\vth )(-z)\vth (z)\frac{dz}{L^4}=0$.

(iii) The identity $-\Theta _L'(r)=\big( \Theta _L(r)-\theta _L(r)\big) /r$ is straightforward.
Since $\theta _L$ is non-increasing, we have 
\eqq{-\Theta _L'(r)=\frac{\Theta _L(r)-\theta _L(r)}{r}=\frac{1}{r^2}\int _0^r\big( \theta _L(s)-\theta _L(r)\big) \,ds\ge 0,}
which implies that $\Theta _L$ is also non-increasing.
The above identity also yields that $|\Theta _L'(r)|\lec L/r^2$.
Furthermore, by the mean value theorem and (ii),
\eqq{|\Theta _L'(r)|\le \frac{1}{r^2}\sup _{0<s<r}|\theta _L'(s)|\int _0^r(r-s) \,ds\lec \min \shugo{\frac{1}{L},\,\frac{r}{L}}.\qedhere} 
\end{proof}

We now compute $\p _t\Sc{M}(t)$.
Derivatives of the weight functions are given by
\eqq{\p _t\Big[ \Theta _L\big( \ti{N}(t)|x|\big) \ti{N}(t)x_j\Big] &=\th \big( \ti{N}(t)|x|\big) \ti{N}'(t)x_j,\\
\p _k\Big[ \Theta _L\big( \ti{N}(t)|x|\big) \ti{N}(t)x_j\Big] &=\Big[ \th _L\big( \ti{N}(t)|x|\big) \de _{jk}+(\Theta _L-\theta _L)\big( \ti{N}(t)|x|\big) \Big( \de _{jk}-\frac{x_jx_k}{|x|^2}\Big) \Big] \ti{N}(t),\\
\p _j\Big[ \Theta _L\big( \ti{N}(t)|x|\big) \ti{N}(t)x_j\Big] &=\Big[ 4\th _L\big( \ti{N}(t)|x|\big) +3(\Th _L-\th _L)\big( \ti{N}(t)|x|\big) \Big] \ti{N}(t),
}
and the time derivative of mass density is
\eqs{\p _t(|U|^2+|V|^2)=-2\p _j\Im \Big( \bbar{U}\p _jU+\frac{1}{2}\bbar{V}\p _jV\Big) -2\big\{ (U,V),(F,G)\big\} _m,\\
\big\{ (U,V),(F,G)\big\} _m:=\, \Im (\bbar{F}U+\bbar{G}V).}
Then, introducing the notation
\eqs{\FR{m}[U,V]:=|U|^2+|V|^2,\qquad \FR{p}[U,V]:=\Im \Big[ \bbar{U}\nabla U+\frac{1}{2}\bbar{V}\nabla V\Big] ,\\
\FR{e}_2[U,V]:=|\nabla U|^2+\frac{1}{4}|\nabla V|^2,\qquad \FR{e}_3[U,V]:=\Re \big[ U^2\bbar{V}\big] ,}
we write down $\p _t\Sc{M}(t)$ as
\begin{align}
\p _t&\Sc{M}(t)\notag \\
=\,&\ti{N}'(t)\iint \th _L(|A|)(x-y)\cdot \FR{p}[U,V](t,x) \, \FR{m}[U,V](t,y)\,dx\,dy \label{Ndash}\\
&+2\ti{N}(t)\iint \th _L(|A|) \, \FR{e}_2[U,V](t,x) \, \FR{m}[U,V](t,y)\,dx\,dy \label{positiveH1} \\
&+2\ti{N}(t)\iint (\Th _L-\th _L)(|A|) \, B_{jk} \, \Re \Big[ \p _jU\bbar{\p _kU}+\frac{1}{4}\p _jV\bbar{\p _kV}\Big] (t,x) \, \FR{m}[U,V](t,y)\,dx\,dy \label{positiveAD} \\
&+2\ti{N}(t)\iint \th _L(|A|) \, \FR{e}_3[U,V](t,x) \, \FR{m}[U,V](t,y)\,dx\,dy \label{negativeL3} \\
&+\frac{3}{2}\ti{N}(t)\iint (\Th _L-\th _L)(|A|) \, \FR{e}_3[U,V](t,x) \, \FR{m}[U,V](t,y)\,dx\,dy \label{L3tooku} \\
&-\frac{1}{2}\ti{N}(t)\iint \Delta _x\big[ (\th _L+3\Th _L)(|A|)\big] \, \Big[ |U|^2+\frac{1}{4}|V|^2\Big] (t,x) \, \FR{m}[U,V](t,y)\,dx\,dy \label{L2L2} \\
&+\ti{N}(t)\iint \Th _L(|A|)(x-y)\cdot \big\{ (U,V),(F,G)\big\} _{p}(t,x) \, \FR{m}[U,V](t,y)\,dx\,dy\label{gomiP} \\
&-2\ti{N}(t)\iint \th _L(|A|)\, \FR{p}[U,V](t,x)\cdot \FR{p}[U,V](t,y)\,dx\,dy \label{Galilean} \\
&-2\ti{N}(t)\iint (\Th _L-\th _L)(|A|) \, B_{jk} \, \FR{p}_j[U,V](t,x) \, \FR{p}_k[U,V](t,y)\,dx\,dy \label{negativeAD} \\
&-2\ti{N}(t)\iint \Th _L(|A|)(x-y)\cdot \FR{p}[U,V](t,x) \, \big\{ (U,V),(F,G)\big\} _m (t,y)\,dx\,dy,\label{gomiM}
\end{align}
where in this subsection we define $A:=\ti{N}(t)(x-y)$, $B_{jk}:=\de _{jk}-\frac{(x-y)_j(x-y)_k}{|x-y|^2}$.

Our goal is to establish the following:
\begin{thm}
\label{thm:im}
There exists a large constant $L>0$ and a $C^1$ function $\ti{N}:[0,\I )\to \R _+$ satisfying \eqref{Ntilde} such that $\Sc{M}(t)$ defined as above satisfies  
\eqq{\int _0^T\frac{d\Sc{M}}{dt}(t)\,dt \gec K}
for any sufficiently large $T$.
\end{thm}

\begin{proof}
The main part of $\p _t\Sc{M}(t)$ will be \eqref{positiveH1}+\eqref{negativeL3}+\eqref{Galilean}, which corresponds to \eqref{positiveH1-rad}+\eqref{negativeL3-rad} in the radial case.
The new term \eqref{Galilean} can be concealed by an appropriate gauge transformation.
To see this, we write $\th _L(|A|)$ in the integral form \eqref{th_integral}, then
\eqq{&\eqref{positiveH1}+\eqref{negativeL3}+\eqref{Galilean}\\
&=2\ti{N}(t)\int _{\R^4}\bigg[ \iint \vth _L(\ti{N}(t)x-z)\vth _L(\ti{N}(t)y-z)\Big[ \FR{e}_2[U,V]+\FR{e}_3[U,V]\Big] (t,x)\, \FR{m}[U,V] (t,y)\,dx\,dy \bigg] \frac{dz}{L^4}\\
&\hx -2\ti{N}(t)\int _{\R^4}\bigg[ \iint \vth _L(\ti{N}(t)x-z)\vth _L(\ti{N}(t)y-z)\, \FR{p}[U,V](t,x)\cdot \FR{p}[U,V](t,y)\,dx\,dy \bigg] \frac{dz}{L^4}\\
&=2\ti{N}(t)\int _{\R^4}\Sc{E}[U,V](t,z)\frac{dz}{L^4},}
where
\eqq{\Sc{E}[U,V](t,z)&:=\Big( \int \vth _L(\ti{N}(t)x-z)\,(\FR{e}_2+\FR{e}_3)[U,V](t,x)\,dx\Big) \Big( \int \vth _L(\ti{N}(t)x-z)\,\FR{m}[U,V](t,x)\,dx\Big) \\
&\hx -\bigg| \int \vth _L(\ti{N}(t)x-z)\,\FR{p}[U,V](t,x)\,dx\bigg| ^2.}

We observe that the gauge transformation $(U,V)\mapsto (e^{-ix\cdot \xi _0}U,e^{-2ix\cdot \xi _0}V)$, $\xi _0\in \R^4$ keeps the densities $\FR{m},\FR{e}_3$ invariant and changes $\FR{p},\FR{e}_2$ as
\eq{shita-star}{\FR{p}[e^{-ix\cdot \xi _0}U,e^{-2ix\cdot \xi _0}V]&=\FR{p}[U,V]-\xi _0\FR{m}[U,V],\\
\FR{e}_2[e^{-ix\cdot \xi _0}U,e^{-2ix\cdot \xi _0}V]&=\FR{e}_2[U,V]-2\xi _0\cdot \FR{p}[U,V]+|\xi _0|^2\FR{m}[U,V].}
From this, we see that the quantity $\Sc{E}[U,V](t,z)$ (for fixed $t,z$) is invariant under such gauge transformations.%
\footnote{This is not true if $\kappa \neq 1/2$.
Hence, it seems difficult to deal with the term \eqref{Galilean} without assuming $\kappa =1/2$.}
We choose $\xi _0=\xi _0(t,z)$ for $(t,z)\in [0,T) \times \R ^4$ so that the last term in $\Sc{E}[U,V](t,z)$ will vanish; 
\eq{def_xizero}{\xi _0(t,z):=\frac{\displaystyle\int \vth _L(\ti{N}(t)x-z)\,\FR{p}[U,V](t,x)\,dx}{\displaystyle\int \vth _L(\ti{N}(t)x-z)\,\FR{m}[U,V](t,x)\,dx},}
and $\xi _0(t,z):=0$ if the denominator vanishes (in this case the numerator is also equal to zero).
We thus consider the estimate on
\eqq{&2\ti{N}(t)\int \Big( \int \vth _L(\ti{N}(t)x-z)\,(\FR{e}_2+\FR{e}_3)[U_*,V_*](t,x,z)\,dx\Big) \Big( \int \vth _L(\ti{N}(t)x-z)\,\FR{m}[U_*,V_*](t,x,z)\,dx\Big) \frac{dz}{L^4}\\
&=2\e \ti{N}(t)\int \Big( \int \vth _L(\ti{N}(t)x-z)\,\FR{e}_2[U_*,V_*] \,dx\Big) \Big( \int \vth _L(\ti{N}(t)x-z)\,\FR{m}[U_*,V_*]\,dx\Big) \frac{dz}{L^4}\\
&\hx +2\ti{N}(t)\int \Big( \int \vth _L(\ti{N}(t)x-z)\big[ (1-\e )\FR{e}_2+\FR{e}_3\big] [U_*,V_*]\,dx\Big) \Big( \int \vth _L(\ti{N}(t)x-z)\,\FR{m}[U_*,V_*]\,dx\Big) \frac{dz}{L^4},}
where $(U_*,V_*)(t,x,z):=(e^{-ix\cdot \xi _0(t,z)}U(t,x), e^{-2ix\cdot \xi _0(t,z)}V(t,x))$ and $\e >0$ is a small constant.

As we did in the radial case, we introduce another cutoff $\chi _L:\R^4 \to [0,1]$ by 
\eqq{\chi _L(x):=\theta (\max \{ 0,|x|-L+3\} ),}
which is a smooth radial function, non-increasing in $|x|$, satisfies $\chi _L(x)=1$ for $0\le |x|\le L-2$, $\chi _L(x)=0$ for $|x|\ge L-1$, so that $\vth _L\equiv 1$ on the support of $\chi _L$.
Following the previous argument, we have
\eqq{&\int \vth _L(\ti{N}(t)x-z)\big[ (1-\e )\FR{e}_2+\FR{e}_3\big] [U_*,V_*](t,x,z)\,dx\\
&\ge \int \Big[ \frac{1-\e}{1+\e}\FR{e}_2+\FR{e}_3\Big] [\ti{U}_{*},\ti{V}_{*}](t,x,z)\,dx\\
&\hx -\frac{1-\e}{\e}\ti{N}(t)^2\int \big| (\nabla \chi _L)(\ti{N}(t)x-z)\big| ^2\Big[ |U|^2+\frac{1}{4}|V|^2\Big] (t,x)\,dx\\
&\hx -\int \big[ \vth _L(1-\chi _L^3)\big] (\ti{N}(t)x-z)\big[ |U|^3+|V|^3\big] (t,x)\,dx,}
where $(\ti{U}_{*},\ti{V}_{*}):=\chi _L(\ti{N}(t)x-z)(U_*,V_*)$.
By the sharp Gagliardo-Nirenberg inequality for the function $(\ti{U}_{*},\ti{V}_{*})(t,\cdot ,z)$ and
\eqq{&\frac{M(\ti{U}_{*}(t,\cdot ,z),\ti{V}_{*}(t,\cdot ,z))}{M(\phi ,\psi )}\le \frac{M_c}{M(\phi ,\psi )}=\eta _0\in (0,1),}
we can select $\e =\e (\eta _0)>0$ so that
\eqq{&\int \Big[ \frac{1-\e}{1+\e}\FR{e}_2+\FR{e}_3\Big] [\ti{U}_{*},\ti{V}_{*}](t,x,z)\,dx\ge \e \int \chi _L^3(\ti{N}(t)x-z)\big[ |U|^3+|V|^3\big] (t,x)\,dx.}
Therefore, we have the following lower bound:
\eqq{&\eqref{positiveH1}+\eqref{negativeL3}+\eqref{Galilean}\\
&\ge c(\eta _0)\ti{N}(t)\int \Big( \int \vth _L(\ti{N}(t)x-z)\,\FR{e}_2[U_*,V_*](t,x,z)\,dx\Big) \Big( \int \vth _L(\ti{N}(t)x-z)\,\FR{m}[U,V](t,x)\,dx\Big) \frac{dz}{L^4}\\
&+c(\eta _0) \ti{N}(t)\iint \Big( \int \chi ^3_L(\ti{N}(t)x-z)\vth _L(\ti{N}(t)y-z)\frac{dz}{L^4}\Big) \big[ |U|^3+|V|^3\big] (t,x)\,\FR{m}[U,V](t,y)\,dx\,dy\\
&-C(\eta _0) \ti{N}(t)^3\iint \Big( \int |\nabla \chi _L|^2(\ti{N}(t)x-z)\vth _L(\ti{N}(t)y-z)\frac{dz}{L^4}\Big) \,\FR{m}[U,V](t,x)\,\FR{m}[U,V](t,y)\,dx\,dy\\
&-2\ti{N}(t)\iint \Big( \int \big( \vth _L(1-\chi ^3_L)\big) (\ti{N}(t)x-z)\vth _L(\ti{N}(t)y-z)\frac{dz}{L^4}\Big) \big[ |U|^3+|V|^3\big] (t,x)\,\FR{m}[U,V](t,y)\,dx\,dy.}

From the support property of $\vth _L$ and $\chi _L$, we see that if $|x-y|\le \frac{L-2}{2\ti{N}(t)}$, then
\eqq{&\int \chi ^3_L(\ti{N}(t)x-z)\vth _L(\ti{N}(t)y-z)\frac{dz}{L^4}\ge \int _{|z-\ti{N}(t)x|\le (L-2)/2}\frac{dz}{L^4}\gec 1.}
On the other hand, since the supports of $|\nabla \chi _L|^2$ and $\vth _L(1-\chi _L^3)$ are of measure $O(L^3)$, we have
\eqq{&\int |\nabla \chi _L|^2(\ti{N}(t)x-z)\vth _L(\ti{N}(t)y-z)\frac{dz}{L^4}+\int \big( \vth _L(1-\chi _L^3)\big) (\ti{N}(t)x-z)\vth _L(\ti{N}(t)y-z)\frac{dz}{L^4}\lec \frac{1}{L}}
for any $t,x,y$.
Therefore, by $\ti{N}(t)\le N(t)$, we have
\eqq{&\eqref{positiveH1}+\eqref{negativeL3}+\eqref{Galilean}\\
&\ge c(\eta _0)\ti{N}(t)\int \Big( \int \vth _L(\ti{N}(t)x-z)\,\FR{e}_2[U_*,V_*](t,x,z)\,dx\Big) \Big( \int \vth _L(\ti{N}(t)x-z)\,\FR{m}[U,V](t,x)\,dx\Big) \frac{dz}{L^4}\\
&\hx +c(\eta _0) \ti{N}(t)\int _{|x-x(t)|\le \frac{L-2}{4N(t)}}\big[ |U|^3+|V|^3\big] (t,x)\,dx\cdot \int _{|x-x(t)|\le \frac{L-2}{4N(t)}}\,\FR{m}[U,V](t,x)\,dx\\
&\hx -C(\eta _0) \frac{\ti{N}(t)N(t)^2}{L}-C\frac{\ti{N}(t)}{L}\int _{\R^4}\big[ |U|^3+|V|^3\big] (t,x)\,dx.}
For the second term on the right-hand side, we use $|\xi (t)|\le 2^{-10}C_*K$ and $N(t)\le 1$ to see that for $q=2,3$,
\eqq{\norm{(U,V)(t)}{L^q(\shugo{|x-x(t)|\le \frac{L-2}{4N(t)}})}
&\ge \norm{(u,v)(t)}{L^q(\R^4)}-\norm{(u,v)(t)}{L^q(\shugo{|x-x(t)|>\frac{L-2}{4N(t)}})}\\
&\hx -C\norm{\big( P_{|\xi -\xi (t)|>\frac{C_*K}{4}N(t)}u(t),\,P_{|\xi -2\xi (t)|>\frac{C_*K}{4}N(t)}v(t)\big) }{L^q(\R^4)},}
which, combined with the almost periodicity and Lemma~\ref{lem:tightness-Str}, implies
\eqq{&\inf _{t\in [0,T)}\int _{|x-x(t)|\le \frac{L-2}{4N(t)}}\,\FR{m}[U,V](t,x)\,dx\gec 1,\qquad \inf _{k\ge 0}\int _{J_k}\int _{|x-x(t)|\le \frac{L-2}{4N(t)}}\big[ |U|^3+|V|^3\big] (t,x)\,dx\,dt \gec 1}
for $K$ and $L$ sufficiently large.
Hence, similarly to the radial case, we obtain the following:
\begin{lem}\label{lem:im-positive}
Let $\ti{N}:[0,T)\to \R _+$ be a $C^1$ function satisfying \eqref{Ntilde}.
Then, there exist a constant $C\gg 1$ depending only on $(u,v)$ such that we have 
\eqq{&\int _0^T\Big( \eqref{positiveH1}+\eqref{negativeL3}+\eqref{Galilean}\Big) dt \\
&\ge \frac{1}{C}\int _0^T\ti{N}(t)\int \Big( \int \vth _L(\ti{N}(t)x-z)\,\FR{e}_2[U_*,V_*](t,x,z)\,dx\Big) \Big( \int \vth _L(\ti{N}(t)x-z)\,\FR{m}[U,V](t,x)\,dx\Big) \frac{dz}{L^4}\,dt\\
&\hx +\Big( \frac{1}{C}-\frac{C}{L}\Big) \int _0^T \ti{N}(t)N(t)^2\,dt}
for sufficiently large $K,L>0$, where $(U_*,V_*):=(e^{-ix\cdot \xi _0}U,e^{-2ix\cdot \xi _0}V)$ with $\xi _0(t,z)$ given in \eqref{def_xizero}.
\end{lem}

Next, we consider $\eqref{positiveAD}+\eqref{negativeAD}$ with the matrix $B=(B_{jk})_{1\le j,k\le 4}$, which corresponds to \eqref{positiveAD-rad} in the radial case.
It turns out that the new term \eqref{negativeAD} can be absorbed into the positive term \eqref{positiveAD}.
\begin{lem}\label{lem:im-AD}
We have $\eqref{positiveAD}+\eqref{negativeAD}\ge 0$.
\end{lem}

\begin{proof}
By symmetry it suffices to show that
\eqq{\sum _{j,k=1}^4B_{jk}\bigg\{ &\Re \Big[ \p _jU\bbar{\p _kU}+\frac{1}{4}\p _jV\bbar{\p _kV}\Big] (x) \,\FR{m}[U,V](y)+\FR{m}[U,V](x)\Re \Big[ \p _jU\bbar{\p _kU}+\frac{1}{4}\p _jV\bbar{\p _kV}\Big] (y)\\[-5pt]
&-2\,\FR{p}_j[U,V](x)\,\FR{p}_k[U,V](y)\bigg\} \ge 0.}
Since $B_{jk}$ is real valued and $B_{jk}=\sum _lB_{lj}B_{lk}$, the left hand side is equal to
\eqq{&\Big[ |B\nabla U(t,x)|^2 +\frac{1}{4}|B\nabla V(t,x)|^2\Big] \,\FR{m}[U,V](t,y)+\FR{m}[U,V](t,x)\Big[ |B\nabla U(t,y)|^2+\frac{1}{4}|B\nabla V(t,y)|^2\Big] \\
&-2\Im \Big[ \bbar{U(x)}B\nabla U(x)+\frac{1}{2}\bbar{V(x)}B\nabla V(x)\Big] \cdot \Im \Big[ \bbar{U(y)}B\nabla U(y)+\frac{1}{2}\bbar{V(y)}B\nabla V(y)\Big] \\
&\ge \Big[ |B\nabla U(x)|^2 +\frac{1}{4}|B\nabla V(x)|^2\Big] \Big[ |U(y)|^2+|V(y)|^2\Big] +\Big[ |U(x)|^2+|V(x)|^2\Big] \Big[ |B\nabla U(y)|^2+\frac{1}{4}|B\nabla V(y)|^2\Big] \\
&\hx -2\Big( |U(x)||B\nabla U(y)|\cdot |U(y)||B\nabla U(x)| +\frac{1}{2}|V(x)||B\nabla V(y)|\cdot \frac{1}{2}|V(y)||B\nabla V(x)|\\
&\hxx +\frac{1}{2}|U(x)||B\nabla V(y)|\cdot |V(y)||B\nabla U(x)|+|V(x)||B\nabla U(y)|\cdot \frac{1}{2}|U(y)||B\nabla V(x)|\Big) \\
&=\Big( |U(x)||B\nabla U(y)|-|U(y)||B\nabla U(x)|\Big) ^2+\Big( \frac{1}{2}|V(x)||B\nabla V(y)|-\frac{1}{2}|V(y)||B\nabla V(x)|\Big) ^2\\
&\hx +\Big( \frac{1}{2}|U(x)||B\nabla V(y)|-|V(y)||B\nabla U(x)|\Big) ^2+\Big( |V(x)||B\nabla U(y)|-\frac{1}{2}|U(y)||B\nabla V(x)|\Big) ^2\ge 0,}
as claimed.
\end{proof}

The others will be error terms.
We begin with:
\begin{lem}\label{lem:im-L3tooku-L2L2}
We have
\eqq{\int _0^T\Big( \eqref{L3tooku}+\eqref{L2L2}\Big) \,dt =\int _0^T\ti{N}(t)N(t)^2\,dt \cdot o(1)\qquad (K,L\to \I ).}
\end{lem}

\begin{proof}
These two terms correspond to $\eqref{L3tooku-rad}+\eqref{L2L2-rad}$ in the radial case, but we need more careful treatment.

For \eqref{L3tooku} we use Lemma~\ref{bound_Thth} (iii) and the tightness in $L^\I _tL^2_x$ and $L^3_{t,x}$.
On each characteristic interval $J_k$, we split the integral as follows:
\eqq{\int _{J_k}|\eqref{L3tooku}|\,dt\lec \ti{N}(t_k)\int _{J_k}\bigg[ &\iint _{|x-x(t)|>L^{1/2}/\ti{N}(t)}+\iint _{|y-x(t)|>L^{1/2}/\ti{N}(t)}+\iint _{\mat{|x-x(t)|\le L^{1/2}/\ti{N}(t)\\ |y-x(t)|\le L^{1/2}/\ti{N}(t)}}\bigg] \\
&(\Th _L-\th _L)\big( \ti{N}(t)|x-y|\big) \big[ |U|^2|V|\big] (t,x)\,\FR{m}[U,V](t,y)\,dx\,dy\,dt.}
In the first and the second term we bound $\Th _L-\th _L$ by $1$, while in the last one we see from Lemma~\ref{bound_Thth} (iii) that $(\Th _L-\th _L)\big( \ti{N}(t)|x-y|\big)\lec L^{1/2}/L=L^{-1/2}$.
Since $\ti{N}(t)\le N(t)$, we apply the almost periodicity and Lemma~\ref{lem:tightness-Str} to conclude that the above is $\ti{N}(t_k)\cdot o(1)$ ($K,L\to \I$) uniformly in $k\ge 0$.
The claim for \eqref{L3tooku} then follows from Corollary~\ref{lem:chinterval}.

We next consider \eqref{L2L2}, in which the weight function has additional two derivatives.
Since
\eqq{\Delta [\th _L+3\Th _L](|x|)&=[\th _L''+3\Th _L''](|x|)+3|x|^{-1}[\th _L'+3\Th _L'](|x|),\qquad x\in \R ^4}
and $\Th_L''(r)=r^{-1}\big( \th _L'(r)-2\Th _L'(r)\big)$, we see from Lemma~\ref{bound_Thth} (ii) and (iii) that $\Delta _x\big[ (\th _L+3\Th _L)(|A|)\big] =O(\ti{N}(t)^2/L)$, and that
\eqq{\int _0^T|\eqref{L2L2}|\,dt\lec \frac{1}{L}\int _0^T\ti{N}(t)^3\,dt\le \frac{1}{L}\int _0^T\ti{N}(t)N(t)^2\,dt,}
as claimed.
\end{proof}

As a result, we still have the same lower bound given in Lemma~\ref{lem:im-positive} if we add the terms \eqref{positiveAD}, \eqref{negativeAD}, \eqref{L3tooku}, and \eqref{L2L2}.

We now fix $L\gg 1$ and proceed to the control of \eqref{Ndash}, which is parallel to that of \eqref{Ndash-rad}.
From the relation \eqref{shita-star}, for fixed $t,z$ we have
\eqq{&\iint \vth _L(\ti{N}(t)x-z)\vth _L(\ti{N}(t)y-z)\, (x-y)\cdot \FR{p}[U,V](t,x)\,\FR{m}[U,V](t,y)\,dx\,dy\\
&=\iint \vth _L(\ti{N}(t)x-z)\vth _L(\ti{N}(t)y-z)\, (x-y)\cdot \FR{p}[U_*,V_*](t,x,z)\,\FR{m}[U,V](t,y)\,dx\,dy\\
&\hx +\xi _0(t,z)\cdot \iint \vth _L(\ti{N}(t)x-z)\vth _L(\ti{N}(t)y-z)\, (x-y)\, \FR{m}[U,V](t,x)\,\FR{m}[U,V](t,y)\,dx\,dy,}
where the last integral is equal to zero by symmetry.
Since the triangle inequality implies
\eqq{&\vth _L(\ti{N}(t)x-z)\vth _L(\ti{N}(t)y-z)|x-y|\le \frac{2L}{\ti{N}(t)}\vth _L(\ti{N}(t)x-z)\vth _L(\ti{N}(t)y-z)}
for any $t,x,y,z$, we estimate \eqref{Ndash} similarly to the radial case with the Cauchy-Schwarz as
\eqq{|\eqref{Ndash}|&\le \e \ti{N}(t)\int \Big( \int \vth _L(\ti{N}(t)x-z)\,\FR{e}_2[U_*,V_*](t,x,z)\,dx\Big) \Big( \int \vth _L(\ti{N}(t)y-z)\,\FR{m}[U,V](t,y)\,dy\Big) \frac{dz}{L^4}\\
&\hx +\frac{CL^2}{\e}\frac{\ti{N}'(t)^2}{\ti{N}(t)^3}}
for any $\e >0$.
We can choose $\e$ small and $\ti{N}(t)$ via Corollary~\ref{cor:im-Ndash} to make the contribution from \eqref{Ndash} smaller than the lower bound given in Lemma~\ref{lem:im-positive}.
Consequently, we have
\eq{im3}{\int _0^T\Big( \eqref{Ndash}+
\cdots +\eqref{L2L2}+\eqref{Galilean}+\eqref{negativeAD}\Big) \,dt\gec K,} 
whenever $K$ is sufficiently large.

There are only $\eqref{gomiP}+\eqref{gomiM}$ remaining.
It is then sufficient to prove
\eq{im4}{\int _0^T\Big( \eqref{gomiP}+\eqref{gomiM}\Big) dt =o(K)\qquad (K\to \I ),}
which corresponds to \eqref{im4-rad} in the radial case.

Here, instead of \eqref{highfreq-rad}--\eqref{lowfreq-rad} we can obtain the following estimates from Theorem~\ref{thm:ls}:
\begin{align}
\label{highfreq} &\norm{\big( e^{-ix\cdot \xi (t)}P_{>\frac{C_*K}{4}}u,\, e^{-2ix\cdot \xi (t)}P_{>\frac{C_*K}{4}}v\big) }{L^2L^4([0,T)\times \R^4)}=o(1),\qquad K\to \I ,\\
\label{lowfreq} &\norm{|\nabla |^s\big( e^{-ix\cdot \xi (t)}P_{\lec C_*K}u,\,e^{-2ix\cdot \xi (t)}P_{\lec C_*K}v)}{L^2L^4([0,T)\times \R^4)}\lec K^s,\qquad s>1/2.
\end{align}
For \eqref{highfreq} we notice that $|\xi (t)|\le 2^{-10}C_*K$ for $t\in [0,T)$, which implies
\eqq{\norm{P_{>\frac{C_*K}{4}}(u,v)}{L^2L^4}\lec \norm{P_{|\xi -\xi (t)|>\frac{C_*K}{16}}(u,v)}{L^2L^4}.}
\eqref{lowfreq} is shown as follows:
\eqq{&\norm{|\nabla |^se^{-ix\cdot \xi (t)}P_{\lec C_*K}u}{L^2L^4}=\norm{|\nabla |^sP_{|\xi +\xi (t)|\lec C_*K}e^{-ix\cdot \xi (t)}u}{L^2L^4}\lec \norm{|\nabla |^sP_{\lec 4C_*K}e^{-ix\cdot \xi (t)}u}{L^2L^4}\\
&\lec \sum _{N<C_*K}N^s\norm{P_{>N}e^{-ix\cdot \xi (t)}u}{L^2L^4}=\sum _{N<C_*K}N^s\norm{e^{-ix\cdot \xi (t)}P_{|\xi -\xi (t)|>N}u}{L^2L^4}\\
&\lec \sum _{N<C_*K}N^s\Big( \frac{K}{N}\Big) ^{1/2}\sim K^s,}
and similarly for $v$.

To prove \eqref{im4}, we observe that $\eqref{gomiP}+\eqref{gomiM}$ is invariant under the gauge transform
\eqq{(U,V,F,G)~\mapsto ~(U^*,V^*,F^*,G^*):=\big( e^{-ix\cdot \xi (t)}U,\,e^{-2ix\cdot \xi (t)}V,\, e^{-ix\cdot \xi (t)}F,\,e^{-2ix\cdot \xi (t)}G\big) .}
Hence, by an integration by parts in $x$ we see that
\begin{align}
&\int _0^T\Big( \eqref{gomiP}+\eqref{gomiM}\Big) dt \notag \\
&=2\int _0^T\iint \Theta _L(|A|)A\cdot \Re \Big[ \bbar{F^*}\nabla U^*+\frac{1}{2}\bbar{G^*}\nabla V^*\Big] (t,x)\,\FR{m}[U,V](t,y)\,dx\,dy\,dt\label{gomi1}\\
&\hx +\int _0^T\ti{N}(t)\iint \big[ \theta _L+3\Theta _L\big] (|A|) \Re \Big[ \bbar{F^*}U^*+\frac{1}{2}\bbar{G^*}V^*\Big] (t,x)\,\FR{m}[U,V](t,y)\,dx\,dy\,dt\label{gomi2}\\
&\hx -2\int _0^T\iint \Theta _L(|A|)A\cdot \FR{p}[U^*,V^*](t,x)\,\Im \big[ \bbar{F^*}U^*+\bbar{G^*}V^*\big] (t,y)\,dx\,dy\,dt.\label{gomi3}
\end{align}
Recall that the weight functions satisfy $|\Th_L(|A|)A|\lec L$, $|\th _L+3\Th _L|\lec 1$.

Since $(F^*,G^*)$ has at least one of $e^{-ix\cdot \xi (t)}P_{>C_*K/4}u$ and $e^{-2ix\cdot \xi (t)}P_{>C_*K/4}v$; for instance,
\eqq{F^*&=\bbar{e^{-ix\cdot \xi (t)}P_{>C_*K}u}\,e^{-2ix\cdot \xi (t)}v+\bbar{e^{-ix\cdot \xi (t)}P_{\le C_*K}u}\,e^{-2ix\cdot \xi (t)}P_{>C_*K}v\\
&\hx -P_{|\xi +\xi (t)|>C_*K}\Big[ \bbar{e^{-ix\cdot \xi (t)}P_{>\frac{C_*K}{4}}u}\,e^{-2ix\cdot \xi (t)}v+\bbar{e^{-ix\cdot \xi (t)}P_{\le \frac{C_*K}{4}}u}\,e^{-2ix\cdot \xi (t)}P_{>\frac{C_*K}{4}}v\Big] ,}
we can estimate \eqref{gomi1} and \eqref{gomi2} using \eqref{highfreq}--\eqref{lowfreq} similarly to the radial case.

We focus on the last integral \eqref{gomi3}, which does not have a counterpart in the radial case.
Notice that $\Im [\bbar{F^*}U^*+\bbar{G^*}V^*]$ has two types of products of three functions; two highs and a low,
\eqq{\Big[ ~e^{-ix\cdot \xi (t)}P_{>\frac{C_*K}{4}}u\quad \text{or}\quad e^{-2ix\cdot \xi (t)}P_{>\frac{C_*K}{4}}v~\Big] ^2\cdot \Big[ ~e^{-ix\cdot \xi (t)}P_{\le C_*K}u\quad \text{or}\quad e^{-2ix\cdot \xi (t)}P_{\le C_*K}v~\Big] ,}
or two lows and a high,
\eqq{\Big[ ~e^{-ix\cdot \xi (t)}P_{>\frac{C_*K}{4}}u\quad \text{or}\quad e^{-2ix\cdot \xi (t)}P_{>\frac{C_*K}{4}}v~\Big] \cdot \Big[ ~e^{-ix\cdot \xi (t)}P_{\le C_*K}u\quad \text{or}\quad e^{-2ix\cdot \xi (t)}P_{\le C_*K}v~\Big] ^2.}
The former case is easier to treat.
In fact, we use the H\"older inequality in $(t,y)$ as $L^2L^4\cdot L^2L^4\cdot L^\I L^2$ and apply \eqref{highfreq} to obtain a bound of $o(K)$:
\eqq{\norm{(U^*,V^*)}{L^\I L^2}\norm{\nabla (U^*,V^*)}{L^\I L^2}\norm{\big( e^{-ix\cdot \xi (t)}P_{>\frac{C_*K}{4}}u,\, e^{-2ix\cdot \xi (t)}P_{>\frac{C_*K}{4}}v\big)}{L^2L^4}^2\norm{(U^*,V^*)}{L^\I L^2}.}
For the latter case, we extract one derivative from the high-frequency functions as
\eqq{\big( e^{-ix\cdot \xi (t)}P_{>\frac{C_*K}{4}}u,\, e^{-2ix\cdot \xi (t)}P_{>\frac{C_*K}{4}}v\big)=-\nabla \cdot \nabla (-\Delta )^{-1}\big( e^{-ix\cdot \xi (t)}P_{>\frac{C_*K}{4}}u,\, e^{-2ix\cdot \xi (t)}P_{>\frac{C_*K}{4}}v\big)}
and integrate it by parts (in $y$).
When the derivative moves to one of the low-frequency functions, we obtain a bound like
\eqq{&\norm{(U^*,V^*)}{L^\I L^2}\norm{\nabla (U^*,V^*)}{L^\I L^2}\\
&\times \norm{|\nabla |^{-1}\big( e^{-ix\cdot \xi (t)}P_{>\frac{C_*K}{4}}u,\, e^{-2ix\cdot \xi (t)}P_{>\frac{C_*K}{4}}v\big)}{L^2L^4}\norm{\nabla (U^*,V^*)}{L^2L^4}\norm{(U^*,V^*)}{L^\I L^2},}
which is again $o(K)$ by \eqref{highfreq}--\eqref{lowfreq}.
If the derivative moves to the weight function $\Th _L(|A|)A$, then we have another $\ti{N}(t)$ as \eqref{gomi2}, and
the resulting bound will be \eqq{&\norm{\ti{N}}{L^3_t}\norm{(U^*,V^*)}{L^\I L^2}\norm{\nabla (U^*,V^*)}{L^\I L^2}\norm{|\nabla |^{-1}\big( e^{-ix\cdot \xi (t)}P_{>\frac{C_*K}{4}}u,\, e^{-2ix\cdot \xi (t)}P_{>\frac{C_*K}{4}}v\big)}{L^2L^4}\\
&\times \norm{(U^*,V^*)}{L^\I L^2}\norm{(U^*,V^*)}{L^\I L^3}^{2/3}\norm{(U^*,V^*)}{L^2L^{12}}^{1/3}.}
Similarly to \eqref{gomi2}, the Sobolev embedding and \eqref{highfreq}--\eqref{lowfreq} show that this is also $o(K)$.

This completes the proof of \eqref{im4}, and hence that of Theorem~\ref{thm:im}.
\end{proof}

We finish the proof of Theorem~\ref{thm:main} by showing
\eqq{\sup _{t\in [0,T)}|\Sc{M}(t)|=o(K),\qquad K\to \I ,}
which precludes the quasi-soliton scenario for the non-radial case.
To prove that, we first observe $\Sc{M}(t)$ is also invariant under the gauge transformation $(U,V)\mapsto (U^*,V^*)$, which can be seen 
in the same manner as \eqref{Ndash} is under $(U,V)\mapsto (U_*,V_*)$.
Now, the almost periodicity of $(u,v)$ yields
\eqq{&\norm{\nabla U^*}{L^\I L^2}=\norm{\nabla P_{|\xi +\xi (t)|\le C_*K}e^{-ix\cdot \xi (t)}u}{L^\I L^2}\lec \norm{\nabla P_{\le 4C_*K}e^{-ix\cdot \xi (t)}u}{L^\I L^2}\\
&\lec K^{1/2}\norm{P_{\le K^{1/2}}e^{-ix\cdot \xi (t)}u}{L^\I L^2}+K\norm{P_{>K^{1/2}N(t)}e^{-ix\cdot \xi (t)}u}{L^\I L^2}\\
&=K^{1/2}\norm{P_{\le K^{1/2}}e^{-ix\cdot \xi (t)}u}{L^\I L^2}+K\norm{P_{|\xi -\xi (t)|>K^{1/2}N(t)}u}{L^\I L^2}
=o(K),\qquad K\to \I ,}
and similarly for $V^*$.
An application of the Cauchy-Schwarz finally lead us to the conclusion.


\appendix

\section{Proof of Remark \ref{rem2-3}}\label{AppendixA}
Here, we give the proof of Remark \ref{rem2-3} on the bases of the argument in \cite{MR2470397}.
\begin{lem}\label{appendix-lem5}
For all $\phi \in L_x^2 (\mathbb{R}^d) $ and $(\theta_0 , \xi_0 ,x_0 , \lambda_0 ) \in \mathbb{R}/2\pi \mathbb{Z} \times \mathbb{R}^d \times \mathbb{R}^d \times (0,\infty)$, it follows that
\begin{align*}
\mathcal{F} h(\theta_0 ,\xi_0 ,x_0 , \lambda_0) \phi =h(\theta_0 +\xi_0 x_0 , -x_0, \xi_0 ,\lambda_0^{-1}) \mathcal{F}\phi .
\end{align*}
\end{lem}
\begin{proof}
This follows by changing variables and density argument.
\end{proof}
\begin{defn}
On the basis of above lemma, we set as follows:
\begin{align*}
\hat{h}(\theta_0, \xi_0 ,x_0,\lambda_0) &:=h(\theta_0 +x_0\xi_0 , -x_0 ,\xi_0 ,\lambda_0^{-1}), \\
\check{h}(\theta_0 ,\xi_0 ,x_0 ,\lambda_0) &:=h(\theta_0 +x_0\xi_0 , x_0, -\xi_0 , \lambda_0^{-1}).
\end{align*}
Under this notation, it follows that $\widehat{h\phi} =\hat{h}\hat{\phi} , \hat{\check{h}}=\check{\hat{h}} =h$.
\end{defn}
\begin{lem}\label{appendix-lem6}
For any $\{g_n\}_n \subset G$, precisely one of the following statements holds:\\
1.\ $g_n \rightarrow 0$\quad in WOT.\\
2.\ $g_n \rightarrow g$\quad  in SOT for some $g\in G$ after passing to a subsequence if necessary.
\end{lem}
\begin{proof}
Let $g_n =g(\theta_n , \xi_n x_n ,\lambda_n)$. First we consider the case any one of $\lambda_n , \lambda_n^{-1}, |\xi_n| , |x_n| $ converges to infinity as $n\rightarrow \infty $.
Take  subsequence of $\{g_n\}$ arbitrarily and use same symbol $\{g_n\}$.\\
Case1-1. $\lambda_n \rightarrow \infty. $
Then for any $\phi ,\psi \in C_0^{\infty}(\mathbb{R}^d)$, we can calculate as follows:
\begin{align*}
&|(h(\theta_n ,\xi_n ,x_n ,\lambda_n)\phi ,\psi)_{L^2}| \leq (\lambda_n)^{-\frac{d}{2}} \|\phi\|_{\infty} \|\psi\|_1 \rightarrow 0 \ as \ n\rightarrow \infty , \\
&|h(2\theta_n ,2\xi_n ,x_n ,\lambda_n \phi ,\psi)_{L^2} |\leq (2\lambda_n)^{-\frac{d}{2}}\|\phi\|_{\infty} \|\psi\|_{1} \rightarrow 0 \ as \ n\rightarrow \infty.
\end{align*}
By density argument we get $g_n \rightarrow 0 \ in\ WOT$.\\
Case1-2. \ $\lambda_n^{-1}\rightarrow \infty $.  This case is similar to Case1-1.(Note that $g_n $ is unitary operator.)\\
Case1-3. $\lambda_n$ and $\lambda_n^{-1} \nrightarrow \infty $ and $|x_n| \rightarrow \infty$. 
In this case , passing to a subsequence if necessary we  may assume $\lambda_n \rightarrow \lambda_0 $ for some $\lambda_0 \in (0,\infty)$.
Then it follows that
\begin{align}
&|(h(\theta_n ,\xi_n ,x_n ,\lambda_n)\phi ,\psi)_{L^2})|\leq \lambda_n^{-\frac{d}{2}} \int_{\mathbb{R}^d} |\phi (\frac{x-x_n}{\lambda_n})||\psi (x)|\ dx \rightarrow 0 \ as\ n\rightarrow \infty ,\\
&|(h(2\theta_n ,2\xi_n ,x_n ,\lambda_n)\phi ,\psi)_{L^2})|\leq \lambda_n^{-\frac{d}{2}} \int_{\mathbb{R}^d} |\phi (\frac{x-x_n}{\lambda_n})||\psi (x)|\ dx \rightarrow 0\ as\ n\rightarrow \infty,
\end{align}
for any $\phi ,\psi \in C_0^{\infty}(\mathbb{R}^d)$. We can get the result by density argument.\\
Case1-4. $\lambda_n$ and $\lambda_n^{-1} \nrightarrow \infty$ and $|\xi_n| \rightarrow \infty $. This case can be treated as Case1-3. (Use Plancherel's theorem and dominated convergence theorem.)\\
Case2. Finally we consider the case each parameters are bounded. Passing to subsequences if necessary, we may assume that
\begin{align*}
&\lambda_n \rightarrow \lambda_0\in (0,\infty),\\
&\xi_n \rightarrow \xi_0 \in \mathbb{R}^d,\\
&x_n \rightarrow x_0 \in \mathbb{R}^d,\\
&\theta_n \rightarrow \theta_0 \in \mathbb{R}/2\pi \mathbb{Z}.
\end{align*}
Then it follows easily that $g_n \rightarrow g( \theta_0 , \xi_0 ,x_0 , \lambda_0 ) \ in\  SOT.$
\end{proof}
\begin{lem}\label{appendix-lem7}
For each $(\phi,\psi)\in L^2(\mathbb{R}^d)^2  $, $G(\phi,\psi) $ is closed in $L^2 (\mathbb{R}^d)^2$.
\end{lem}
\begin{proof}
If $(\phi,\psi) =(0,0)$, the result follows easily. So we assume $(\phi,\psi) \neq (0,0)$. Take any $\{g_n\} \subset G$ and $(\psi_0,\phi_0)\in L^2 (\mathbb{R}^d)^2 $ satisfying 
$g_n (\phi, \psi) \rightarrow (\phi_0 ,\psi_0) $ in $L^2 (\mathbb{R}^d)^2$. Then we get $0<\|(\phi ,\psi)\|_{(L^2)^2} =\|(\phi_0,\psi_0)\|_{(L^2)^2}$. Therefore by Lemma \ref{appendix-lem6}, passing to a subsequence if necessary,
we establish
\begin{align*}
g_n \rightarrow g_0 \ in \ SOT\ for \ some \ g_0 \in G.
\end{align*}
Then it follows that $g_n (\phi ,\psi) \rightarrow g_0 (\phi ,\psi)$ and so $(\phi_0 ,\psi_0)=g_0 (\phi ,\psi ) \in G(\phi ,\psi)$.
\end{proof}
\begin{defn}
We denote by $\mathcal{O}$ the quotient topology of $G\setminus L^2 (\mathbb{R}^d)^2$. Let $\pi :L^2(\mathbb{R}^d)^2 \rightarrow G\setminus L^2 (\mathbb{R}^d)^2$ be the canonical projection.
\end{defn}
\begin{lem}\label{appendix-lem8}
We define the metric $d$ on $G\setminus L^2 (\mathbb{R}^d)^2$ as follows:
\begin{align*}
d(\pi (\phi_1,\psi_1) ,\pi (\phi_2, \psi_2)) :=\inf_{g\in G} \|g(\phi_1,\psi_1) -(\phi_2,\psi_2)\|_{L^2(\mathbb{R}^d)^2},
\end{align*} 
where $(\phi_1,\psi_1), (\phi_2 ,\psi_2) \in L^2 (\mathbb{R}^d)^2$. Then $\mathcal{O} =\mathcal{O}_d$, where 
$\mathcal{O}_d $ denotes the topology defined by the metric $d$.
\end{lem}
\begin{proof}
Note that  $d$ is really a metric on $G\setminus L^2 (\mathbb{R}^d)^2$ by Lemma \ref{appendix-lem7}.
Take any $(\phi ,\psi )\in G \setminus L^2 (\mathbb{R}^d)^2$ and $\varepsilon >0$. Then $B((\phi ,\psi) ,\varepsilon) \in \mathcal{O}_{L^2(\mathbb{R}^d)^2}$ and 
$\pi [B((\phi, \psi), \varepsilon)] =B((\pi (\phi ,\psi)) ,\varepsilon)$. Therefore $\pi \colon (L^2(\mathbb{R}^d)^2 ,\mathcal{O}_{L^2 (\mathbb{R}^d)^2}) \rightarrow (G \setminus L^2 (\mathbb{R}^d)^2, \mathcal{O}_d )$
is continuous. By the definition of the quotient topology, we get $\mathcal{O}_d \subset \mathcal{O}$. Let $U\in \mathcal{O}$ and take any $(\phi ,\psi )\in L^2 (\mathbb{R}^d)^2$ with $\pi (\phi ,\psi) \in U$.
Then $(\phi ,\psi) \in \pi^{-1}[U] \in L^2 (\mathbb{R}^d)$, and so there exists $\varepsilon >0$ such that $B((\phi ,\psi) ,\varepsilon) \subset \pi^{-1}[U]$. Therefore we get $\pi [B((\phi ,\psi) \varepsilon)] =B(\pi (\phi ,\psi), \varepsilon) \subset
U$. This implies $U\in \mathcal{O}_{d}$.
\end{proof}
\begin{rem}
Let $(\phi_n ,\psi_n) \subset L^2 (\mathbb{R})^2\, and\, (\phi_0 ,\psi_0) \in L^2 (\mathbb{R}^d)^2$. Then followings are equivalent:\\
1.\, $\pi (\phi_n ,\psi_n) \rightarrow \pi (\phi_0 ,\psi_0) \ in\ G\setminus L^2(\mathbb{R}^d)^2$ \\
2.\, There exists $\{g_n\} \subset G$ such that $g_n (\phi_n ,\psi_n) \rightarrow (\phi_0 ,\psi_0) \ in \ L^2(\mathbb{R}^d)^2$.
\end{rem}
\begin{proof}
$1\Rightarrow 2$.\  Suppose $\pi (\phi_n ,\psi_n) \rightarrow \pi (\phi_0,\psi_0).$
For each $n\in \mathbb{N}$, there exists $g_n \in G$ such that $\|g_n (\phi_n ,\psi_n) -(\phi_0 , \psi_0) \|_{L^2 (\mathbb{R}^d)^2} < \inf_{g\in G} \|g(\phi_n ,\psi_n) -(\phi_0 ,\psi_0)\|_{L^2 (\mathbb{R}^d)^2} +\frac{1}{n} \rightarrow 0 $ as $n\rightarrow \infty$.\\
$2\Rightarrow 1$. This is clear and so we omit the proof.
\end{proof}
\begin{lem}
Let $K \subset G\setminus L^2 (\mathbb{R}^d)^2$ be precompact. Assume also that
\begin{align}\label{ap-lem8-1}
\exists  \ \eta >0 \ s.t.\ \forall (\phi,\psi) \in \pi ^{-1}[K] ,\ \eta =\|(\phi ,\psi)\|_{L^2 (\mathbb{R}^d)^2}.
\end{align}
Then there exists a precompact set $\tilde{K} \subset L^2 (\mathbb{R}^d)^2$ such that $\pi [\tilde{K}] =K.$
\end{lem}
\begin{proof}
First we prove by contradiction that
\begin{align*}
\exists \varepsilon >0 \ &s.t\ \forall p \in K ,\ \exists  f(p) \in \pi^{-1}[\{p\}] \ s.t.\\
&\min \{\|f(p) \|_{L^2 (B(0,1))^2} ,\ \|\widehat{f(p)}\|_{L^2 (B(0,1))^2}\} \geq \varepsilon .
\end{align*}
If not, for each $n\in \mathbb{N}$ there exists $(\phi_n ,\psi_n) \in \pi^{-1}[K] $ such that
\begin{align}
\sup_{g\in G}\min\{ \|g(\phi _n, \psi_n )\|_{L^2 (B(0,1))^2},
\|\mathcal{F} [g(\phi_n ,\psi_n)]\|_{L^2 (B(0,1))^2} \}
\leq \frac{1}{n}.
\end{align}
By precompactness of $K$ and $\phi_n \in \pi^{-1}[K]$, passing to a subsequence if necessary, there exist $g_n \in G$ and $(\phi , \psi ) \in L^2 (\mathbb{R}^d)^2$ such that
\begin{align}\label{ap-lem8-2}
g_n (\phi_n ,\psi_n ) \rightarrow (\phi, \psi ) \ in \ L^2 (\mathbb{R}^d)^2.
\end{align}
Take $g\in G$ arbitrarily. Then it follows that
\begin{align*}
\min\{&\|g(\phi ,\psi)\|_{L^2 (B(0,1))^2} , \|\mathcal{F}[g(\phi ,\psi)]\|_{L^2 (B(0,1))^2}\}\\
&\leq \|(\phi ,\psi ) -g_n (\phi_n ,\psi _n)\|_{L^2(\mathbb{R}^d)^2} +\min
\{ \|gg_n(\phi _n, \psi_n )\|_{L^2 (B(0,1))^2}, 
\|\mathcal{F} [gg_n(\phi_n ,\psi_n)]\|_{L^2 (B(0,1))^2} \}\\
&\rightarrow 0\ as \ n\rightarrow \infty.
\end{align*}
Therefore we obtain
\begin{align*}
g(\phi ,\psi)=(0,0) \quad\text{or} \quad  \mathcal{F}[g(\phi, \psi)]=0\quad \text{for any}\quad g\in G,
\end{align*}
and so $(\phi ,\psi)=(0,0)$. This contradicts \eqref{ap-lem8-1}.
Next we set $\tilde{K} := \{f(p) |  p \in K\}$ and prove this satisfies the result. It is clear that $\pi [\tilde{K}]=K.$
 To prove precompactness of $\tilde{K}$, take $\{(\phi_n ,\psi_n)\}\subset \tilde{K} $ arbitrarily. Since $K$ is precompact, passing to a subsequence if necessary, there exists $g_n = g(\theta_n ,\xi_n , x_n ,\lambda_n) \in G$ 
and $(\phi ,\psi) \in L^2 (\mathbb{R}^d)^2$ such that $g_n (\phi_n ,\psi_n) \rightarrow (\phi ,\psi) $ in $L^2 (\mathbb{R}^d)^2$.
We prove by contradiction that $\lambda_n $, $\lambda_n^{-1}$, $\xi_n$ and $x_n$ are bounded.\\
Case1.  Suppose $\lambda_n^{-1}$ is unbounded. Passing to a subsequence if necessary, we may assume $\lambda_n \rightarrow 0$.
Then it follows that
\begin{align*}
\|(\phi_n ,\psi_n)\|_{L^2 (B(0,1))^2} 
&=\|g_n (\phi_n ,\psi_n)\|_{L^2 (B(x_n ,\lambda_n))^2 } \\
&\leq \|g_n (\phi_n ,\psi_n) -(\phi ,\psi)\|_{L^2(\mathbb{R}^d)^2} +\|(\phi, \psi)\|_{L^2 (B(x_n ,\lambda_n))^2} \rightarrow 0.
\end{align*}
This contradicts the definition of $\tilde{K}$.\\
Case2. Suppose $\lambda_n$ is unbounded. Then passing to a subsequence if necessary, we may assume that 
$\lambda_n \rightarrow \infty$. Then it follows that
\begin{align*}
\|\hat{ \phi}_n \|_{L^2 (B(0,1))} &
=\|\hat{h} (\theta_n , \xi_n , x_n , \lambda_n) \hat{\phi}_n\|_{L^2 (B(\xi_n , \lambda_n^{-1}))}\\
&\leq \|h(\theta_n , \xi_n , x_n , \lambda_n) \phi_n -\phi\|_{L^2 (\mathbb{R}^d)} +\|\hat{\phi}\|_{L^2 (B(\xi_n , \lambda_n^{-1}))} \rightarrow 0, \\
\|\hat{\psi}_n\|_{L^2 (B(0,1))} &
=\|\hat{h}(2\theta_n , 2\xi_n ,x_n ,\lambda_n)\hat{\psi}_n\|_{L^2 (B(2\xi_n , \lambda_n^{-1}))} \\
&\leq \| h(2\theta_n , 2\xi_n ,x_n ,\lambda_n)\psi_n -\psi \|_{L^2 (\mathbb{R}^d)}+\|\hat{\psi}\|_{L^2 (B(2\xi_n ,\lambda_n^{-1}))} \rightarrow 0 .
 \end{align*}
This contradicts the definition of $\tilde{K}$. Therefore $\lambda_n $ and $\lambda_n^{-1}$ are bounded.\\
Case3. Suppose $x_n$ is unbounded or $\xi_n $ is unbounded. Passing to a subsequence if necessary, we may assume $|x_n| \rightarrow \infty$ or $|\xi_n| \rightarrow \infty $. Then we get the contradiction by the same argument.\\
Therefore all parameters are bounded, and so we may assume 
\begin{align*}
g_n^{-1} \rightarrow g_0\quad \text{in SOT  for  some $g_0 \in G$.}
\end{align*}
Then we establish that
\begin{align*}
(\phi_n ,\psi_n) =g_n^{-1} g_n (\phi_n ,\psi_n) \rightarrow g_0 (\phi, \psi) \quad \text{in} \  L^2 (\mathbb{R}^d)^2. 
\end{align*}
\begin{cor}
Let $(u,v):I\times \mathbb{R}^d \colon \rightarrow \mathbb{C}^2$ be a  nonzero solution to NLS. Assume $Gu[I] :=\{Gu(t) \mid t\in I\}$ is precompact in $G\setminus L^2 (\mathbb{R}^d)^2$. 
Then there exists $g\in \text{Map}(I,G) $ such that $\{g(t)u(t) \mid t\in I\}$ is precompact in $L^2 (\mathbb{R}^d)^2$.
\end{cor}
\end{proof}


\section{Proof of the profile decomposition in the radial case}\label{AppendixB}

\begin{lem}\label{lem1-1}
Let $\{u_n\}_{L^2 (\mathbb{R}^d)}, u\in L^2 (\mathbb{R}^d)$. Then, following three conditions are equivalent.\\
1.\ $u_n \rightharpoonup u \ weakly\ in\ L^2 (\mathbb{R}^d)$.\\
2.\ $e^{i\alpha t\Delta } u_n \rightharpoonup e^{i\alpha t \Delta } u \ weakly\ in \ L^{\frac{2(d+2)}{d}} (\mathbb{R}^{1+d})$\text{ for some $\alpha>0$}.\\
3.\ $e^{i\alpha t\Delta } u_n \rightharpoonup e^{i\alpha t\Delta } u \ weakly\ in \ L^{\frac{2(d+2)}{d}} (\mathbb{R}^{1+d})$\text{ for any $\alpha>0$}.
\end{lem}
\begin{proof}
See for instance Lemma 3.63 in \cite{MR1628235}.
\end{proof}
\begin{lem}\label{sot-lemma}Fix $\kappa >0$. Then for any $\{g_n = g_{\kappa}(\theta_n , \xi_n , x_n , \lambda_n) \}\subset G_{\kappa}$ and $\{t_n\} \subset \mathbb{R}$, following are equivalent.\\
1.\  $g_n U_{\kappa}(t_n) \rightarrow 0\quad \text{in WOT}$\\
2.\ $g_n U_{\kappa} (t_n) \rightarrow g U_{\kappa}(t_0) \quad \text{in SOT for some $g \in G_{\kappa}$ and $t_0 \in \mathbb{R}$ passing to a subsequence if necessary.}$
\end{lem}
\begin{proof}
This follows from same argument in the proof of Lemma \ref{appendix-lem6}.
\end{proof}

\begin{lem}\label{ap-lem4}
Let $\{(u_n , v_n)\}$ be a bounded sequence of $L^2 (\mathbb{R}^d)^2$, and let
\begin{align*}
\begin{pmatrix}
u_n \\
v_n
\end{pmatrix}
=\sum_{j=1}^{J} g_n^j U_{\kappa}(t_n^j) 
\begin{pmatrix}
\phi^j \\
\psi^j
\end{pmatrix}
+W_n^J
\end{align*}
be a profile decomposition given in Theorem \ref{profile}, where $U_{\kappa}(t)=(e^{it\Delta} , e^{i \kappa t\Delta})$. Assume also that $[g_n U_{\kappa}(t_n)]^{-1}(u_n ,v_n) \rightharpoonup (\phi ,\psi)=:\Psi$
weakly in $L^2 \times L^2 $ for some $\{g_n\} \subset G_{\kappa} ,\ \{t_n\} \subset \mathbb{R} $ and $(\phi ,\psi)\in L^2 \times L^2\setminus \{(0,0)\}$. Then, after passing to a subsequence if necessary, there exists a unique $j_0 \leq J^{\ast}$ such that
\begin{align*}
[g_n U_{\kappa}(t_n)]^{-1} g_n^{j_0} U_{\kappa}(t_n^{j_0}) \rightarrow gU_{\kappa}(t_0)\ in\ SOT\ for \ some\ g\in G_{\kappa} \ and\ t_0\in \mathbb{R} 
\end{align*}
and
\begin{align*}
(\phi ,\psi) =gU_{\kappa}(t_0) (\phi^{j_0} , \psi^{j_0})
\end{align*}
\end{lem}
\begin{proof}
If not, by Lemma \ref{sot-lemma} we have
\begin{align*}
[g_n U_{\kappa}(t_n)]^{-1} g_n^{j} U_{\kappa}(t_n^j) \rightarrow 0\ in \ WOT\ for\ all\ j\leq J^{\ast}.
\end{align*}
Then it follows from profile decomposition that
\begin{align*}
[g_n U_{\kappa}(t_n)]^{-1} W_n^{\ell} \rightharpoonup \Psi \ for\ all\ \ell \leq J^{\ast}.
\end{align*}
Applying Lemma \ref{lem1-1} and using weak lower semicontinuity of $L^{\frac{2(d+2)}{d}} (\mathbb{R}^{1+d})$ norm , we obtain that
\begin{align*}
0<\|U_{\kappa}(t)\Psi \|_{L_{t,x}^{\frac{2(d+2)}{d}} (\mathbb{R}^{1+d})^2} 
\leq \liminf_{n\rightarrow \infty} \|U_{\kappa}(t)[g_n U_{\kappa}(t_n)]^{-1} W_n^{\ell}\|_{L_{t,x}^{\frac{2(d+2)}{d}} (\mathbb{R}^{1+d})^2}
=\liminf_{n\rightarrow \infty} \|U_{\kappa}(t)W_n^{\ell}\|_{L_{t,x}^{\frac{2(d+2)}{d}} (\mathbb{R}^{1+d})^2}.
\end{align*}
This contradicts the property of $W_n^{\ell}$. Finally we show the uniqueness. If $j_1 \neq j_0$, then we have by asymptotic orthogonality that
\begin{align*}
[g_n U_{\kappa}(t_n)]^{-1} g_n^{j_1} U_{\kappa}(t_n^{j_1}) 
= [g_n U_{\kappa}(t_n)]^{-1} (g_n^{j_0} U_{\kappa}(t_n^{j_0})) U_{\kappa}(-t_n^{j_0})( g_n^{j_0})^{-1} g_n^{j_1} U_{\kappa}(t_n^{j_1}) \rightarrow 0 \ in\ WOT.
\end{align*}
\end{proof}
\begin{proof}[Proof of Theorem \ref{radial-profile}]
1. First we show by contradiction that all $\{\lambda_n^{j} \xi_n^j\}_{n}$ and $\{(\lambda_n^j)^{-1} x_n^j -2t_n^{j} \lambda_n^{j} \xi_n^{j}\}_n$ are bounded. If not , there exists 
$j_0 \leq J^{\ast}$ such that 
\begin{align*}
|\lambda_n^{j_0} \xi_n^{j_0}| +|(\lambda_n^{j_0})^{-1} x_n^{j_0} -2t_n^{j_0} \lambda_n^{j_0} \xi_n^{j_0}| \rightarrow \infty.
\end{align*}
 Since $d\geq 2$ , we can take a sequence 
$\{A_{\ell}\}_{\ell} \subset SO(d)$ such that 
\begin{align*}
|(A_{\ell} -A_{m})[\lambda_n^{j_0} \xi_n^{j_0}]| +|(A_{\ell} -A_{m})[(\lambda_n^{j_0})^{-1} x_n^{j_0} -2t_n^{j_0} \lambda_n^{j_0} \xi_n^{j_0}]| \rightarrow \infty \ for\ \ell \neq m .
\end{align*}
Then we have $\{G_n^{\ell} := g_{\kappa}(\theta_n^{j_0} , A_{\ell} \xi_n^{j_0} , A_{\ell} x_n^{j_0} , \lambda_n^{j_0})U_{\kappa}(t_n^{j_0})\}$ is asymptotically orthogonal. On the other hand, from radial property of 
$\{(u_n ,v_n)\}$ and a profile decomposition we have 
\footnote{
Set $A(\phi ,\psi)= (\phi (A\cdot ) , \psi (A\cdot))$ for $(\phi ,\psi ) \in L^2 (\mathbb{R}^d)^2.$}
\begin{align*}
[G_n^{\ell}]^{-1}(u_n, v_n) 
&= U_{\kappa}(-t_n^{j_0}) g_{\kappa}(\theta_n^{j_0} , A_{\ell}\xi_n^{j_0}, A_{\ell} x_n^{j_0} , \lambda_n^{j_0})^{-1} A_{\ell} (u_n, v_n)\\
&=U_{\kappa}(-t_n^{j_0}) A_{\ell} (g_n^{j_0})^{-1} (u_n. v_n) \\
&=A_{\ell} U_{\kappa}(-t_n^{j_0}) (g_n^{j_0})^{-1} (u_n ,v_n) \rightharpoonup A_{\ell} (\phi^{j_0} , \psi^{j_0}) \ weakly\ in\ L^2\times L^2.
\end{align*}
Therefore by Lemma \ref{ap-lem4}, there exist $j_{\ell} ,\ g^{\ell} \in G_{\kappa}$ and $t^{\ell}$ such that
\begin{align*}
[G_n^{\ell}]^{-1} g^{j_{\ell}} U_{\kappa}(t_n^{j_{\ell}}) \rightarrow g^{\ell}U_{\kappa}(t^{\ell}) \ in\ SOT\ and \ A_{\ell} (\phi^{j_0} ,\psi^{j_0}) =g^{\ell} U_{\kappa}(t^{\ell}) (\phi^{\ell} ,\psi^{\ell}).
\end{align*}
Noting that $[g_n^{j_{\ell}} U_{\kappa}(t_n^{j_{\ell}})]^{-1} g_n^{j_m}U_{\kappa}(t_n^{j_m}) = [[G_n^{\ell}]^{-1}g_n^{j_{\ell}}U_{\kappa}(t_n^{j_{\ell}}) ]^{-1} [[G_n^{\ell}]^{-1} G_n^{m}][ [G_n^m]^{-1}
g_n^{j_m} U_{\kappa}(t_n^{j_m})]\rightarrow 0 \ in \ WOT$ for $\ell \neq m$, we obtain $j_{\ell} \neq j_{m} $ for $\ell \neq m$. Therefore we obtain that
\begin{align*}
\liminf_{n\rightarrow \infty} \|(u_n ,v_n)\|_{L^2 \times L^2}^2 \geq \sum_{\ell=1}^{\infty} \|(\phi^{\ell},\psi^{\ell})\|_{L^2 \times L^2 }^2 =\sum_{\ell=1}^{\infty} \|(\phi^{j_0} ,\psi^{j_0})\|_{L^2 \times L^2 }^2 =\infty.
\end{align*}
This  contradicts boundedness of $\{(u_n ,v_n)\}$. Therefore passing to a subsequence if necessary, we may assume that
\begin{align*}
g_{\kappa}(t_n^{j}|\lambda_n^j \xi_n^j|^2 , \lambda_n^j \xi_n^j , (\lambda_n^j)^{-1} x_n^j -2t_n^j \lambda_n^j \xi_n^j ,1) \rightarrow h_j \ in\ SOT\ for \ some\ h_j\in G_{\kappa}.
\end{align*}
Noting that 
\begin{align*}
g_{\kappa}(\theta_n^j ,\xi_n^j , \lambda_n^j ,x_n^j) U_{\kappa}(t_n^j) 
&=g_{\kappa}(\theta_n^j ,0,0 , \lambda_n^j) g_{\kappa}(0, \lambda_n^j \xi_n^j , (\lambda_n^j)^{-1}x_n^j ,1) U_{\kappa}(t_n^j)\\
&=g_{\kappa }(\theta_n^j , 0,0, \lambda_n^j) U_{\kappa}(t_n^j) g_{\kappa}(t_n^{j}|\lambda_n^j \xi_n^j|^2 , \lambda_n^j \xi_n^j , (\lambda_n^j)^{-1} x_n^j -2t_n^j \lambda_n^j \xi_n^j ,1),
\end{align*}
we may assume that scaling parameters and translation parameters are zero (after modifying remainder terms).\\
2. Next we prove that $W_n^j$ and $(\phi^j,\psi^j)$ are radially symmetric. Let $A\in SO(d)$. Then from the profile decomposition, we obtain that
\begin{align}\label{lemc:eq1}
\sum_{j=1}^{J} g_n^j U_{\kappa}(t_n^j) (\phi^j ,\psi^j) +W_n^J =\sum_{j=1}^J g_n^j U_{\kappa}(t_n^j) A(\phi^j ,\psi^j) +AW_n^J,
\end{align}
and so for $1\leq \ell \leq J$, we have
\begin{align*}
\sum_{j=1}^{J} [g_n^{\ell} U_{\kappa}(t_n^{\ell})]^{-1} g_n^j U_{\kappa}(t_n^j) (\phi^j ,\psi^j) +[g_n^{\ell} U_{\kappa}(t_n^{\ell})]^{-1} W_n^J
=\sum_{j=1}^J [g_n^{\ell} U_{\kappa}(t_n^{\ell})]^{-1} g_n^j U_{\kappa}(t_n^j) A(\phi^j ,\psi^j) +A[g_n^{\ell} U_{\kappa}(t_n^{\ell})]^{-1} W_n^J
\end{align*} 
Taking a weak limit in above equation, we establish that 
\begin{align*}
(\phi^{\ell} ,\psi^{\ell}) =A(\phi^{\ell} ,\psi^{\ell}).
\end{align*}
From the equation \eqref{lemc:eq1}, we also have $W_n^{J}=AW_n^J$.
\end{proof}


\section{Proof of lemmas in Section~\ref{sec3.5}}\label{AppendixC}

\begin{proof}[Proof of Lemma~\ref{lem:localconstancy}]
We may assume that $(0<)\,S_J(u,v)\le \e _1$ with $\e_1$ sufficiently small, by cutting $J$ into finite number of small intervals.
Take arbitrary $t_1,t_2\in J$.
By the almost periodicity, we have
\eq{locconst1}{\int _{|x-x(t_2)|\le \frac{R}{N(t_2)}}\Big( |u(t_2,x)|^2+|v(t_2,x)|^2\Big) \,dx \ge \frac{9}{10}M(u,v)}
for sufficiently large $R>0$ which is independent of $t_1,t_2$.
We divide the left hand side of the above inequality with respect to frequency.
Using the H\"older inequality and the Hausdorff-Young inequality,
\eqq{&\int _{|x-x(t_2)|\le \frac{R}{N(t_2)}}|P_{|\xi -\xi (t_1)|\le RN(t_1)}u(t_2,x)|^2\,dx\lec \Big( \frac{R}{N(t_2)}\Big) ^4\norm{P_{|\xi -\xi (t_1)|\le RN(t_1)}u(t_2)}{L^\I (\R^4)}^2\\
&\le \Big( \frac{R}{N(t_2)}\Big) ^4\norm{\hhat{u}(t_2)}{L^1(\shugo{|\xi -\xi (t_1)|\le 2RN(t_1)})}^2\le \Big( \frac{R}{N(t_2)}\Big) ^4\big( RN(t_1)\big) ^4 \norm{\hhat{u}(t_2)}{L^2(\R ^4)}^2.}
Together with a similar estimate for $v$ (replacing $\xi (t_1)$ with $2\xi (t_1)$), we have
\eq{locconst2}{&\int _{|x-x(t_2)|\le \frac{R}{N(t_2)}}\Big( |P_{|\xi -\xi (t_1)|\le RN(t_1)}u(t_2,x)|^2+|P_{|\xi -2\xi (t_1)|\le RN(t_1)}v(t_2,x)|^2\Big) \,dx\\
&\lec R^8 \Big( \frac{N(t_1)}{N(t_2)}\Big) ^4M(u,v).}
On the other hand, we use the equations (write $P_u:=P_{|\xi -\xi (t_1)|>RN(t_1)}$ and $P_v:=P_{|\xi -2\xi (t_1)|>RN(t_1)}$)
\eqq{&P_uu(t_2)=e^{i(t_2-t_1)\Delta}P_uu(t_1)-iP_u\int _{t_1}^{t_2}e^{i(t_2-t')\Delta}(\bbar{u(t')}v(t'))\,dt',\\
&P_vv(t_2)=e^{i(t_2-t_1)\kappa \Delta}P_vv(t_1)-iP_v\int _{t_1}^{t_2}e^{i(t_2-t')\kappa \Delta}(u(t')^2)\,dt',}
the Strichartz estimates, and the almost periodicity again to obtain
\eq{locconst3}{&\int _{\R ^4}\Big( |P_uu(t_2,x)|^2+|P_vv(t_2,x)|^2\Big) \,dx\le 2\int _{\R ^4}\Big( |P_uu(t_1,x)|^2+|P_vv(t_1,x)|^2\Big) \,dx+C\e _1^{4/3}\\
&\le \frac{1}{10}M(u,v)+C\e _1^{4/3}}
for $R$ sufficiently large.
From \eqref{locconst1}, \eqref{locconst2}, \eqref{locconst3}, and the fact that $M(u,v)>0$, we obtain the estimate
\eqq{\Big( \frac{N(t_1)}{N(t_2)}\Big) ^4\gec R^{-8}}
whenever $R\gg 1$ and $\e _1\ll 1$, which implies the claim.
\end{proof}

\begin{proof}[Proof of Lemma~\ref{lem:movement-xi}]
We may assume again that $(0<)\,S_J(u,v)\le \e _1$ with $\e_1$ sufficiently small.
Take $t_1,t_2\in J$ arbitrarily.
Using the almost periodicity we have
\eqq{&\norm{P_{|\xi -\xi (t_1)|\le RN(J)}u(t_1)}{L^2}^2+\norm{P_{|\xi -2\xi (t_1)|\le RN(J)}v(t_1)}{L^2}^2\ge \frac{9}{10}M(u,v),\\
&\norm{P_{|\xi -\xi (t_2)|\le RN(J)}u(t_2)}{L^2}^2+\norm{P_{|\xi -2\xi (t_2)|\le RN(J)}v(t_2)}{L^2}^2\ge \frac{9}{10}M(u,v)}
for sufficiently large $R>0$ independent of $t_1,t_2$, where $N(J):=\sup _{t\in J}N(t)$.
But from the Duhamel formula and the Strichartz estimates we see that
\eqq{&\norm{P_{|\xi -\xi (t_2)|\le RN(J)}u(t_2)}{L^2}+\norm{P_{|\xi -2\xi (t_2)|\le RN(J)}v(t_2)}{L^2}\\
&\le \norm{P_{|\xi -\xi (t_2)|\le RN(J)}u(t_1)}{L^2}+\norm{P_{|\xi -2\xi (t_2)|\le RN(J)}v(t_1)}{L^2}+C\e _1^{2/3},}
hence for sufficiently small $\e _1>0$,
\eqq{\norm{P_{|\xi -\xi (t_2)|\le RN(J)}u(t_1)}{L^2}^2+\norm{P_{|\xi -2\xi (t_2)|\le RN(J)}v(t_1)}{L^2}^2\ge \frac{4}{5}M(u,v).}

Now assume that $|\xi (t_1)-\xi (t_2)|>10RN(J)$, which would imply that $\shugo{|\xi -\xi (t_1)|\le 2RN(J)}\cap \shugo{|\xi -\xi (t_2)|\le 2RN(J)}=\emptyset$ and similarly for $2\xi (\cdot )$, then we had a contradiction as follows,
\eqq{M(u,v)\ge &\norm{P_{|\xi -\xi (t_1)|\le RN(J)}u(t_1)}{L^2}^2+\norm{P_{|\xi -2\xi (t_1)|\le RN(J)}v(t_1)}{L^2}^2\\
&+\norm{P_{|\xi -\xi (t_2)|\le RN(J)}u(t_1)}{L^2}^2+\norm{P_{|\xi -2\xi (t_2)|\le RN(J)}v(t_1)}{L^2}^2\\
\ge &\frac{17}{10}M(u,v).}
Therefore, we have $|\xi (t_1)-\xi (t_2)|\le 10RN(J)$.
\end{proof}

\begin{proof}[Proof of Lemma~\ref{lem:tightness-Str}]
By the Duhamel formula and the Strichartz estimates, we have
\eqq{\norm{u}{L^2_tL^4_x(J\times \R^4)}+\norm{v}{L^2_tL^4_x(J\times \R^4)}\lec _{u,v}1+S_J(u,v)^{2/3}.}
The desired estimate follows from an interpolation between the above and \eqref{tightness} if $S_J(u,v)\le 1$.
When $S_J(u,v)>1$, we first divide $J$ into $O(S_J(u,v))$ subintervals $\shugo{J_k}$ so that $S_{J_k}(u,v)\sim 1$ on each $J_k$, and sum up the obtained estimates on $J_k$.
\end{proof}

\begin{proof}[Proof of Lemma~\ref{lem:Strichartz-SJ}]
The almost periodicity with the non-zero assumption gives an $R=R(u,v)>0$ satisfying
\eqq{\inf _{t\in J}\int _{|x-x(t)|\le \frac{R}{N(t)}}\big( |u(t,x)|^2+|v(t,x)|^2\big) \,dx\gec _{u,v}1.}
Applying the H\"older inequality to the left hand side, we have
\eqq{1\lec _{u,v}\Big( \frac{R}{N(t)}\Big) ^{4/3}\Big( \int _{\R^4}\big( |u(t,x)|^3+|v(t,x)|^3\big) \,dx\Big) ^{2/3}}
for any $t\in J$.
The first inequality in \eqref{Strichartz-SJ} then follows after an integration in $t$.

For the second inequality in \eqref{Strichartz-SJ}, we may focus on the case $S_J(u,v)>1$.
Applying Lemma~\ref{lem:tightness-Str} with $\eta =\frac{S_J(u,v)/100}{1+S_J(u,v)}$ ($\sim 1$), we see that there exists $R=R(u,v)>0$ satisfying
\eqq{S_J(u,v)\lec \norm{P_{|\xi -\xi (t)|\le RN(t)}u}{L^3(J\times \R^4 )}^3+\norm{P_{|\xi -2\xi (t)|\le RN(t)}v}{L^3(J\times \R^4 )}^3,}
which is, via the Hausdorff-Young inequality followed by the H\"older, bounded by
\eqq{\int _J\Big( \big( RN(t)\big) ^{2/3}\big( \norm{u(t)}{L^2}+\norm{v(t)}{L^2}\big) \Big) ^3\,dt\lec _{u,v}\int _JN(t)^2\,dt,}
as desired.
\end{proof}

\begin{acknowledgement}
The first author was partially supported by Grant-in-Aid for Early-Career Scientists 18K13444. The second author was supported in part by Grant-in-Aid for Young Scientists (B) 24740086 and 16K17626. 
\end{acknowledgement}

\providecommand{\bysame}{\leavevmode\hbox to3em{\hrulefill}\thinspace}
\providecommand{\MR}{\relax\ifhmode\unskip\space\fi MR }
\providecommand{\MRhref}[2]{%
  \href{http://www.ams.org/mathscinet-getitem?mr=#1}{#2}
}
\providecommand{\href}[2]{#2}

\end{document}